\documentclass[11pt, reqno]{amsart}
\usepackage{amsmath,amssymb,amsthm,amsfonts,verbatim}
\usepackage{microtype}
\usepackage[all,2cell]{xy}
\usepackage{mathtools}
\usepackage{graphicx}
\usepackage{pinlabel}
\usepackage{hyperref}
\usepackage{mathrsfs}
\usepackage{color}
\usepackage[dvipsnames]{xcolor}
\usepackage{enumitem}
\usepackage{cite}
\usepackage{soul}
\usepackage{tikz}

\allowdisplaybreaks

\CompileMatrices

\usepackage[top=1.2in,bottom=1.2in,left=1in,right=1in]{geometry}

\usepackage{hyperref} %Hyperlink reference
\hypersetup{          %Set up  for hyperlinks
	colorlinks=true, breaklinks, linkcolor=[RGB]{51 102 204}, filecolor=Orchid, urlcolor=[RGB]{51 102 204},
	citecolor=Orchid, linktoc=all, }
\usepackage[nameinlink]{cleveref}
\theoremstyle{plain}
\newtheorem{theorem}{Theorem}[section]
\newtheorem{maintheorem}{Theorem}

\newtheorem{proposition}[theorem]{Proposition}
\newtheorem{lemma}[theorem]{Lemma}

\newtheorem{corollary}[theorem]{Corollary}

\theoremstyle{definition}
\newtheorem{definition}[theorem]{Definition}
\newtheorem{example}[theorem]{Example}
\newtheorem{remark}[theorem]{Remark}

\newcommand{\ZZ}{\mathbb{Z}}
\newcommand{\QQ}{\mathbb{Q}}
\newcommand{\CC}{\mathbb{C}}

\DeclareMathOperator{\Sp}{Sp}
\DeclareMathOperator{\GL}{GL}
\DeclareMathOperator{\SL}{SL}

\DeclareMathOperator{\Mod}{Mod}

\DeclareMathOperator{\Aut}{Aut}

\DeclareMathOperator{\Teich}{Teich}

\counterwithin{figure}{section}
\counterwithin{equation}{section}

\title{Rigidity of the period map up to finite covers}
\author{Xiyan Zhong}

\begin{document}
\maketitle
\vspace{-2em}
\begin{abstract}
We first give a complete classification of bi-affine representations of mapping class groups of surfaces with finitely many boundary components or punctures. We also show that every linear representation of the mapping class group of a genus-$g$ surface with two boundary components of dimension at most $3g-3$ is bi-affine.

We then classify low-dimensional symplectic representations of the mapping class group associated to triple covers. Let $[\beta]\in H_1(S_g;\mathbb Z/3\mathbb Z)^*$, and let $\widetilde{S}\to S_g$ be the corresponding triple cover with deck transformation $\sigma$. For $h\le g$, every non-abelian homomorphism from either $\Mod(S_g,[\beta])$, the stabilizer of $[\beta]$ in $\Mod(S_g)$, or $\Mod(\widetilde{S},\sigma)$, the centralizer of $\sigma$ in $\Mod(\widetilde{S})$, to $\Sp_{2h}(\mathbb Z)$ is, up to conjugation, the standard symplectic representation on $H_1(S_g;\mathbb Z)$.

As an application, we obtain a rigidity theorem for holomorphic maps from the moduli space $R_g^{(3)}$ of genus-$g$ curves equipped with a $3$-sheeted (unbranched) normal covering to the moduli space $\mathcal{A}_h$ of $h$-dimensional principally polarized abelian varieties. We prove that, for $g\ge 6$ and $h\le g$, the unique nonconstant holomorphic map from $R_g^{(3)}$, equipped with either of its two natural complex-orbifold structures, to $\mathcal{A}_h$ is the period map sending a cover $Y\to X$ to the Jacobian of the base curve $X$.
\end{abstract}
\vspace{0em}
\tableofcontents

\section{Introduction}
Let $\mathcal{M}_{g,r}$ denote the moduli space of smooth closed curves of genus $g$ with $r$ distinct marked points, and let $\mathcal{A}_h$ denote the moduli space of $h$-dimensional principally polarized abelian varieties. The period map
\[\begin{aligned}
J\colon \mathcal{M}_{g,r}&\longrightarrow\mathcal{A}_g,\\
(X,x_1,\cdots,x_r)&\longmapsto\operatorname{Jac}(X),
\end{aligned}
\]
is rigid in a strong global sense: Farb \cite[Theorem 1.1]{FarbRigidity} proved that the period map $J$ is the unique nonconstant holomorphic map $\mathcal{M}_{g,r}\to\mathcal{A}_h$ for $g\ge 3$ and $h\le g$.

There are similar rigidity results for holomorphic maps between moduli spaces. Serv\'an \cite[Theorem 1.1]{CarlosPrym} proved that the Prym map for unbranched double covers is the unique nonconstant holomorphic map from the corresponding moduli space to $\mathcal{A}_h$ for $g\ge4$ and $h\le g-1$. Antonakoudis--Aramayona--Souto \cite[Theorem 1.1]{AAS}, with an improved range obtained by De-Pool--Souto \cite[Theorem 1.1]{Souto2}, proved that every nonconstant holomorphic map 
\[\mathcal{M}_{g,r}\longrightarrow \mathcal{M}_{g',r'}\]
is induced by forgetting marked points, for $g\ge 4$ and $g'\le 3\cdot 2^{g-3}$.

In this paper, we show that the rigidity of the period map persists after passing to a certain finite cover of $\mathcal{M}_g$. We consider the moduli space of genus-$g$ curves equipped with a $3$-sheeted (unbranched) normal covering, defined by
\[R_g^{(3)}=\left\{(X,\theta_X)
\mathrel{}\middle|\mathrel{}
\begin{array}{c}X\text{ is a smooth curve of genus }g,\\
\theta_X\in H^1(X;\mathbb Z/3\mathbb Z)^*\end{array}\right\}\bigg/\sim,\]
where $(X_1,\theta_{X_1})\sim(X_2,\theta_{X_2})$ if and only if there exists a biholomorphism
$f\colon X_1\to X_2$ such that
\[f^*(\theta_{X_2})=\theta_{X_1}.\]
The period map naturally lifts to
\[\begin{aligned}
J^{(3)}\colon R_g^{(3)}&\longrightarrow\mathcal{A}_g,\\
(X,\theta_X)&\longmapsto\operatorname{Jac}(X).
\end{aligned}\]
We show that this map has the following global rigidity.
\begin{maintheorem}[\protect{Theorem~\ref{thm: unique holomorphic map}}]\label{thm: A}
Let $g\ge6$ and $h\le g$. If
\[F\colon R_g^{(3)}\longrightarrow\mathcal{A}_h\]
is a nonconstant holomorphic map of complex orbifolds, then $h=g$ and
$F=J^{(3)}$.
\end{maintheorem}

\begin{remark}
There are two natural complex-orbifold structures on $R_g^{(3)}$, introduced in \cite[Section~2]{CarlosPrym}; see Definition~\ref{def: two orbifold structures}. Theorem~\ref{thm: A} holds for both structures. The Prym map is defined for one of these structures and gives a holomorphic map to $\mathcal{A}_{2g-2}$. We conjecture that, for $g<h< 2g-2$, every nonconstant holomorphic map from $R_g^{(3)}$ to $\mathcal{A}_h$ is obtained by taking the product of the period map $J^{(3)}$ with a constant principally polarized abelian variety of dimension $h-g$.
\end{remark}

The proof of Theorem~\ref{thm: A} follows the general strategy developed by Farb \cite{FarbRigidity} using Teichm\"uller theory; see Section~\ref{sec: rigidity}. We also present an alternative approach to one of the steps in Section~\ref{sec: PVHS}, using variations of Hodge structures, following a suggestion of Hain to Farb.

The crucial first step in the proof of Theorem~\ref{thm: A} is to classify the homomorphisms induced on orbifold fundamental groups. Let $S=S_g$ be a closed surface of genus $g$. Let $[\beta]\in H_1(S_g;\mathbb Z/3\mathbb Z)^*$, and let $\widetilde{S}\to S$ be the corresponding unbranched triple cover with deck transformation $\sigma$. For the two complex-orbifold structures on $R_g^{(3)}$ introduced in Definition~\ref{def: two orbifold structures}, the orbifold fundamental group is respectively
\[\Mod(S,[\beta])=\operatorname{Stab}_{\Mod(S)}([\beta]),\]
the stabilizer of $[\beta]$ in the mapping class group $\Mod(S)$, and
\[\Mod(\widetilde{S},\sigma)=C_{\Mod(\widetilde{S})}(\sigma),\]
the centralizer of $\sigma$ in $\Mod(\widetilde{S})$. The map induced by $J^{(3)}$ on orbifold fundamental groups is the standard symplectic action on $H_1(S_g;\mathbb Z)$ in both cases.

We classify the corresponding homomorphisms on orbifold fundamental groups as follows.
\begin{maintheorem}[\protect{Theorem~\ref{thm: classify Mod to Sp}}]\label{thm: B}
Let $g\ge6$, and $h\le g$. Suppose that
\[\rho\colon\Mod(S,[\beta])\longrightarrow\Sp_{2h}(\mathbb Z)\]
or
\[\rho\colon\Mod(\widetilde{S},\sigma)\longrightarrow\Sp_{2h}(\mathbb Z)\]
is a homomorphism. Then either $\rho$ has finite abelian image, or $h=g$ and $\rho$ is conjugate to the standard symplectic action on $H_1(S_g;\mathbb Z)$ by an element of the extended symplectic group
\[\Delta_{2g}(\mathbb Z)=\{X\in\GL_{2g}(\mathbb Z):X^tJX=\pm J\},\]
where $J$ denotes the matrix of the standard symplectic form.
\end{maintheorem}

The proof of Theorem~\ref{thm: B} is based on the finite generating sets for $\Mod(S,[\beta])$ and $\Mod(\widetilde{S},\sigma)$ given by Dey--Dhanwani--Patil--Rajeevsarathy \cite[Theorem~2]{Dey}, together with a detailed analysis of linear representations of the mapping class group of the genus-$(g-1)$ surface with two boundary components obtained by cutting $S_g$ along a simple closed curve representing $[\beta]$.

A second part of the paper is devoted to studying linear representations of mapping class groups. Let $S_{g,r}^n$ denote a genus-$g$ surface with $n$ boundary components and $r$ punctures, and let $\Mod(S_{g,r}^n)$ denote its mapping class group. We omit $n$ or $r$ from the notation when it is zero. We obtain a complete classification of homomorphisms
\[\Mod(S_g^2)\to\GL_m(\mathbb C)\qquad\text{and}\qquad\Mod(S_{g,2})\to\GL_m(\mathbb C),\qquad \text{for } m\le 3g-3.\]
We also give a general framework for bi-affine representations and classify bi-affine representations of $\Mod(S_g^n)$ and $\Mod(S_{g,n})$ for arbitrary $n$, using cohomological calculations of Kawazumi--Morita \cite{KawazumiMoritaStable}.

We now explain these results. For any group $G$ and a self-dual $G$-module $H$, following \cite{reps3g-3}, we call a representation of $G$ \emph{bi-affine with core $H$} if its underlying module $V$ admits a filtration
\[0\subset V_1\subset V_2\subset V\]
such that $V_1$ and $V/V_2$ are trivial $G$-modules and $V_2/V_1\cong H$. If $V_1=0$, we call the representation \emph{affine}; if $V=V_2$, we call it \emph{co-affine}. In Proposition~\ref{prop: classify bi-affine}, we classify bi-affine representations in terms of two extension classes and the vanishing of their contracted cup product. In the case
\[H=H_1(S_g;\mathbb C),\qquad G=\Mod(S_{g,r}^n),\]
we compute the contracted cup product explicitly. This leads to the following classification of bi-affine representations of mapping class groups of surfaces with boundary components or punctures.
\begin{maintheorem}[\protect{Theorem~\ref{thm: classify bi-affine for n bdry}, Theorem~\ref{thm: classify bi-affine for n punctures}}]\label{thm: C}
Let $g\ge4$ and $m,n\ge 0$.
\begin{enumerate}
    \item The isomorphism classes of affine or co-affine representations of $\Mod(S_g^n)$ or $\Mod(S_{g,n})$ with core $H=H_1(S_g;\CC)$ of dimension $2g+m$ are classified by
    \[H^1(\Mod(S_g^n);H)\otimes\mathbb C^m \cong H^1(\Mod(S_{g,n});H)\otimes\mathbb C^m \cong \mathbb C^n\otimes\mathbb C^m, \]
    modulo the natural action of $\GL_m(\mathbb C)$ on $\mathbb C^m$.

    \item The isomorphism classes of bi-affine representations of $\Mod(S_g^n)$ with core $H=H_1(S_g;\CC)$ of dimension $2g+m$ that are neither affine nor co-affine are classified by pairs of nonzero elements
    \[x=\sum_{j=1}^d
    \big(\sum_{i=1}^n s_i^{(j)}[\kappa_i]\big)\otimes v_j
    \in H^1(\Mod(S_g^n);H)\otimes\mathbb C^d \]
    and
    \[y=\sum_{j=1}^{m-d}
    \big(\sum_{i=1}^n t_i^{(j)}[\kappa_i]\big)\otimes w_j
    \in H^1(\Mod(S_g^n);H)\otimes\mathbb C^{m-d}, \]
    where $0<d<m$, $\{v_j\}$ and $\{w_j\}$ are the standard bases of $\mathbb C^d$ and $\mathbb C^{m-d}$, respectively, and $[\kappa_i]$, $1\le i\le n$, form a basis of
    \[H^1(\Mod(S_g^n);H)\cong\mathbb C^n.\]
    These pairs satisfy
    \[\big(\sum_{i=1}^n s_i^{(k)}\big)
    \big(\sum_{i=1}^n t_i^{(\ell)}\big)=0,
    \qquad
    1\le k\le d,\quad 1\le\ell\le m-d,\]
    and are considered modulo the natural action of
    \[ \GL_d(\mathbb C)\times\GL_{m-d}(\mathbb C).\]

    For $\Mod(S_{g,n})$, the same description holds, with the additional conditions
    \[ s_i^{(k)}T^{(\ell)}
    +t_i^{(\ell)}S^{(k)}
    +2(g-1)s_i^{(k)}t_i^{(\ell)}=0,\]
    for all $i,k,\ell$, where
    \[ S^{(k)}=\sum_{i=1}^n s_i^{(k)},
    \qquad
    T^{(\ell)}=\sum_{i=1}^n t_i^{(\ell)}. \]
\end{enumerate}
\end{maintheorem}

For given cohomology classes, an explicit matrix representative of the corresponding isomorphism class of bi-affine representations is given in Remark~\ref{rmk: matix form}.

\begin{remark}
There is also a Hodge-theoretic interpretation of Theorem~\ref{thm: C}. Affine representations are closely related to normal functions, while bi-affine representations arise as the underlying local systems of biextensions of variation of mixed Hodge structures. We explain this connection in Section~\ref{sec: 2.6}.
\end{remark}

We next show that, when there are two boundary components or two punctures, bi-affine representations account for all representations in low dimensions.

\begin{maintheorem}[\protect{Theorem~\ref{thm: any low-dim rep is bi-affine}, Corollary~\ref{cor: 2 puncture only affine}}]\label{thm: D}
Let $g\ge3$ and $m\le3g-3$. Then every representation
\[\rho:\Mod(S_g^2)\longrightarrow\GL_m(\mathbb C)\]
is either trivial, or bi-affine with core $H=H_1(S_g;\mathbb C)$. Moreover, every nontrivial representation
\[\rho:\Mod(S_{g,2})\longrightarrow\GL_m(\mathbb C)\]
is affine or co-affine with core $H=H_1(S_g;\mathbb C)$.
\end{maintheorem}

Thus, in these cases, all such low-dimensional representations are completely determined by the classification in Theorem~\ref{thm: C}. Many of these representations also admit explicit geometric interpretations; see Remark~\ref{rmk: geometric interpretation} and Proposition~\ref{prop: singular pt}. The proof of Theorem~\ref{thm: D} uses the theory of partitioned surfaces developed by Putman \cite{cuttingandpasting} and further extended by Church \cite{ChurchPartitionedJohnsonHom}.

In the case of $\Mod(S_{g,2})$, the additional conditions in Theorem~\ref{thm: C} rule out representations that are neither affine nor co-affine. We also give an alternative proof using relations in the surface braid group. As a consequence, the only nonconstant holomorphic maps
\[\mathcal{M}_{g,2}\longrightarrow\mathcal{A}_h,
\qquad 2h\le3g-3,\]
are obtained by taking the product of the period map with a constant $(h-g)$-dimensional principally polarized abelian variety. This result will appear in a subsequent joint paper with Zhong Zhang, using methods from the theory of variations of Hodge structures.

\begin{remark}
Theorem~\ref{thm: D} extends the previous classification of low-dimensional linear representations of mapping class groups. It was known that any nontrivial representation
\[\rho:\Mod(S_{g,r}^n)\longrightarrow\GL_m(\mathbb C)\]
in the following ranges is affine or co-affine with core $H=H_1(S_g;\mathbb C)$:
\begin{itemize}
    \item for $m<2g$, by Korkmaz \cite[Theorem~1]{korkmaz} and Franks--Handel \cite[Theorem~1.1]{FranksHandel};
    \item for $m=2g$, by Korkmaz \cite[Theorem~2]{korkmaz};
    \item for $m=2g+1$, by Kasahara \cite[Theorem~1.1]{kasahara};
    \item for $m\le 3g-3$ and $r+n\le1$, by joint work with Kaufmann, Salter, and Zhang \cite[Theorem~A]{reps3g-3}.
\end{itemize}
\end{remark}

\noindent\textbf{Outline.}
In Section~\ref{sec: 2.1}, we introduce bi-affine representations and give a cohomological criterion for classifying their isomorphism classes. In Section~\ref{sec: 2.2}, we classify bi-affine representations of $\Mod(S_g^n)$. In Section~\ref{sec: 2.3}, we study bi-affine representations of $\Mod(S_g^2)$ and their geometric interpretations. In Section~\ref{sec: 2.4}, we prove that every linear representation of $\Mod(S_g^2)$ of dimension at most $3g-3$ is bi-affine. In Section~\ref{sec: 2.5}, we classify bi-affine representations of $\Mod(S_{g,n})$ and prove that every linear representation of $\Mod(S_{g,2})$ of dimension at most $3g-3$ is affine or co-affine. In Section~\ref{sec: 2.6}, we give a Hodge-theoretic interpretation of our results using normal functions and biextensions.

In Section~\ref{sec: 3}, we classify symplectic representations of $\Mod(S,[\beta])$ and $\Mod(\widetilde{S},\sigma)$ of dimension at most $2g$, proving Theorem~\ref{thm: B}. In Section~\ref{sec: 4}, we prove Theorem~\ref{thm: A}, which establishes the uniqueness of the period map among nonconstant holomorphic maps $R_g^{(3)}\to\mathcal{A}_h$ for $h\le g$.

\noindent\textbf{Acknowledgements.} I am very grateful to Andrew Putman for suggesting this problem and for many useful comments. I thank Carlos Serv\'an for answering questions related to his paper \cite{CarlosPrym}. I also thank Ursula Hamenst\"adt, Jaya Iyer, and Nariya Kawazumi for helpful conversations. I thank Benson Farb, Nick Salter, and Hank Morris for useful comments on an earlier draft of this paper, and Hank Morris in particular for pointing out the connection between bi-affine representations and biextensions of line bundles. I am grateful to the Max Planck Institute for Mathematics in Bonn for its hospitality.

\section{Linear representations of the mapping class group}
In this section, we introduce bi-affine representations of groups and establish a general classification theorem. We then apply this result to obtain a complete classification of bi-affine representations of $\Mod(S_g^n)$ and $\Mod(S_{g,n})$. Furthermore, we show that when $n=2$, every linear representation of $\Mod(S_g^n)$ and $\Mod(S_{g,n})$ of dimension at most $3g-3$ is bi-affine, and hence completely determined by the classifications obtained in this section.
\subsection{Bi-affine representations}\label{sec: 2.1}
In this subsection, we introduce bi-affine representations and give a cohomological criterion for classifying their isomorphism classes. We begin by recalling the following definition from \cite[Definition~3.1]{reps3g-3}.
\begin{definition}\label{def: biaffine rep}
Let $G$ be a group and $V$ be a $\CC[G]$-module. The associated representation $\rho: G\to \GL(V)$ is called \emph{bi-affine with core} a fixed $\CC[G]$-module $H$ if there are submodules
\[0\subset V_1 \subset V_2\subset V\]
such that $G$ acts on $V_1$ and $V/V_2$ trivially, and $V_2/V_1\cong H$ as $\CC[G]$-modules. If $V_1=0$, $\rho$ is called \emph{affine}; if $V_2=V$, $\rho$ is called \emph{co-affine}.
\end{definition}
\begin{remark}
   Let $V^*=\operatorname{Hom}(V,\CC)$ be the dual $G$-module, then the dual representation \[\rho^*:G\to \GL(V^*)\] is bi-affine with core $H^*$. The filtration
   \[0\subset (V/V_2)^*\subset (V/V_1)^* \subset V^*\]
   consists of $G$-invariant submodules, and \[(V/V_1)^*/(V/V_2)^*\cong (V_2/V_1)^*\cong H^*,\]
   and $(V/V_2)^*$ and $V^*/(V/V_1)^*\cong V_1^*$ are trivial $G$-modules.
\end{remark}

The class of bi-affine representations is closed under extensions by trivial representations, provided that $H^1(G;\CC)=0$.
\begin{proposition}\label{prop: extend bi-affine}
Assume that $H^1(G;\CC)=0$. Let $W$ be a $\CC[G]$-module with a submodule $V$. If $V$ (resp.~$W/V$) is a bi-affine representation with core $H$, and $W/V$ (resp.~$V$) is a trivial $G$-module, then $W$ is also a bi-affine representation with core $H$.
\end{proposition}
\begin{proof}
We prove the first case; the second is dual.
Assume that $V$ is bi-affine with core $H$ and that $W/V$ is a trivial $G$-module. Let
\[0\subset V_1\subset V_2\subset V\]
be the filtration defining the bi-affine structure on $V$. Then the short exact sequence of $G$-modules
\[0\to V/V_2\to W/V_2 \to W/V \to 0\]
splits since both $V/V_2$ and $W/V$ are trivial $G$-modules, and
\[\text{Ext}_{G}^1(W/V,V/V_2)\cong H^1(G;\operatorname{Hom}(W/V,V/V_2))\cong H^1(G;\CC)\otimes \operatorname{Hom}(W/V,V/V_2)=0.\]
Hence $W/V_2$ is a trivial $G$-module, and the filtration
\[0\subset V_1 \subset V_2\subset W\]
makes $W$ a bi-affine representation with core $H$.
\end{proof}

In our setting, we are primarily interested in the case
\[G=\Mod(S), \quad H=H_1(S_g;\CC).\]
However, in this subsection we formulate more general results on bi-affine representations, using only the properties of $G$ and $H$ that are needed.

From now on, assume that $G$ satisfies 
\[H^1(G;\CC)=0,\]
and that $H$ is an irreducible $\CC[G]$-module with
\[\text{End}_G(H)=\CC.\]
Suppose further that $H$ admits a non-degenerate alternating $G$-invariant bilinear form
\[\omega: H\otimes H \to \CC,\]
which induces a $G$-equivariant isomorphism $H^*\cong H$. Put
\[\mathcal{E}_G=H^1(G;H).\]
Cup product followed by contraction with $\omega$ defines a symmetric bilinear form
\begin{equation}\label{eq: contraction form}
    \mathcal{B}_G:\mathcal{E}_G\otimes \mathcal{E}_G \to H^2(G;\CC),\quad \mathcal{B}_G(x,y) \coloneq \omega_*(x\cup y).
\end{equation} 
Bi-affine representations of $G$ with core $H$ are then classified as follows.
\begin{proposition}\label{prop: classify bi-affine}
The isomorphism classes of bi-affine representations of $G$ with core $H$ are in one-to-one correspondence with the orbits of pairs
\[x\in \mathcal{E}_G \otimes V_1,\quad y\in \mathcal{E}_G \otimes (V/V_2)^*\]
satisfying 
\[\widehat{\mathcal{B}}_G(x,y)=0\in H^2(G;\CC)\otimes \operatorname{Hom}(V/V_2,V_1),\]
where $\widehat{\mathcal{B}}_G$ is the map induced by $\mathcal{B}_G$ and the natural pairing
\[
V_1\otimes (V/V_2)^*\longrightarrow \operatorname{Hom}(V/V_2,V_1).
\]
The orbits are taken with respect to the action of $\GL(V_1)\times \GL(V/V_2)$ given by
\[(P,R)\cdot(x,y)=(xP,yR^{-1}).\]
\end{proposition}
\begin{proof}
    Suppose $V$ is a bi-affine representation of $G$ with core $H$ as in Definition \ref{def: biaffine rep}. 
    
    The extension
    \begin{equation}\label{ses: extension 1}
        0\to V_1\to V_2\to H\to 0
    \end{equation}
    determines an element
    \[x\in \text{Ext}_G^1(H,V_1)\cong H^1(G;\operatorname{Hom}(H,V_1))\cong H^1(G;H^*)\otimes V_1 \cong \mathcal{E}_G \otimes V_1,\]
    where the last isomorphism uses the identification $H^*\cong H$ induced by the form $\omega$.
    The extension
    \[0\to V_2\to V\to V/V_2\to 0,\]
    then determines an element in $\text{Ext}^1_G(V/V_2,V_2)$.
    
    Applying the functor $\text{Ext}^*_G(V/V_2,-)$ to \eqref{ses: extension 1} gives the long exact sequence
    \begin{equation}\label{les: ext}
        \cdots\to \text{Ext}^1_G(V/V_2,V_1)\to \text{Ext}^1_G(V/V_2,V_2)\to \text{Ext}^1_G(V/V_2,H)\to \text{Ext}^2_G(V/V_2,V_1)\to \cdots.
    \end{equation}
    Here
    \[\text{Ext}^1_G(V/V_2,V_1)\cong H^1(G;\operatorname{Hom}(V/V_2,V_1))\cong H^1(G;\CC)\otimes \operatorname{Hom}(V/V_2,V_1)=0\]
    where the last equality follows from our assumption. Therefore, the long exact sequence \eqref{les: ext} gives
    \[\text{Ext}^1_G(V/V_2,V_2)\cong \text{Ker}( \text{Ext}^1_G(V/V_2,H)\xrightarrow{\partial_x} \text{Ext}^2_G(V/V_2,V_1),\]
    where the connecting map $\partial_x$ is given by the Yoneda product with $x$:
    \begin{align*}
        \partial_x:\text{Ext}^1_G(V/V_2,H)\cong H^1(G;H)\otimes (V/V_2)^*&\longrightarrow \text{Ext}^2_G(V/V_2,V_1)\cong H^2(G;\CC)\otimes \operatorname{Hom}(V/V_2,V_1) \\
        y&\longmapsto \widehat{B}_G(x,y).
    \end{align*}
    Hence the isomorphism class of the bi-affine representation $V$ is determined by
    \[x\in \mathcal{E}_G \otimes V_1,\quad y\in \mathcal{E}_G \otimes (V/V_2)^*,\] 
    up to changes of bases in the trivial modules $V_1$ and $V/V_2$.
\end{proof}
\begin{remark}\label{rmk: matix form}
   Under the vector-space splitting $V=V_1\oplus H\oplus (V/V_2)$, a bi-affine representation $\rho:G\to \GL(V)$ has the matrix form
   \begin{equation}\label{eq: matrix form}
       \rho(f)=\begin{pmatrix}
   I_{V_1} & \delta(f) & \chi(f) \\
   0 &\Psi(f) & \sigma(f) \\
   0 & 0 & I_{V/V_2}   \end{pmatrix}
   \end{equation}
   where $\Psi:G\to \GL(H)$ is the given action on $H$. The homomorphism equation $\rho(fh)=\rho(f)\rho(g)$ is equivalent to the identities
   \begin{align}
   \delta(fh)&=\delta(f)\Psi(h)+\delta(h),\label{eq: delta cocycle}\\
   \sigma(fh)&=\sigma(f)+\Psi(f)\sigma(h),\label{eq: sigma cocycle}\\
   \chi(fh)&=\chi(f)+\chi(h)+\delta(f)\sigma(h).\label{eq: chi equation}
   \end{align}
   The identities \eqref{eq: delta cocycle} and \eqref{eq: sigma cocycle} imply that
   \[[\delta]\in \mathcal{E}_G \otimes V_1,\quad [\sigma]\in \mathcal{E}_G \otimes (V/V_2)^*.\]
   The condition $\widehat{\mathcal{B}}_G([\delta],[\sigma])=0$ is precisely the condition for the existence of a map $\chi$ satisfying
   \eqref{eq: chi equation}. Since
   \[H^1(G;\operatorname{Hom}(V/V_2,V_1))\cong H^1(G;\CC)\otimes \operatorname{Hom}(V/V_2,V_1)=0,\]
   any two such choices of \(\chi\) differ by a coboundary, and hence correspond to changes of bases.
\end{remark}

\begin{definition}\label{def: self dual}
   Let $\rho:G\to \GL(V)$ be a bi-affine representation determined by the pair
   \[x\in \mathcal{E}_G \otimes V_1,\quad y\in \mathcal{E}_G \otimes (V/V_2)^*.\]
   Then the dual representation\[\rho^*:G\to \GL(V^*)\]
   is bi-affine with core $H^*\cong H$ and is determined by the pair
   \[y\in \text{Ext}_{G}^1(H^*,(V/V_2)^*)\cong \mathcal{E}_G \otimes (V/V_2)^*,\quad x\in \text{Ext}_{G}^1(V_1^*,H^*)\cong  \mathcal{E}_G \otimes V_1.\]
   If $V_1\cong (V/V_2)^*$, and under this identification $x=y$ up to the action of $\GL(V_1)$,
   then we call $\rho$ \emph{self-dual}.
\end{definition}
\begin{remark}\label{rmk: self dual}
   Indeed, by Proposition \ref{prop: classify bi-affine}, a bi-affine representation $\rho:G\to \GL(V)$ is self-dual if and only if $V\cong V^*$ as $G$-modules.
\end{remark}

\subsection{Classification of bi-affine representations of $\Mod(S_g^n)$}\label{sec: 2.2}
In the previous subsection, we established a general cohomological criterion for classifying the isomorphism classes of bi-affine representations. We now study to the case
\[G=\Mod(S_g^n), \quad H=H_1(S_g;\CC), \]
and obtain a more explicit description of the corresponding isomorphism classes.

We begin with the following result. It was proved by Morita \cite{MoritaI} for $n+r\le 1$ (over $\ZZ$) and generalized by Hain \cite[Proposition~5.2]{HainTorelliGrp}.
\begin{proposition} For $g\ge 3$, we have
    \[H^1(\Mod(S_{g,r}^n);H)\cong \CC^{n+r}.\]
\end{proposition}
In particular,
\[\mathcal{E}_{\Mod(S_g^n)}=H^1(\Mod(S_g^n);H)\cong \CC^n.\]
We give an explicit basis for this cohomology group. We do so in terms of the generator of
\[H^1(\Mod(S_{g,1});H_1(S_g;\ZZ))\cong \ZZ,\]
which is a crossed homomorphism\footnote{A map $\sigma$ from a group $G$ to a $G$-module $M$ is called a \emph{crossed homomorphism} if $\sigma$ satisfies $\sigma(fh)=\sigma(f)+f\cdot\sigma(h)$ for any $f,h\in G$. A crossed homomorphism is called \emph{principal} if there exists $u\in M$ such that $\sigma(f)=f\cdot u-u$ for all $f\in G$. The set of crossed homomorphisms module principal ones is isomorphic to $H^1(G;M)$.} constructed explicitly by 
Morita \cite[Section 6]{MoritaI}.
\begin{definition}\label{def: kappa}
The crossed homomorphism
\[\kappa:\Mod(S_{g,1})\longrightarrow H_1(S_g;\ZZ)\]
is defined as follows. Under the identification of $H_1(S_g;\ZZ)$ with its dual induced by the algebraic intersection pairing $\omega$, we view $\kappa$ as the map
\begin{align*}
    \Mod(S_{g,1})\times H_1(S_g;\ZZ) &\longrightarrow \ZZ, \\
    (f,x)&\longmapsto d(f\cdot \gamma)-d(\gamma)
\end{align*}
where $\gamma\in\pi_1(S_{g,1})$ is any element satisfying
$x=[\gamma]\in H_1(S_g;\ZZ)$, and $d$ is the map
\[d=\sum_{k=1}^g d_k:\pi_1(S_{g,1})\to \ZZ\]
where $d_k$ is given by
\begin{align*}
    d_k: \pi_1(S_g^1) \xrightarrow{\mathrm{proj}} \mathbb{F}_2 \langle \alpha_k,\beta_k \rangle &\longrightarrow \ZZ \\
    \alpha_k^{\epsilon_1}\beta_k^{\rho_1}\cdots \alpha_k^{\epsilon_n}\beta_k^{\rho_n} &\longmapsto \sum\limits_{i=1}^n \epsilon_i \sum\limits_{j=i}^n \rho_j-\sum\limits_{i=1}^n \rho_i \sum\limits_{j=i+1}^n \epsilon_j.
\end{align*}
Here $\{\alpha_1,\beta_1,\cdots, \alpha_g,\beta_g\}$ is a chosen basis of $\pi_1(S_{g,1})$ satisfying
\[\omega(\alpha_j,\beta_k)=\delta_j^k,\quad \omega(\alpha_j,\alpha_k)=\omega(\beta_j,\beta_k)=0.\]
\end{definition}
\begin{remark}
   Besides the above combinatorial definition, $\kappa$ also admits several geometric interpretations. Earle \cite{Earle} interpreted it via the moduli space of Riemann surfaces. Trapp \cite{Trapp} and Furuta \cite{FurutabyMorita} (described by Morita) constructed it using the unit tangent bundle and winding numbers. Chen \cite{chen} gave another construction in terms of circle actions and rotation numbers.
\end{remark}
We recall the following useful property of the crossed homomorphism $\kappa$.
\begin{proposition}[\protect{\cite[Proposition 3.1]{MoritaII}}]\label{prop: image of point pushing}
    Let $\kappa:\Mod(S_{g,1})\to H_1(S_g;\ZZ)$ be the map defined in Definition \ref{def: kappa}. Let $\pi_1(S_g)$ be the point-pushing subgroup of $\Mod(S_{g,1})$, which is the kernel of the forgetful map $\Mod(S_{g,1})\to \Mod(S_g)$. Then for any $\gamma\in \pi_1(S_g)$, we have
    \[\kappa(\gamma)=(2-2g)[\gamma].\]
\end{proposition}

For $i\in\{1,2,\cdots,n\}$, by gluing a punctured disk to the $i$-th boundary component of $S_g^n$ and disks to the remaining boundary components, and extending homeomorphisms by the identity, we obtain induced surjective maps
\[\operatorname{Cap}_i: \Mod(S_{g}^n)\longrightarrow \Mod(S_{g,1}).\]
\begin{lemma}\label{lem: two crossed homomorphisms}
    Let $\kappa:\Mod(S_{g,1})\to H_1(S_g;\ZZ)$ be the crossed homomorphism defined in Definition \ref{def: kappa}. Define
    \[\kappa_i: \Mod(S_g^n)\xrightarrow{\operatorname{Cap}_i}  \Mod(S_{g,1})\xrightarrow{\kappa} H_1(S_g;\ZZ), \quad i\in\{1,2,\cdots,n\}.\] Then the cohomology classes $[\kappa_i]$ ($1\le i \le n$) generate $ \mathcal{E}_{\Mod(S_g^n)}$.
\end{lemma}
\begin{proof}
    It suffices to show that $[\kappa_i]$ ($1\le i \le n$) are linearly independent. Suppose that \[s_1[\kappa_1]+\cdots+s_n[\kappa_n]=0\in \mathcal{E}_{\Mod(S_g^n)}, \quad s_i\in \CC.\] Then $\sum\limits_{i=1}^n s_i\kappa_i$ is a principal crossed homomorphism, so there exists $u\in H$ such that
    \[\sum\limits_{i=1}^n s_i\kappa_i(f)=f\cdot u-u, \quad \forall f\in \Mod(S_g^n).\]
    Take an element $\gamma$ in the point-pushing subgroup $\pi_1(S_g)$ of $\Mod(S_{g,1})$ with
    $ 0\neq[\gamma]\in H_1(S_g;\ZZ)$. Let $\gamma_{\ell}\in\Mod(S_g^n)$ be a lift of $\gamma$ under the map $\operatorname{Cap}_{\ell}$ for $1\le \ell \le n$. By Proposition~\ref{prop: image of point pushing},
    \[\kappa_{\ell}(\gamma_{\ell})=\kappa(\gamma)=(2-2g)[\gamma].\]
    Moreover, $\kappa_i(\gamma_{\ell})=0$ for $i\neq {\ell}$, since $\gamma_{\ell}$ lies in the kernel of $\operatorname{Cap}_i$.
    Therefore we have 
    \[\sum\limits_{i=1}^n s_i\kappa_i(\gamma_{\ell})=s_{\ell}(2-2g)[\gamma]=\gamma_{\ell}\cdot u-u=0\]
    since the point-pushing subgroup acts trivially on $H$. Since $[\gamma]\neq 0$ and $g\ge 2$, it follows that $s_{\ell}=0$ for each $1\le \ell\le n$.
\end{proof}
\begin{remark}\label{rmk: dual for crossed}
Dually, the cohomology group
\[H^1(\Mod(S_g^n);\text{Hom}(H_1(S_g;\CC),\CC))\cong \CC^n\]
is generated by $[\kappa_i^*]$ ($1\le i \le n$) where $\kappa_i^*:\Mod(S_g^n)\to \text{Hom}(H_1(S_g;\CC),\CC)$ is the crossed homomorphism defined by
\[\kappa_i^*(f)(x)=\omega(x,\kappa_i(f^{-1})), \quad f\in \Mod(S_g^n).\]\end{remark}

In Proposition~\ref{prop: classify bi-affine}, we gave a general cohomological criterion for classifying bi-affine representations. To apply this criterion, one needs to compute the contraction form $\mathcal{B}_G$ defined in \eqref{eq: contraction form}. We now compute this form explicitly for $G=\Mod(S_g^n)$ using the basis of
$\mathcal{E}_{\Mod(S_g^n)}$ constructed above.
\begin{lemma}\label{lem: contraction form for Mod}
    Let $g\ge 4$ and $n\ge 1$. Let 
    \[x=\sum\limits_{i=1}^n s_i[\kappa_i],\quad y=\sum\limits_{i=1}^n t_i[\kappa_i] \in \mathcal{E}_{\Mod(S_g^n)}.\]
    Then 
    \[\mathcal{B}_{\Mod(S_g^n)}(x,y)=-\big(\sum_{i=1}^n s_i\big)
    \big(\sum_{j=1}^n t_j\big)e_1\in H^2(\Mod(S_g^n);\CC) \]
    where 
    \[H^2(\Mod(S_g^n);\CC)\cong \CC\]
    is generated by the first Miller-Morita-Mumford class $e_1$.
\end{lemma}
\begin{proof}
    The case $n=1$ follows from the contraction formula in \cite[Theorem 6.2]{KawazumiMoritaStable}, which gives 
   \begin{align*}
       \mathcal{B}_{\Mod(S_g^1)}([\kappa],[\kappa])=\omega_*([\kappa]\cup[\kappa])=-e_1.
   \end{align*}
   
   Now assume $n\geq 2$. Embed $S_g^1$ into $S_g^n$ by attaching a sphere with $n+1$ boundary components along the boundary of $S_g^1$. This induces a group inclusion
   \[\text{Emb}:\Mod(S_g^1)\hookrightarrow \Mod(S_g^n).\]
   By the naturality of the contraction pairing, we obtain the following commutative diagram:
   \begin{equation}\label{eq: commuting diagram}
       \xymatrix@C=7em{H^1(\Mod(S_g^n);H)\times H^1(\Mod(S_g^n);H)\ar[r]^-{\mathcal{B}_{\Mod(S_g^n)}} \ar[d]^{\text{Emb}^*\times \text{Emb}^*} &  H^2(\Mod(S_g^n);\CC) \ar[d]^{\text{Emb}^*} \\ H^1(\Mod(S_g^1);H)\times H^1(\Mod(S_g^1);H)\ar[r]^-{\mathcal{B}_{\Mod(S_g^1)}} & H^2(\Mod(S_g^1);\CC).}
   \end{equation}
   The second vertical map in \eqref{eq: commuting diagram}
   \[\text{Emb}^*: H^2(\Mod(S_g^2);\CC)\longrightarrow H^2(\Mod(S_g^1);\CC)\]
   is an isomorphism by Harer stability \cite{Harer} for $g\ge 4$, and by Madsen-Weiss Theorem \cite[Theorem~1.1.1]{Madsen-Weiss}
   \[H^2(\Mod(S_g^1);\CC)\cong \CC\]
   is generated by the first Miller-Morita-Mumford class $e_1$.
   
   By Lemma~\ref{lem: two crossed homomorphisms}, the cohomology class $[\kappa_i]$ is the image of $[\kappa]$ under the map
   \[\text{Cap}_i^*:H^1(\Mod(S_g^1);H)\longrightarrow H^1(\Mod(S_g^n);H),\]
   where $\text{Cap}_i:\Mod(S_g^n)\to \Mod(S_g^1)$ is obtained by capping off all boundary components of $S_g^n$ except the $i$th one. Since
   \[ \text{Cap}_i\circ\text{Emb}= \operatorname{Id},\]
   we obtain
   \[\text{Emb}^*(x)=\text{Emb}^*(\sum\limits_{i=1}^n s_i[\kappa_i])=\sum\limits_{i=1}^n s_i\text{Emb}^*\text{Cap}_i^*([\kappa])=\big(\sum\limits_{i=1}^n s_i\big)[\kappa],\]
   and similarly
   \[\text{Emb}^*(y)=\big(\sum_{j=1}^n t_j\big)[\kappa].\]
   Therefore, the commuting diagram \eqref{eq: commuting diagram} implies that
   \begin{align*}
   \operatorname{Emb}^*\circ \mathcal{B}_{\Mod(S_g^n)}(x,y)&=
   \mathcal{B}_{\Mod(S_g^1)}\bigl(\operatorname{Emb}^*(x),\operatorname{Emb}^*(y)\bigr)\\
   &=\mathcal{B}_{\Mod(S_g^1)}\Big(\big(\sum_{i=1}^n s_i\big)[\kappa],\big(\sum_{j=1}^n t_j\big)[\kappa]\Big)\\
   &=\big(\sum_{i=1}^n s_i\big)\big(\sum_{j=1}^n t_j\big)\mathcal{B}_{\Mod(S_g^1)}([\kappa],[\kappa])\\
   &=-\big(\sum_{i=1}^n s_i\big)
   \big(\sum_{j=1}^n t_j\big)e_1. \qedhere
   \end{align*}
\end{proof}

We now classify bi-affine representations of $\Mod(S_g^n)$
 with core $H$ using this result.
\begin{theorem}\label{thm: classify bi-affine for n bdry}
   Let $g\ge 4$. Every bi-affine representation of $\Mod(S_g^n)$ with core $H=H_1(S_g;\CC)$ of dimension $2g+m$ is of one of the following types:
   \begin{enumerate}
       \item Affine or co-affine. In either case, the isomorphism classes are classified by
       \[(\mathcal{E}_{\Mod(S_g^n)}\otimes \CC^m )/\GL_m(\CC)\cong (\CC^n\otimes \CC^m)/\GL_m(\CC)\cong
      \bigsqcup_{r=0}^{\min(n,m)}\operatorname{Gr}(r,\CC^n),\]
      where $\operatorname{Gr}(r,\CC^n)$ denotes the Grassmannian of $r$-dimensional subspaces of $\CC^n$.
       \item Neither affine nor co-affine. Such representations are classified by pairs of nonzero elements
       \[x=\sum_{j=1}^d\big(\sum\limits_{i=1}^n s_i^{(j)}[\kappa_i] \big)\otimes v_j\in \mathcal{E}_{\Mod(S_g^n)}\otimes \CC^d,\]
       and
       \[y=\sum_{j=1}^{m-d}\big(\sum\limits_{i=1}^n t_i^{(j)}[\kappa_i] \big)\otimes w_j\in \mathcal{E}_{\Mod(S_g^n)}\otimes \CC^{m-d}\]
       for $0<d<m$, where ${v_j}$ and ${w_j}$ denote the standard bases of $\CC^d$ and $\CC^{m-d}$, respectively.
       These pairs satisfy the condition
       \[\big(\sum_{i=1}^n s_i^{(k)}\big)\big(\sum_{j=1}^n t_j^{(\ell)}\big)=0,\qquad 1\le k\le d,\quad 1\le \ell\le m-d,\]
       and are considered modulo the natural action of
       \[\GL_d(\CC)\times \GL_{m-d}(\CC).\]
       Hence the parameter space is identified with the union
       \[ \left(\bigsqcup_{r=1}^{d}\operatorname{Gr}(r,\mathcal{H}^{n-1})\times \bigsqcup_{s=1}^{m-d}\operatorname{Gr}(s,\CC^n)\right)\bigcup\limits_{X}\left(\bigsqcup_{r=1}^{d}\operatorname{Gr}(r,\CC^n)\times\bigsqcup_{s=1}^{m-d}\operatorname{Gr}(s,\mathcal{H}^{n-1})\right),\]
       with the two components glued along
       \[X=\bigsqcup_{r=1}^{d}\bigsqcup_{s=1}^{m-d}\operatorname{Gr}(r,\mathcal{H}^{n-1})\times\operatorname{Gr}(s,\mathcal{H}^{n-1}),\]
       where 
       \[\mathcal{H}^{n-1}=\big\{(s_1,\dots,s_n)\in\CC^n\mid \sum_{i=1}^n s_i=0\big\}.\]
   \end{enumerate}
\end{theorem}
\begin{proof}
    This follows directly from Proposition~\ref{prop: classify bi-affine}, which gives the cohomological criterion for classifying bi-affine representations, together with the explicit computation of the contraction form $\mathcal{B}_{\Mod(S_g^n)}$ in Lemma~\ref{lem: contraction form for Mod}.
\end{proof}

The following corollary is an immediate consequence of Theorem~\ref{thm: classify bi-affine for n bdry}.
\begin{corollary}\label{cor: lowest dim to consider}
Let $g\ge 4$.
\begin{enumerate}
\item Every affine (resp.~co-affine) representation of $\Mod(S_g^n)$ of dimension $m\ge 2g+n$ is the direct sum of a $(2g+n)$-dimensional affine (resp.~co-affine) representation and a trivial representation.
\item Every bi-affine representation of $\Mod(S_g^n)$ that is neither affine nor co-affine and has dimension $m\ge 2g+2n-1$ is the direct sum of a $(2g+2n-1)$-dimensional bi-affine representation that is neither affine nor co-affine and a trivial representation.
\end{enumerate}
\end{corollary}

\subsection{Example: bi-affine representations of $\Mod(S_g^2)$}\label{sec: 2.3}
In this subsection, we specialize the general classification theorem to the case
$G=\Mod(S_g^2)$ and give a more explicit description of the resulting bi-affine representations, including their geometric interpretations.

By Corollary~\ref{cor: lowest dim to consider}, in order to classify all bi-affine representations of $\Mod(S_g^2)$, it suffices to consider affine/co-affine representations of dimension at most $2g+2$ and bi-affine representations that are neither affine nor co-affine of dimension at most $2g+3$. The classification follows from Theorem~\ref{thm: classify bi-affine for n bdry}, and is stated explicitly as follows.
\begin{corollary}\label{cor: 2g+2 bi-affine rep}
Let $g\ge 4$. Every $(2g+2)$-dimensional bi-affine representation of $\Mod(S_g^2)$ with core $H$ is of one of the following types:
\begin{enumerate}
    \item Affine or co-affine. In either case, the isomorphism classes are classified by
    \[(\mathcal{E}_{\Mod(S_g^2)}\otimes \CC^2)/\GL_2(\CC)\cong(\CC^2\otimes\CC^2)/\GL_2(\CC).\]
    Corresponding to tensors of rank $0$, $1$, and $2$, the $\GL_2(\CC)$-orbits are represented by $0$,
    \[(s_1[\kappa_1]+s_2[\kappa_2])\otimes(1,0),
    \qquad [s_1:s_2]\in\CC\mathbb{P}^1,\]
    and
    \[[\kappa_1]\otimes(1,0)+[\kappa_2]\otimes(0,1).\]
    \item Neither affine nor co-affine. Such representations are classified by pairs
    \[x=s_1[\kappa_1]+s_2[\kappa_2],\qquad
    y=t_1[\kappa_1]+t_2[\kappa_2]
    \in \mathcal{E}_{\Mod(S_g^2)}\setminus\{0\}, \]
    satisfying
    \[(t_1+t_2)(s_1+s_2)=0,\]
    modulo the natural action of
    \[\GL_1(\CC)\times\GL_1(\CC)\cong\CC^*\times\CC^*.\]
    Equivalently, the isomorphism classes are parametrized by
    \[\CC\mathbb{P}^1\vee\CC\mathbb{P}^1.\]
\end{enumerate}
Moreover, up to taking duals, there is a unique isomorphism class of $(2g+3)$-dimensional bi-affine representations that are not direct sums of a $(2g+2)$-dimensional bi-affine representation and a trivial representation. This representation is indexed by the pair
\[([\kappa_1]-[\kappa_2])\in \mathcal{E}_{\Mod(S_g^2)}\otimes\CC,
\qquad
[\kappa_1]\otimes(1,0)+[\kappa_2]\otimes(0,1)
\in \mathcal{E}_{\Mod(S_g^2)}\otimes\CC^2.\]
\end{corollary}
\begin{remark}\label{rmk: geometric interpretation} Some of the representations in Corollary~\ref{cor: 2g+2 bi-affine rep} admit natural geometric interpretations.
\begin{enumerate}
    \item The affine/co-affine representation indexed by $0$ is the split extension of $H$ by a trivial module. 
    \item By \cite[Theorem 2.2]{Trapp}, the co-affine representation indexed by the rank-one tensor
   \[[\kappa_i]\otimes(1,0),\quad i=1,2,\]
   is isomorphic to the direct sum of the standard action of $\Mod(S_g^2)$ on
   \[H_1(UT(S_g^1);\CC)\]
   and a trivial representation. Here $UT(S_g^1)$ denotes the unit tangent bundle of $S_g^1$, viewed as a $\Mod(S_g^2)$-module via the homomorphism
   \[\Mod(S_g^2)\longrightarrow \Mod(S_g^1)\]
   induced by capping the boundary component other than the $i$-th one.
   \item The co-affine representation indexed by the rank-one tensor
   \[([\kappa_1]-[\kappa_2])\otimes(1,0)\]
   is isomorphic to the direct sum of the standard action of $\Mod(S_g^2)$ on $H_1(S_g^2;\CC)$ and a trivial representation. This follows from the computation in the proof of Proposition~\ref{prop: singular pt}.
   \item Combining (2) and (3), the co-affine representation indexed by the rank-two tensor
    \[[\kappa_1]\otimes(1,0)+[\kappa_2]\otimes(0,1)\]
    is isomorphic to the standard action of $\Mod(S_g^2)$ on
    \[H_1(UT(S_g^2);\CC).\]
\end{enumerate}
\end{remark}
The singular locus
\[s_1+s_2=0,\quad t_1+t_2=0,\]
in the second case of Corollary~\ref{cor: 2g+2 bi-affine rep} corresponds to a unique isomorphism class of self-dual $(2g+2)$-dimensional bi-affine representations of $\Mod(S_g^2)$ that is neither affine nor co-affine. We now show that this representation admits a natural geometric interpretation.
\begin{proposition}\label{prop: singular pt}
    Let $g\ge 4$. Let
    \[\rho:\Mod(S_g^2)\to \GL_{2g+2}(\CC)\]
    be the bi-affine representation that is neither affine nor co-affine and corresponds to the pair
    \[x=[\kappa_1]-[\kappa_2],\quad y=[\kappa_1]-[\kappa_2]\in \mathcal{E}_{\Mod(S_g^2)}.\]
    Then $\rho$ is isomorphic to the standard action of $\Mod(S_g^2)$ on $H_1(S_{g+1};\CC)$, where $S_g^2$ embeds into $S_{g+1}$ via gluing a pair of pants attached to the two boundary components.
\end{proposition}
\begin{proof}
    First, we show that the representation
    \[\Phi:\Mod(S_g^2)\longrightarrow \GL(H_1(S_{g+1};\CC))\]
    is bi-affine. Let
    \[\{a_1,b_1,\cdots, a_g,b_g\}\]
    be a symplectic basis of $H=H_1(S_g^1;\CC)$
    satisfying \[\omega(a_j,b_k)=\delta_j^k,\quad \omega(a_j,a_k)=\omega(b_j,b_k)=0.\]
    Extend this to a symplectic basis of $H_1(S_{g+1};\CC)$
    \[\{a_1,b_1,\cdots, a_g,b_g,a_{g+1},b_{g+1}\}\]
    where $b_{g+1}$ is represented by a loop around one of the boundary components of $S_g^2$ after the gluing construction. Let 
    \[\langle b_{g+1}\rangle^{\perp}=\{v\in H_1(S_{g+1};\CC)\mid\omega(v,b_{g+1})=0\}.\]
    Then 
    \[0\subset \langle b_{g+1} \rangle \subset \langle b_{g+1}\rangle^{\perp} \subset H_1(S_{g+1};\CC)\]
    is a filtration by $\Mod(S_g^2)$-submodules that satisfies
    \[\langle b_{g+1}\rangle^{\perp}/\langle b_{g+1} \rangle \cong H,\]
    and $\langle b_{g+1}\rangle$ and $ H_1(S_{g+1};\CC)/\langle b_{g+1}\rangle^\perp$ are trivial $\Mod(S_g^2)$-modules, since $\Mod(S_g^2)$ fixes $b_{g+1}$ and preserves $\omega$. Therefore $\Phi$ is bi-affine, determined by a pair
    \[x=s_1[\kappa_1]+s_2[\kappa_2],\quad y=t_1[\kappa_1]+t_2[\kappa_2]\in \mathcal{E}_{\Mod(S_g^2)}.\]
    The key observation is that as $\Mod(S_g^2)$-modules
    \[H_1(S_{g+1};\CC)\cong \operatorname{Hom}(H_1(S_{g+1};\CC),\CC)\]
    induced by the algebraic intersection pairing. Hence, by Remark \ref{rmk: self dual} $\Phi$ is self dual, and therefore
    \[x=y,\]
    up to the scalar action of $\CC^*$. Corollary \ref{cor: 2g+2 bi-affine rep} then gives
    \[(s_1+s_2)(s_1+s_2)=0,\]
    so $s_2=-s_1$. It remains to show that $s_1\neq 0$, or equivalently, that $\Phi$ is neither affine nor co-affine.
    
    Instead, we give a direct computation of the classes $x,y\in\mathcal{E}_{\Mod(S_g^2)}$, using only the property of $\kappa$ in Proposition \ref{prop: image of point pushing}, without relying on the self-duality relation $x=y$.
    
    For $i=1,2$, let $\gamma_i$ be a loop based at the $i$-th boundary component of $S_g^2$ that winds once around $a_g$. Let $\operatorname{Push}(\gamma_i)$ denote the corresponding point-pushing map
    \[\operatorname{Push}(\gamma_i)=T_{\gamma_i^+} T_{\gamma_i^-}^{-1}\] where $\gamma_i^+$ is the simple closed curve lying to the left of $\gamma_i$ and $\gamma_i^-$ is the simple closed curve lying to the right of $\gamma_i$. Then
    \[[\gamma_1^+]=a_g+b_{g+1}, \quad [\gamma_1^-]=a_g, \quad  [\gamma_2^+]=a_g-b_{g+1}, \quad [\gamma_2^-]=a_g.\]
    Then we apply $\text{Push}(\gamma_1)$ to $a_{g+1}$:
    \begin{align*}
        \text{Push}(\gamma_1)(a_{g+1})&=T_{\gamma_1^+}  T_{\gamma_1^-}^{-1}(a_{g+1}) \\
        &=T_{\gamma_1^+}(a_{g+1}-\omega([\gamma_1^-],a_{g+1})[\gamma_1^-])\\
        &=T_{\gamma_1^+}(a_{g+1})\\
        &=a_{g+1}+\omega([\gamma_1^+],a_{g+1})[\gamma_1^+]\\
        &=a_{g+1}-a_g-b_{g+1}.
    \end{align*}
    Let $\sigma:\Mod(S_g^2)\to H$ be the crossed homomorphism appearing in the matrix form \eqref{eq: matrix form} of $\Phi$, representing the cohomology class $x=s_1[\kappa_1]+s_2[\kappa_2]$. The above computation gives
    \[\sigma(\text{Push}(\gamma_1))=-a_g.\]
    On the other hand, Proposition \ref{prop: image of point pushing} gives
    \[\sigma(\text{Push}(\gamma_1))=s_1\kappa_1(\text{Push}(\gamma_1))+s_2\kappa_2(\text{Push}(\gamma_1))=s_1(2-2g)[\gamma_1]=s_1(2-2g)a_g.\]
    Hence\[s_1=\frac{1}{2g-2}.\]Similarly, computing the action of $\operatorname{Push}(\gamma_2)$ on $a_{g+1}$ gives \[s_2=-\frac{1}{2g-2}.\]
    An analogous computation gives $y=\frac{1}{2g-2}([\kappa_1]-[\kappa_2])$.
\end{proof}

In particular, this geometric representation is the unique non-affine, non-co-affine bi-affine representation of $\Mod(S_g^2)$ that factors through the image of
\[
\Mod(S_g^2)\longrightarrow \Mod(S_{g+1}),
\]
as stated in the following theorem.
\begin{theorem}\label{thm: boundary kills}
   Let $g\ge 4$ and $m\ge 2g+2$. Let 
   \[\rho:\Mod(S_g^2)\to \GL_{m}(\CC)\]
   be a bi-affine representation with core $H$ that is neither affine nor co-affine. Suppose that
   \[\rho(T_{\partial_1}T_{\partial_2}^{-1})=I,\]
   where $\partial_i$ denotes the boundary loop around the $i$-th boundary component of $S_g^2$.
   Then $\rho$ is isomorphic to the direct sum of the standard representation in Proposition \ref{prop: singular pt}
   \[\Phi:\Mod(S_g^2)\longrightarrow \GL(H_1(S_{g+1};\CC)),\]
   and a trivial representation.
\end{theorem}
\begin{proof}
    By Corollary~\ref{cor: lowest dim to consider}, it suffices to consider the cases $m=2g+2$ and $m=2g+3$.

    First assume that $m=2g+2$. Let $\rho$ be indexed by a pair of nonzero elements
    \[x=s_1[\kappa_1]+s_2[\kappa_2],\quad y=t_1[\kappa_1]+t_2[\kappa_2]\in \mathcal{E}_{\Mod(S_g^2)}.\]
    By Corollary~\ref{cor: 2g+2 bi-affine rep}, these satisfy
    \[(s_1+s_2)(t_1+t_2)=0.\]
    Under the assumption
    \[\rho(T_{\partial_1}T_{\partial_2}^{-1})=I,\]
    we will prove that
    \begin{equation}\label{eq: st12}
    s_1t_1=s_2t_2.
    \end{equation}
    Together, these two relations imply
    \[s_1+s_2=0,\qquad t_1+t_2=0,\]
    and hence the theorem follows from Proposition~\ref{prop: singular pt}.

    We now prove \eqref{eq: st12}. Let the matrix form (see Remark \ref{rmk: matix form}) of $\rho$ be
    \begin{equation*}
       \rho(f)=\begin{pmatrix}
   1 & \delta(f) & \chi(f) \\
   0 &\Psi(f) & \sigma(f) \\
   0 & 0 & 1   \end{pmatrix},
   \end{equation*}
   where $\Psi$ is the standard action on $H$. We may take 
   \[\delta=s_1\kappa_1^*+s_2\kappa_2^*,\quad \sigma=t_1\kappa_1+t_2\kappa_2,\]
   where $\kappa_i^*$ is the crossed homomorphism defined in Remark~\ref{rmk: dual for crossed}.
   
   A presentation of the kernel of the capping map
   \[\Mod(S_g^n)\longrightarrow \Mod(S_g),\]
   known as the surface braid group, is given in \cite[Theorem~8]{Braids}. The relation we need is
   \begin{equation}\label{eq: surface braid relation}
       \prod\limits_{r=1}^g[B_{r,i},A_{r,i}]=C_{12}(T_{\partial_i})^{2g-2},\qquad i=1,2.
   \end{equation}
   Given a symplectic basis
   \[\{a_1,b_1,\ldots,a_g,b_g\}\]
   of $H$ satisfying
   \[\omega(a_j,b_k)=\delta_j^k,\quad \omega(a_j,a_k)=\omega(b_j,b_k)=0,\]
   the mapping class $B_{r,i}$ is obtained by pushing the $i$-th boundary component along a loop representing $b_r$, while $A_{r,i}$ is obtained by pushing the $i$-th boundary component along a loop representing $a_r$. Here $C_{12}$ denotes the mapping class obtained by pushing one boundary component once around the other.
   
   We now apply $\rho$ to the relation \eqref{eq: surface braid relation}. From the definition of $\kappa_i$ in Lemma \ref{lem: two crossed homomorphisms} and the property in Proposition \ref{prop: image of point pushing}, we have
   \[\kappa_i(B_{r,i})=(2-2g)b_r,\quad \kappa_i(A_{r,i})=(2-2g)a_r,\qquad 1\le r \le g,\ i=1,2, \]
   and $\kappa_j(B_{r,i})=\kappa_j(A_{r,i})=0$ if $i\neq j$. From the definition of $\kappa_i^*$ in Remark~\ref{rmk: dual for crossed}, we similarly have
   \[\kappa_i^*(B_{r,i})=(2-2g)b_r^*,\quad \kappa_i^*(A_{r,i})=(2-2g)a_r^*,\qquad 1\le r \le g,\ i=1,2, \]
   and $\kappa_j^*(B_{r,i})=\kappa_j^*(A_{r,i})=0$ if $i\neq j$. Here
   \[a_r^*(x)=\omega(a_r,x),\qquad b_r^*(x)=\omega(b_r,x), \qquad \forall x\in H.\] Hence
   \[\rho(B_{r,i})=\begin{pmatrix}
   1 & \delta(B_{r,i}) & \chi(B_{r,i}) \\
   0 &\Psi(B_{r,i}) & \sigma(B_{r,i}) \\
   0 & 0 & 1   \end{pmatrix}=\begin{pmatrix}
   1 & (2-2g)s_i b_r & \chi(B_{r,i}) \\
   0 & I & (2-2g)t_i b_r^* \\
   0 & 0 & 1   \end{pmatrix}, \]
   and 
   \[ \rho(A_{r,i})=\begin{pmatrix}
   1 & \delta(A_{r,i}) & \chi(A_{r,i}) \\
   0 &\Psi(A_{r,i}) & \sigma(A_{r,i}) \\
   0 & 0 & 1   \end{pmatrix}=\begin{pmatrix}
   1 & (2-2g)s_i a_r & \chi(A_{r,i}) \\
   0 & I & (2-2g)t_i a_r^* \\
   0 & 0 & 1   \end{pmatrix}.\]
   Then a direct computation gives
   \[\rho([B_{r,i},A_{r,i}])=\begin{pmatrix}
   1 & 0 & (2-2g)^2s_it_i\big(b_r^*(a_r)-a_r^*(b_r)\big) \\
   0 & I & 0 \\
   0 & 0 & 1   \end{pmatrix}=\begin{pmatrix}
   1 & 0 & -2(2-2g)^2s_it_i \\
   0 & I & 0 \\
   0 & 0 & 1   \end{pmatrix}.\]
   Moreover, from the definition, we have
   \[\kappa_i(T_{\partial_j})=0, \quad \kappa_i^*(T_{\partial_j})=0.\]
   Then
   \[ \rho(T_{\partial_i})=\begin{pmatrix}
   1 & \delta(T_{\partial_i}) & \chi(T_{\partial_i}) \\
   0 &\Psi(T_{\partial_i}) & \sigma(T_{\partial_i}) \\
   0 & 0 & 1   \end{pmatrix}=\begin{pmatrix}
   1 & 0 & \chi(T_{\partial_i}) \\
   0 & I & 0 \\
   0 & 0 & 1   \end{pmatrix}.\]
   Applying $\rho$ to the relation \eqref{eq: surface braid relation}, we obtain in the $\chi$-component, for $i=1,2$,
   \begin{equation}\label{eq: compare chi}
   \begin{aligned}
   \chi\left(\prod_{r=1}^g[B_{r,i},A_{r,i}]\right)
   &=\chi\left(C_{12}(T_{\partial_i})^{2g-2}\right),\\
   -2g(2-2g)^2s_it_i
   &=\chi(C_{12})+(2g-2)\chi(T_{\partial_i}).
   \end{aligned}
   \end{equation}
   By the assumption $\rho(T_{\partial_1}T_{\partial_2}^{-1})=I$,
   we have
   \[\chi(T_{\partial_1})=\chi(T_{\partial_2}).\]
   Since $g\ge 3$, comparing the two equations in \eqref{eq: compare chi} for $i=1,2$ gives
   \[s_1t_1=s_2t_2.\]
   This proves the desired relation.
   
   It remains to consider the case $m=2g+3$. Assume that the representation is indexed by the pair
   \[s_1[\kappa_1]+s_2[\kappa_2]\in \mathcal{E}_{\Mod(S_g^2)}\otimes\CC,\]
   and\[
   \big(t_1^{(1)}[\kappa_1]+t_2^{(1)}[\kappa_2]\big)\otimes(1,0)+\big(t_1^{(2)}[\kappa_1]+t_2^{(2)}[\kappa_2]\big)\otimes(0,1)
   \in \mathcal{E}_{\Mod(S_g^2)}\otimes\CC^2.\] 
   The parameters satisfy
   \[(s_1+s_2)(t_1^{(i)}+t_2^{(i)})=0,\qquad i=1,2.\]
   By the same computation as above, we also obtain
   \[s_1t_1^{(i)}=s_2t_2^{(i)},\qquad i=1,2.\]
   Combining these relations gives
   \[s_1+s_2=0,\qquad t_1^{(i)}+t_2^{(i)}=0,\qquad i=1,2.\]
   By the classification in Theorem~\ref{thm: classify bi-affine for n bdry}, the representation is conjugate to the direct sum of a $(2g+2)$-dimensional bi-affine representation and a trivial representation.  The result now follows from the case $m=2g+2$ proved above.
\end{proof}
This theorem will be useful in excluding many bi-affine representations in our classification of linear representations of $\Mod(S,[\beta])$ and $\Mod(\widetilde{S},\sigma)$ in the next section. It also has consequences for the classification of linear representations of $\Mod(S_{g,2})$, which we discuss in a later subsection.

\subsection{Low-dimensional linear representations of $\Mod(S_g^2)$ are bi-affine}\label{sec: 2.4}

We have classified bi-affine representations of $\Mod(S_g^n)$. We now return to the study of general linear representations. The preceding classification becomes effective once one knows that every sufficiently low-dimensional nontrivial linear representation is bi-affine.

This is already known in several cases:
\begin{enumerate}
\item Korkmaz \cite[Theorem~1]{korkmaz} and Franks--Handel \cite[Theorem~1.1]{FranksHandel} proved that for $g\ge 3$, every linear representation of $\Mod(S_{g,r}^n)$ of dimension at most $2g-1$ is trivial.
\item Korkmaz \cite[Theorem~2]{korkmaz} proved that for $g\ge 3$, every nontrivial $2g$-dimensional linear representation of $\Mod(S_{g,r}^n)$ is conjugate to the standard symplectic representation on $H=H_1(S_g;\CC)$.
\item Kasahara \cite[Theorem~1.1]{kasahara} proved that for $g\ge 7$, every nontrivial $(2g+1)$-dimensional linear representation of $\Mod(S_{g,r}^n)$ is bi-affine with core $H$. The genus assumption was subsequently improved to $g\ge 4$ in \cite[Theorem~A]{reps3g-3}.
\item More recently, the author's joint work with Kaufmann, Salter, and Zhang \cite[Theorem~A]{reps3g-3} extended Kasahara's result to linear representations of dimension at most $3g-3$ when $r+n\le 1$.
\end{enumerate}

We now establish the corresponding result for $\Mod(S_g^2)$.
\begin{theorem}\label{thm: any low-dim rep is bi-affine}
Let $g\ge 3$ and $m\le 3g-3$. Then every nontrivial linear representation
\[
\rho:\Mod(S_g^2)\longrightarrow \GL_m(\CC)
\]
is bi-affine with core $H=H_1(S_g;\CC)$.
\end{theorem}
\begin{remark}\label{rmk: quotient group also bi-affine}
   Since $\Mod(S_{g,r}^n)$ is a quotient of $\Mod(S_g^2)$ whenever
$r+n\le 2$, the same result applies to $\Mod(S_{g,r}^n)$
for $r+n\le 2$.
\end{remark}

The proof of Theorem~\ref{thm: any low-dim rep is bi-affine} uses the theory of partitioned surfaces developed by Putman \cite{cuttingandpasting} and further extended by Church \cite{ChurchPartitionedJohnsonHom}. This theory is likely to play an important role in obtaining analogous results for mapping class groups of surfaces with more than two boundary components. We recall the results from this theory that will be used in the proof.
\begin{enumerate}
    \item Let
    \[\Sigma=(S_g^2,P),\qquad P=\{\{\partial_1,\partial_2\}\},\]
    be the surface $S_g^2$ equipped with the partition $P$ consisting of a singleton of its two boundary components. The canonical embedding
    \[\operatorname{Emb}:S_g^2\hookrightarrow S_{g+1}^1,\]
    obtained by gluing a pair of pants, respects the partition $P$ and induces a homomorphism
    \[\operatorname{Emb}_*:\Mod(S_g^2)\longrightarrow \Mod(S_{g+1}^1).\]

    The partitioned Torelli group $\mathcal{I}(\Sigma)$ is defined by
    \[\mathcal{I}(\Sigma)=\{f\in\Mod(S_g^2)\mid \operatorname{Emb}_*(f)\text{ acts trivially on }H_1(S_{g+1}^1;\ZZ)\}.\]

    The definition of the partitioned Johnson kernel $\mathcal{K}(\Sigma)$ is more technical, since it also involves the arc connecting the two boundary components. Instead, using the functoriality of the partitioned Johnson kernel, we recall the equivalent characterization
    \[\mathcal{K}(\Sigma)=\{f\in\Mod(S_g^2)\mid
    \operatorname{Emb}_*(f)\in\mathcal{K}(S_{g+1}^1)\},\]
    where $\mathcal{K}(S_{g+1}^1)$ denotes the ordinary Johnson kernel of $\Mod(S_{g+1}^1)$.
    \item By \cite[Theorem~A]{JohnsonKernel}, the partitioned Johnson kernel $\mathcal{K}(\Sigma)$ is generated by 
    \begin{equation}\label{gen: Pseparatingtwists}
    \{\,T_\gamma \mid \operatorname{Emb}(\gamma)\text{ is a separating simple closed curve in }S_{g+1}^1\,\}.
    \end{equation}
    \item By \cite[Theorem~1.3]{cuttingandpasting}, $\mathcal{I}(\Sigma)$ is generated
    by the separating twists in \eqref{gen: Pseparatingtwists} and bounding pair
    maps $T_{\gamma_1}T_{\gamma_2}^{-1}$, where $\gamma_1,\gamma_2$ are disjoint homologous nonseparating simple closed curves in $S_{g+1}^1$ contained in
    $S_g^2$.
    \item The quotient
    \[\mathcal{I}(\Sigma)/\mathcal{K}(\Sigma)\]
    is abelian. More precisely, \cite[Corollary~5.7]{ChurchPartitionedJohnsonHom} identifies it with
    \[\wedge^3\!\left(\operatorname{Emb}_*\big(H_1(S_g^2;\ZZ)\big)\right).\]
\end{enumerate}
The restriction of $\rho$ to the partitioned Torelli group $\mathcal{I}(\Sigma)$ satisfies the following structural property.

\begin{lemma}\label{lem: abel and unipotent}
Let $g\ge 3$ and $m\le 3g-3$. Let
\[\rho:\Mod(S_g^2)\longrightarrow \GL_m(\CC)\]
be a homomorphism. Then the image of the partitioned Torelli group $\mathcal{I}(\Sigma)$ under $\rho$ is an abelian unipotent subgroup of $\GL_m(\CC)$.
\end{lemma}
\begin{proof}
    First, we prove that the image of $\mathcal{I}(\Sigma)$ is abelian. It suffices to show that $\rho$ is trivial on the partitioned Johnson kernel $\mathcal{K}(\Sigma)$, since the quotient
    \[\mathcal{I}(\Sigma)/\mathcal{K}(\Sigma)\]
    is abelian. We know that $\mathcal{K}(\Sigma)$ is generated by Dehn twists
    $T_\gamma$ where $\gamma$ is a separating simple closed curve in $S_{g+1}^1$
    contained in $S_g^2$. Since the two boundary components of $S_g^2$ lie on the same side of $\gamma$, the curve $\gamma$ is contained in a subsurface of $S_g^2$ homeomorphic to $S_g^1$. Thus, it suffices to consider the restriction
    \[\Mod(S_g^1)\hookrightarrow \Mod(S_g^2)\xrightarrow{\rho}\GL_m(\CC).\]
    By \cite[Proposition~6.7]{reps3g-3}, any homomorphism
    \[\Mod(S_g^1)\to \GL_m(\CC),\qquad m\le 3g-3,\]
    kills the Johnson kernel $\mathcal{K}(S_g^1)$, which is generated by
    separating twists. Hence $\rho(T_\gamma)=I$ for every generator $T_\gamma$ of
    $\mathcal{K}(\Sigma)$, and therefore $\rho$ is trivial on $\mathcal{K}(\Sigma)$.

    It remains to prove that the image is unipotent. By \cite[Proposition~5.1]{reps3g-3},
    for any homomorphism
    \[\Mod(S_{g,r}^n)\longrightarrow \GL_m(\CC)\]
    with $g\ge4$ and $m\le 4g-3$, the image of every Dehn twist about a nonseparating simple closed curve is unipotent. For $g=3$ and $m\le6$, the same conclusion follows from \cite[Theorem~2]{korkmaz}.
    
    Recall that $\mathcal{I}(\Sigma)$ is generated by the generators of
    $\mathcal{K}(\Sigma)$ and bounding pair maps
    \[T_{\gamma_1}T_{\gamma_2}^{-1},\]
    where $\gamma_1,\gamma_2$ are disjoint homologous nonseparating simple closed curves in $S_{g+1}^1$ contained in $S_g^2$. Since $\rho$ is trivial on $\mathcal{K}(\Sigma)$, it remains to consider the bounding pair maps. By the argument above, both $\rho(T_{\gamma_1})$ and $\rho(T_{\gamma_2})$ are
    unipotent. Moreover, they commute, and hence 
    \[\rho(T_{\gamma_1}T_{\gamma_2}^{-1})=\rho(T_{\gamma_1})\rho(T_{\gamma_2})^{-1}\]
    is unipotent. Since $\rho(\mathcal{I}(\Sigma))$ is abelian and generated by unipotent elements, every element of $\rho(\mathcal{I}(\Sigma))$ is unipotent.
\end{proof}

Before proving Theorem~\ref{thm: any low-dim rep is bi-affine}, we establish several
useful lemmas. We first give a finite generating set for $\Mod(S_g^2)$.

\begin{lemma}\label{lem: Mod S_g^2 generators}
Let $S_g^1$ denote the subsurface of $S_g^2$ obtained by cutting off a pair of
pants. Then $\Mod(S_g^2)$ is generated by a finite generating set of
$\Mod(S_g^1)$ together with the Dehn twists about the curves
$\gamma_i$ ($1\le i\le 2g+1$) shown in Figure~\ref{fig: gens of mod S_g^2}.
\begin{figure}
\centering
\includegraphics[width=0.48\linewidth,height=3.5cm,keepaspectratio]{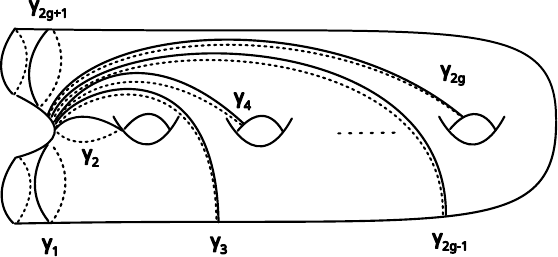}
\caption{Generators of $\Mod(S_g^2)$ not supported on $S_g^1$}
\label{fig: gens of mod S_g^2}
\end{figure}
\end{lemma}
\begin{proof}
An explicit finite generating set for $\Mod(S_g^n)$ was given in
\cite[Theorem~1]{GenerateModNonclosed}. We observe that, for the case
$\Mod(S_g^2)$, the generators in this set that are not supported on the
subsurface $S_g^1$ are precisely the Dehn twists about the curves
\[
\{\gamma_i:1\le i\le 2g+1\}
\]
shown in Figure~\ref{fig: gens of mod S_g^2}.
\end{proof}

We note the following linear-algebra fact.
\begin{lemma}\label{lem: determine T gamma2}
Let $V=H\oplus \CC^{m-2g}$, and let
\[\{a_1,b_1,\dots,a_g,b_g\}\]
be a symplectic basis of $H$ such that
\[\omega(a_j,b_k)=\delta_{jk},\qquad\omega(a_j,a_k)=\omega(b_j,b_k)=0.\]
For $x\in H$, let $T_x\in \GL(V)$ be the transvection
\[T_x(y)=y+\omega(x,y)x,\]
extended by the identity on $\CC^{m-2g}$.

Suppose $M\in \GL(V)$ is unipotent, commutes with
\[T_{b_1},T_{a_2},T_{b_2},\dots,T_{a_g},T_{b_g},\]
and satisfies the braid relation 
\begin{equation}\label{eq: braid assumption}
    T_{a_1}MT_{a_1}=MT_{a_1}M.
\end{equation}
Then $M$ fixes $b_1,a_2,b_2,\dots,a_g,b_g$, and there exist
\[
M'\in \GL(\CC^{m-2g}),\qquad v\in \CC^{m-2g},\qquad w^*\in \operatorname{Hom}(\CC^{m-2g},\CC)
\]
such that
\[M(a_1)=a_1-b_1+v,\qquad M(x)=M'x+w^*(x)b_1\quad \forall x\in \CC^{m-2g},
\]
and
\begin{equation}\label{eq: relations in matrix}
    M'v=v, \qquad w^*M'=w^*,\qquad w^*(v)=0,\qquad (M')^2-M'=-v\otimes w^*.
\end{equation}
\end{lemma}
\begin{proof}
Note that
\[\operatorname{Im}(T_{a_i}-I)=\CC a_i,\qquad 
\operatorname{Im}(T_{b_j}-I)=\CC b_j,\qquad 2\leq i\leq g,\ 1\leq j\leq g.\]
These subspaces are preserved by $M$ since $M$ commutes with
\[T_{b_1},T_{a_2},T_{b_2},\dots,T_{a_g},T_{b_g}.\]
Moreover, since $M$ is unipotent, it fixes 
$b_1,a_2,b_2,\dots,a_g,b_g$.

On the other hand, $M$ preserves the subspaces
\[\left(\bigcap_{i=2}^g\operatorname{Ker}(T_{a_i}-I)\right)
\cap
\left(\bigcap_{j=1}^g\operatorname{Ker}(T_{b_j}-I)\right)
=
\CC b_1\oplus \CC^{m-2g}\]
and
\[\left(\bigcap_{i=2}^g\operatorname{Ker}(T_{a_i}-I)\right)
\cap
\left(\bigcap_{j=2}^g\operatorname{Ker}(T_{b_j}-I)\right)=
\CC a_1\oplus\CC b_1\oplus \CC^{m-2g}.\]
Therefore,
\[M(a_1)=a_1+cb_1+v,\qquad 
M(x)=M'x+w^*(x)b_1\quad \forall x\in\CC^{m-2g},\]
for some
\[
M'\in\GL(\CC^{m-2g}),\qquad 
v\in\CC^{m-2g},\qquad 
w^*\in\operatorname{Hom}(\CC^{m-2g},\CC),\qquad 
c\in\CC.\]

We now apply the braid relation \eqref{eq: braid assumption}. It gives, in $\GL(\CC a_1\oplus\CC b_1\oplus \CC^{m-2g})$,
\[\begin{pmatrix}
1&1&0\\
0&1&0\\
0&0&I
\end{pmatrix}\cdot
\begin{pmatrix}
1&0&0\\
c&1&w^*\\
v&0&M'
\end{pmatrix}\cdot
\begin{pmatrix}
1&1&0\\
0&1&0\\
0&0&I
\end{pmatrix}
=
\begin{pmatrix}
1&0&0\\
c&1&w^*\\
v&0&M'
\end{pmatrix}\cdot
\begin{pmatrix}
1&1&0\\
0&1&0\\
0&0&I
\end{pmatrix}\cdot
\begin{pmatrix}
1&0&0\\
c&1&w^*\\
v&0&M'
\end{pmatrix}.\]
Expanding both sides, we obtain
\[
\begin{pmatrix}
1+c&2+c&w^*\\
c&1+c&w^*\\
v&v&M'
\end{pmatrix}
=
\begin{pmatrix}
1+c&1&w^*\\
c(2+c)+w^*(v)&1+c&(1+c)w^*+w^*M'\\
(1+c)v+M'v&v&v\otimes w^*+(M')^2
\end{pmatrix}.
\]
Comparing entries gives $c=-1$ and the relations in
\eqref{eq: relations in matrix}.
\end{proof}

\begin{corollary}\label{cor: M2 preserve imply M}
Let $M\in \GL(H\oplus \CC^{m-2g})$ be as in Lemma~\ref{lem: determine T gamma2}. If $M^2$ preserves $H$ (resp.~$\CC^{m-2g}$), then $M$ also preserves $H$ (resp.~$\CC^{m-2g}$).
\end{corollary}
\begin{proof}
Suppose first that $M^2$ preserves $H$. We compute
\[
M^2(a_1)=M(a_1-b_1+v)
=a_1-b_1+v-b_1+M'v+w^*(v)b_1
=a_1+(w^*(v)-2)b_1+2v,
\]
where the last equality follows from the relation $M'v=v$. Since $M^2(a_1)\in H$, we obtain $v=0$. Hence $M$ preserves $H$.

Now suppose that $M^2$ preserves $\CC^{m-2g}$. For $x\in \CC^{m-2g}$, we have
\[
M^2(x)=M(M'x+w^*(x)b_1)
=(M')^2x+w^*(M'x)b_1+w^*(x)b_1
=(M')^2x+2w^*(x)b_1,
\]
where the last equality follows from the relation $w^*M'=w^*$. Since $M^2(x)\in \CC^{m-2g}$ for all $x\in\CC^{m-2g}$, we must have $w^*=0$. Hence $M$ preserves $\CC^{m-2g}$.
\end{proof}

We now prove that any linear representation of $\Mod(S_g^2)$ of dimension at most $3g-3$ must be bi-affine.
\begin{proof}[Proof of Theorem \ref{thm: any low-dim rep is bi-affine}]
    By Lemma~\ref{lem: abel and unipotent}, the image of $\mathcal{I}(\Sigma)$ under $\rho$ is an abelian unipotent subgroup of $\GL_m(\CC)$. Hence the invariant subspace
    \[I=\{v\in \CC^m\mid f\cdot v=v,\ \forall f\in \mathcal{I}(\Sigma)\}\] is nonzero. Since $\mathcal{I}(\Sigma)$ is a normal subgroup of $\Mod(S_g^2)$, the subspace $I$ is preserved by the action of $\Mod(S_g^2)$. Therefore, $\rho$ induces representations
    \[\rho|_{I}:\Mod(S_g^2)\longrightarrow \GL(I),\]
    and
    \[\rho|_{\CC^m/I}:\Mod(S_g^2)\longrightarrow \GL(\CC^m/I).\]

    Suppose first that $\dim_{\CC}(I)<m$. Then both induced representations have dimension strictly smaller than $m$, and we argue by induction on $m$. The base case is immediate. By the induction hypothesis, both $\rho|_I$ and $\rho|_{\CC^m/I}$ are either trivial or bi-affine with core $H$. If both are trivial, then $\rho$ has solvable image, and hence is trivial since $\Mod(S_g^2)$ is perfect. 

    Suppose that one of the two induced representations is bi-affine with core $H$. Then the other one must be trivial since $m\le 3g-3<4g$. Therefore, Proposition~\ref{prop: extend bi-affine} implies that $\rho$ is bi-affine with core $H$.

    It remains to consider the case $\dim_{\CC} I=m$, namely when $\rho$ is trivial on $\mathcal{I}(\Sigma)$. Let $S_g^1$ be the subsurface of $S_g^2$ obtained by cutting off a pair of pants. This induces an inclusion
    \[\Mod(S_g^1)\hookrightarrow\Mod(S_g^2).\]
    Let $\rho|_{\Mod(S_g^1)}$ denote the restriction of $\rho$ to this subgroup. By \cite[Theorem~A]{reps3g-3}, the representation $\rho|_{\Mod(S_g^1)}$ is either trivial or affine/co-affine with core $H$. In the latter case, $\rho|_{\Mod(S_g^1)}$ must be a direct sum of the standard action on $H$ and a trivial representation. Indeed, the crossed homomorphism $\kappa$ generating $\mathcal{E}_{\Mod(S_g^1)}$ does not vanish on the Torelli subgroup $\mathcal{I}(S_g^1)$ by Proposition~\ref{prop: image of point pushing}. If $\rho|_{\Mod(S_g^1)}$ corresponds to a nonzero element of
    \[\mathcal{E}_{\Mod(S_g^1)}\otimes \CC^{m-2g},\]
    then its restriction to $\mathcal{I}(S_g^1)$ would be nontrivial, contradicting our assumption that $\rho$ is trivial on $\mathcal{I}(\Sigma)$, since $\mathcal{I}(S_g^1)$ is a subgroup of $\mathcal{I}(\Sigma)$.
    
    Suppose first that $\rho|_{\Mod(S_g^1)}$ is trivial. By Lemma~\ref{lem: Mod S_g^2 generators}, the group $\Mod(S_g^2)$ is generated by a finite generating set of $\Mod(S_g^1)$ together with the Dehn twists about the curves $\gamma_i$ ($1\leq i\leq 2g+1$) shown in Figure~\ref{fig: gens of mod S_g^2}. Since the curves $\gamma_i$ are pairwise disjoint, their Dehn twists commute. Therefore, the image of $\rho$ is abelian, and hence is trivial since $\Mod(S_g^2)$ is perfect.
    
    It remains to consider the case where $\rho|_{\Mod(S_g^1)}$ is a direct sum of the standard action on $H$ and an $(m-2g)$-dimensional trivial representation. By induction, it suffices to show that there exists a proper subspace of $\CC^m$ preserved by the action of $\Mod(S_g^2)$. Then the conclusion follows from the extension property of bi-affine representations (Proposition~\ref{prop: extend bi-affine}) together with the bound on $m$.
    
    \begin{figure}
        \centering\includegraphics[width=0.50\linewidth]{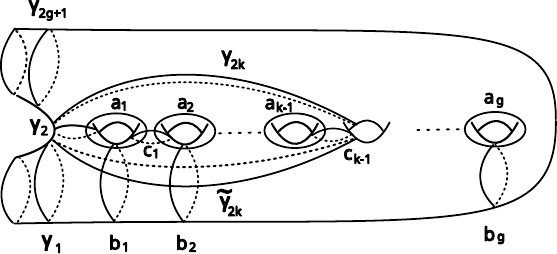}\caption{Chain relation for $\gamma_{2k}$\label{fig: relation for even gamma}}\end{figure}
        
    Realize the symplectic basis of $H$
    \[a_1,b_1,\cdots,a_g,b_g\]
    as simple closed curves in Figure~\ref{fig: relation for even gamma}, with the same notation. The curve $\gamma_2$ is disjoint from
    \[b_1,a_2,b_2,\cdots,a_g,b_g\]
    and intersects $a_1$ once. Therefore, $\rho(T_{\gamma_2})$ commutes with
    \[\rho(T_{b_1}),\rho(T_{a_2}),\rho(T_{b_2}),\cdots,\rho(T_{a_g}),\rho(T_{b_g}),\]
    and satisfies the braid relation (\cite[Proposition 3.11]{Primer})
   \[\rho(T_{a_1})\rho(T_{\gamma_2})\rho(T_{a_1})=\rho(T_{\gamma_2})\rho(T_{a_1})\rho(T_{\gamma_2}).\]
   Hence, by Lemma~\ref{lem: determine T gamma2} $\rho(T_{\gamma_2})$ fixes $b_1,a_2,b_2,\dots,a_g,b_g\in H$, and there exist
    \[M'\in \GL(\CC^{m-2g}),\qquad v\in \CC^{m-2g},\qquad w^*\in \operatorname{Hom}(\CC^{m-2g},\CC)\]
    such that
    \[\rho(T_{\gamma_2})(a_1)=a_1-b_1+v,\qquad \rho(T_{\gamma_2})(x)=M'x+w^*(x)b_1\quad \forall x\in \CC^{m-2g},\]
    and
    \begin{equation}\label{eq: relations in matrix-re}
    M'v=v, \qquad w^*M'=w^*,\qquad w^*(v)=0,\qquad (M')^2-M'=-v\otimes w^*.
    \end{equation}
    
   We proceed by considering two cases depending on whether $v$ vanishes. In either case, we construct a proper $\Mod(S_g^2)$-invariant subspace of $\CC^m$.
   \begin{enumerate}
   \item Suppose that $v=0$. We show that $H$ is preserved by $\Mod(S_g^2)$. By Lemma~\ref{lem: Mod S_g^2 generators}, it suffices to show that $\CC^{m-2g}$ is preserved by $\rho(T_{\gamma_i})$ for $1\leq i\leq 2g+1$.
   
   First, note that, by definition,
   \[T_{\gamma_1}T_{\gamma_{2k+1}}^{-1}\in\mathcal{I}(\Sigma),
    \qquad 1\leq k\leq g.\]
    Therefore, it suffices to consider $T_{\gamma_1}$ for the odd-indexed curves. Since $\gamma_1$ is disjoint from the curves
   \[a_1,b_1,\cdots,a_g,b_g,\]
   the element $\rho(T_{\gamma_1})$ preserves
   \[\operatorname{Span}\left(\operatorname{Im}(T_{a_i}-I),\operatorname{Im}(T_{b_i}-I)\mid 1\leq i\leq g\right)=H.\]
   
   For $T_{\gamma_{2k}}$ with $1\leq k\leq g$, we already know that $T_{\gamma_2}$ preserves $H$ since $v=0$. For the remaining even-indexed curves, we apply the chain relation \cite[Proposition 4.12]{Primer} indicated in Figure~\ref{fig: relation for even gamma}:
   \[(T_{\gamma_2}T_{a_1}T_{c_1}T_{a_2}T_{c_2}\cdots
   T_{a_{k-1}}T_{c_{k-1}})^{2k}=
   T_{\gamma_{2k}}T_{\widetilde{\gamma}_{2k}},\]
   where $\widetilde{\gamma}_{2k}$ is the curve obtained by applying the hyperelliptic involution to $\gamma_{2k}$. Since
   \[T_{\gamma_{2k}}T_{\widetilde{\gamma}_{2k}}^{-1}\in\mathcal{I}(\Sigma),\]
   and $\rho$ is trivial on $\mathcal{I}(\Sigma)$, we obtain
   \begin{equation}\label{eq: chain relation for gamma2k}
   \left(\rho(T_{\gamma_2}T_{a_1}T_{c_1}T_{a_2}T_{c_2}\cdots T_{a_{k-1}}T_{c_{k-1}})\right)^{2k}=\rho(T_{\gamma_{2k}})^2.\end{equation}
   The left-hand side preserves $H$, and therefore $\rho(T_{\gamma_{2k}})^2$ preserves $H$.

   Observe that $\gamma_{2k}$ is disjoint from
   \[a_1,b_1,\cdots,\widehat{a_k},b_k,\cdots,a_g,b_g,\]
   and intersects $a_k$ once. Hence $\rho(T_{\gamma_{2k}})$ has the form described in Lemma~\ref{lem: determine T gamma2}. By Corollary~\ref{cor: M2 preserve imply M}, since $\rho(T_{\gamma_{2k}})^2$ preserves $H$, the same is true for $\rho(T_{\gamma_{2k}})$.

   \item Suppose that $v\neq0$. We show that $v$ is fixed by $\Mod(S_g^2)$. First, $\rho(T_{\gamma_2})$ fixes $v$, since
   \[\rho(T_{\gamma_2})(v)=M'v+w^*(v)b_1=v,\]
   where the last equality follows from the relations in \eqref{eq: relations in matrix-re}. By the above discussion, it remains to show that $\rho(T_{\gamma_1})$ and $\rho(T_{\gamma_{2k}})$ for $2\leq k\leq g$ also fix $v$.

   Since $\rho(T_{\gamma_1})$ and $\rho(T_{\gamma_{2k}})$ are unipotent by \cite[Proposition~5.1]{reps3g-3}, it suffices to show that their squares fix $v$. The chain relation \eqref{eq: chain relation for gamma2k} implies that $\rho(T_{\gamma_{2k}})^2$ fixes $v$, since the left-hand side of \eqref{eq: chain relation for gamma2k} fixes $v$.
   
   \begin{figure} \centering\includegraphics[width=0.48\linewidth]{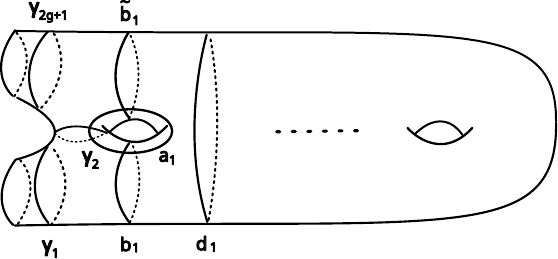}\caption{Star relation for $\gamma_1$\label{fig: star relation}}\end{figure}

   For $\rho(T_{\gamma_1})^2$, consider the star relation \cite[Theorem 1]{GenerateModNonclosed} indicated in Figure~\ref{fig: star relation}:
   \[(T_{\gamma_2}T_{\widetilde{b}_1}T_{b_1}T_{a_1})^3=T_{\gamma_{2g+1}}T_{\gamma_1}T_{d_1}.\]
   Applying $\rho$ to both sides gives
   \[\big(\rho(T_{\gamma_2})\rho(T_{b_1})^2\rho(T_{a_1})\big)^3=\rho(T_{\gamma_1})^2,\]
   where we have used the fact that \[T_{\gamma_1}T_{\gamma_{2g+1}}^{-1}\in\mathcal{I}(\Sigma)\]
   and that $\rho$ is trivial on $\mathcal{I}(\Sigma)$. Hence $\rho(T_{\gamma_1})^2$ fixes $v$, since the left-hand side fixes $v$. \qedhere
    \end{enumerate}
\end{proof}

\subsection{A discussion on $\Mod(S_{g,n})$}\label{sec: 2.5}
So far, we have focused on linear representations of $\Mod(S_g^n)$, where $S_g^n$ is a genus-$g$ surface with $n$ boundary components. We now study bi-affine representations of $\Mod(S_{g,n})$, where $S_{g,n}$ is a genus-$g$ surface with $n$ punctures. In contrast to the case of surfaces with boundary, new phenomena arise for punctured surfaces. As an application, we show that every linear representation of $\Mod(S_{g,2})$ of dimension at most $3g-3$ is affine or co-affine.

There are two natural approaches to classfifying bi-affine representations of $\Mod(S_{g,n})$.

The first approach is to relate $\Mod(S_{g,n})$ to $\Mod(S_g^n)$ via the capping map. By capping each boundary component of $S_g^n$ with a punctured disk, we obtain the central extension
\[1\to \ZZ^n \to \Mod(S_g^n)\to \Mod(S_{g,n})\to 1,\]
where the kernel is generated by the Dehn twists about the boundary components of $S_g^n$.

Consequently, every bi-affine representation of $\Mod(S_{g,n})$ with core $H$ induces a bi-affine representation of $\Mod(S_g^n)$ with the same core $H$, which is completely classified in Theorem~\ref{thm: classify bi-affine for n bdry}. Conversely, a bi-affine representation
\[\rho:\Mod(S_g^n)\longrightarrow \GL(V)\]
with core $H$ descends to a bi-affine representation of $\Mod(S_{g,n})$ if and only if
\[\rho(T_{\partial_i})=I,\qquad 1\le i\le n,\]
where $\partial_i$ denotes a loop around the $i$-th boundary component of $S_g^n$.

Every affine or co-affine representation of $\Mod(S_g^n)$ automatically descends to $\Mod(S_{g,n})$, since the corresponding cohomology classes in
\[
\mathcal{E}_{\Mod(S_g^n)}
\cong
\mathcal{E}_{\Mod(S_{g,n})}
\]
vanish on the boundary twists. In contrast, for bi-affine representations that are neither affine nor co-affine, the condition that the boundary twists act trivially imposes additional constraints on the corresponding cohomology classes in Theorem~\ref{thm: classify bi-affine for n bdry}.

A second approach is to apply the general classification of bi-affine representations given in Proposition~\ref{prop: classify bi-affine} and compute the contraction form
\[
\mathcal{B}_{\Mod(S_{g,n})}(-,-)
\]
explicitly. This leads to a complete classification of bi-affine representations of $\Mod(S_{g,n})$ as follows.

The capping map
\[\Mod(S_g^n)\longrightarrow \Mod(S_{g,n})\]
induces an isomorphism
\[\mathcal{E}_{\Mod(S_{g,n})}=H^1(\Mod(S_{g,n});H)
\longrightarrow
H^1(\Mod(S_g^n);H)
=\mathcal{E}_{\Mod(S_g^n)}.\]
Via this isomorphism, we also denote by
\[[\kappa_1],\ldots,[\kappa_n],\]
the corresponding generators of $\mathcal{E}_{\Mod(S_{g,n})}$. The classification is as follows.
\begin{theorem}\label{thm: classify bi-affine for n punctures}
   Let $g\ge 4$. Every bi-affine representation of $\Mod(S_{g,n})$ with core $H=H_1(S_g;\CC)$ of dimension $2g+m$ is of one of the following types:
   \begin{enumerate}
       \item Affine or co-affine. In either case, the isomorphism classes are classified by
       \[(\mathcal{E}_{\Mod(S_{g,n})}\otimes \CC^m )/\GL_m(\CC)\cong (\CC^n\otimes \CC^m)/\GL_m(\CC)\cong
      \bigsqcup_{r=0}^{\min(n,m)}\operatorname{Gr}(r,\CC^n).\]
      \item Neither affine nor co-affine. Such representations are classified by pairs of nonzero elements
      \[x=\sum_{k=1}^d\big(\sum_{i=1}^n s_i^{(k)}[\kappa_i]\big)\otimes v_k\in\mathcal{E}_{\Mod(S_{g,n})}\otimes\CC^d,\]
      and
      \[y=\sum_{\ell=1}^{m-d}\big(\sum_{i=1}^n t_i^{(\ell)}[\kappa_i]\big)\otimes w_\ell\in
      \mathcal{E}_{\Mod(S_{g,n})}\otimes\CC^{m-d},\]
      where $0<d<m$, and $\{v_k\}$ and $\{w_\ell\}$ are the standard bases of $\CC^d$ and $\CC^{m-d}$, respectively.

      For $1\le k\le d$ and $1\le \ell\le m-d$, let
      \[S^{(k)}=\sum_{i=1}^n s_i^{(k)},\qquad T^{(\ell)}=\sum_{i=1}^n t_i^{(\ell)}.\]
      Then the coefficients satisfy
      \[S^{(k)}T^{(\ell)}=0,\]
      and
      \[s_i^{(k)}T^{(\ell)}+t_i^{(\ell)}S^{(k)}+2(g-1)s_i^{(k)}t_i^{(\ell)}=0,\qquad1\le i\le n.\]
      Two such representations are isomorphic if the pairs differ by the natural action of
      \[\GL_d(\CC)\times\GL_{m-d}(\CC).\]
   \end{enumerate}
\end{theorem}
\begin{proof}
By Proposition~\ref{prop: classify bi-affine}, it suffices to determine when the symmetric contraction form
\[\begin{aligned}\mathcal{B}_{\Mod(S_{g,n})}:\mathcal{E}_{\Mod(S_{g,n})}\times\mathcal{E}_{\Mod(S_{g,n})}&\longrightarrow H^2(\Mod(S_{g,n});\CC),\\
(x,y)&\longmapsto
\omega_*(x\cup y)
\end{aligned}\]
vanishes. Write
\[x=\sum_{i=1}^n s_i[\kappa_i],\qquad y=\sum_{i=1}^n t_i[\kappa_i],\]
then by bilinearity, it suffices to compute
\[\mathcal{B}_{\Mod(S_{g,n})}([\kappa_i],[\kappa_j]).\]

By \cite[Proposition~2.2]{Looijenga},
\[H^2(\Mod(S_{g,n});\CC)\cong\CC^{n+1},\]
generated by the first Miller-Morita-Mumford class $e_1$ together with the first Chern classes
\[c_1(\theta_i),\quad 1\le i\le n,\]
where $\theta_i$ is the pullback of the relative tangent bundle of the universal curve
\[\mathcal{M}_{g,1}\longrightarrow\mathcal{M}_g\]
along the forgetful map
\[\mathcal{M}_{g,n}\longrightarrow\mathcal{M}_{g,1}\]
forgetting all but the $i$-th marked point.

Forgetting all but the $i$-th puncture of $S_{g,n}$ yields the commutative diagram
\begin{equation}\label{eq: comm diagram k i}
\xymatrix@C=5em{
\mathcal{E}_{\Mod(S_{g,1})}\times\mathcal{E}_{\Mod(S_{g,1})}
\ar[r]^-{\mathcal{B}_{\Mod(S_{g,1})}}
\ar[d]^{\mathrm{Forget}^*\times\mathrm{Forget}^*}& H^2(\Mod(S_{g,1});\CC)
\ar[d]^{\mathrm{Forget}^*}\\
\mathcal{E}_{\Mod(S_{g,n})}\times\mathcal{E}_{\Mod(S_{g,n})}
\ar[r]^-{\mathcal{B}_{\Mod(S_{g,n})}}&
H^2(\Mod(S_{g,n});\CC).}
\end{equation}
By \cite[Example~1.4]{KawazumiMoritaStable},
\[\mathcal{B}_{\Mod(S_{g,1})}([\kappa_i],[\kappa_i])=-e_1-4g(g-1)c_1(\theta_i).\]
Since the vertical maps in \eqref{eq: comm diagram k i} are injective,
\begin{equation}\label{eq: pair k ii}
\mathcal{B}_{\Mod(S_{g,n})}([\kappa_i],[\kappa_i])
=
-e_1-4g(g-1)c_1(\theta_i).
\end{equation}

Next suppose that $i\neq j$. Forgetting all punctures of $S_{g,n}$ except the $i$-th and the $j$-th ones yields the commutative diagram
\begin{equation}\label{eq: comm diagram k i j}
\xymatrix@C=5em{
\mathcal{E}_{\Mod(S_{g,2})}\times\mathcal{E}_{\Mod(S_{g,2})}
\ar[r]^-{\mathcal{B}_{\Mod(S_{g,2})}}
\ar[d]^{\mathrm{Forget}^*\times\mathrm{Forget}^*}&
H^2(\Mod(S_{g,2});\CC)
\ar[d]^{\mathrm{Forget}^*}\\
\mathcal{E}_{\Mod(S_{g,n})}\times\mathcal{E}_{\Mod(S_{g,n})}
\ar[r]^-{\mathcal{B}_{\Mod(S_{g,n})}}&
H^2(\Mod(S_{g,n});\CC).}
\end{equation}
By \cite[Theorem~6.1]{KawazumiMoritaStable},
\[\mathcal{B}_{\Mod(S_{g,2})}(k_0,k_0)=-c_1(\theta_i)-c_1(\theta_j),\]
where
\[k_0=\frac{1}{2g-2}\bigl([\kappa_i]-[\kappa_j]\bigr)\]
is the class computed in Proposition~\ref{prop: singular pt}. Therefore,
\begin{equation}\label{eq: pair k i-j}
\mathcal{B}_{\Mod(S_{g,n})}
([\kappa_i]-[\kappa_j],[\kappa_i]-[\kappa_j])=
-(2g-2)^2\bigl(c_1(\theta_i)+c_1(\theta_j)\bigr).
\end{equation}
Using \eqref{eq: pair k ii}, the symmetry of $\mathcal{B}_{\Mod(S_{g,n})}$, and expanding \eqref{eq: pair k i-j}, we obtain
\begin{equation}\label{eq: pair k i j}
\mathcal{B}_{\Mod(S_{g,n})}([\kappa_i],[\kappa_j])=
-e_1-(2g-2)\bigl(c_1(\theta_i)+c_1(\theta_j)\bigr),
\qquad i\neq j.
\end{equation}

Combining \eqref{eq: pair k ii} and \eqref{eq: pair k i j}, we conclude that
\[\mathcal{B}_{\Mod(S_{g,n})}(x,y)=-ST\,e_1-2(g-1)\sum_{i=1}^n
\bigl(s_iT+t_iS+2(g-1)s_it_i\bigr)
c_1(\theta_i),\]
where 
\[S=\sum_{i=1}^n s_i,\qquad T=\sum_{i=1}^n t_i.\]
Since
\[e_1,\,c_1(\theta_1),\,\dots,\,c_1(\theta_n)\]
form a basis of
$H^2(\Mod(S_{g,n});\CC)$,
the contraction form vanishes if and only if
\[ST=0,\qquad
s_iT+t_iS+2(g-1)s_it_i=0,\quad
1\le i\le n.\]
This completes the proof.
\end{proof}

Combining Theorem~\ref{thm: any low-dim rep is bi-affine} with the above analysis, we obtain the following complete classification of linear representations of $\Mod(S_{g,2})$ of dimension at most $3g-3$, in which no representations other than affine or co-affine representations occur.
\begin{corollary}\label{cor: 2 puncture only affine}
Let $g\geq 3$ and $m\leq 3g-3$. Then every nontrivial linear representation
\[\rho:\Mod(S_{g,2})\to \GL_m(\CC)\]
is affine or co-affine. In either case, the isomorphism class is classified by
\[(\mathcal{E}_{\Mod(S_{g,2})}\otimes \CC^2)/\GL_2(\CC)
\cong(\CC^2\otimes\CC^2)/\GL_2(\CC).\]
\end{corollary}
\begin{remark}
This result has the following application. Any nonconstant holomorphic map
\[\mathcal{M}_{g,2}\longrightarrow \mathcal{A}_h,\]
where $\mathcal{M}_{g,2}$ is the moduli space of genus-$g$ curves with two marked points and $\mathcal{A}_h$ is the moduli space of $h$-dimensional principally polarized abelian varieties, with
\[h\le \frac{3g-3}{2},\]
must be the product of the period map
\begin{align*}
    J:\mathcal{M}_{g,2}&\longrightarrow \mathcal{M}_g\longrightarrow \mathcal{A}_g \\
    (X,x_1,x_2)&\longmapsto X \longmapsto \operatorname{Jac}(X)
\end{align*}
with a constant $(h-g)$-dimensional principally polarized abelian variety. This extends Farb’s result \cite{FarbRigidity}, which establishes the uniqueness of the period map $\mathcal{M}_{g,n}\to\mathcal{A}_g$ among nonconstant holomorphic maps $\mathcal{M}_{g,n}\to\mathcal{A}_h$ for $h\le g$. The additional ingredients needed for our application come from the theory of variations of Hodge structures. A complete proof will appear in a forthcoming joint paper with Zhong Zhang. This result is interesting as it shows that the canonical map
\[\mathcal{M}_g^2\hookrightarrow\mathcal{M}_{g+1}
\xrightarrow{J}
\mathcal{A}_{g+1}\]
does not extend to a holomorphic map
\[\mathcal{M}_{g,2}\longrightarrow\mathcal{A}_{g+1}.\]
\end{remark}

\begin{proof}[Proof of Corollary \ref{cor: 2 puncture only affine}]
We give two alternative proofs from different perspectives.

\noindent\textbf{First proof.} The first proof uses Theorem~\ref{thm: classify bi-affine for n punctures} directly. By Remark~\ref{rmk: quotient group also bi-affine}, $\rho$ is bi-affine with core $H$. It remains to rule out the case that $\rho$ is neither affine nor co-affine. By Theorem~\ref{thm: classify bi-affine for n punctures}, such a representation is determined by a pair of nonzero elements 
\[ x=\sum_{k=1}^d (s_1^{(k)}[\kappa_1]+s_2^{(k)}[\kappa_2])\otimes v_k \in \mathcal{E}_{\Mod(S_{g,2})}\otimes\CC^d, \]
and 
\[ y=\sum_{\ell=1}^{m-d-2g} (t_1^{(\ell)}[\kappa_1]+t_2^{(\ell)}[\kappa_2])\otimes w_\ell \in \mathcal{E}_{\Mod(S_{g,2})}\otimes\CC^{m-d-2g}, \]
where $0<d<m-2g$, satisfying 
\[ (s_1^{(k)}+s_2^{(k)})(t_1^{(\ell)}+t_2^{(\ell)})=0, \] 
and 
\[ s_i^{(k)}(t_1^{(\ell)}+t_2^{(\ell)}) +t_i^{(\ell)}(s_1^{(k)}+s_2^{(k)}) +2(g-1)s_i^{(k)}t_i^{(\ell)} =0, \qquad i=1,2. \] 
Suppose first that $s_1^{(k)}+s_2^{(k)}=0$. Since $x$ is nonzero, we have $s_1^{(k)}=-s_2^{(k)}\neq0$. The second relation then implies \[ t_1^{(\ell)}=t_2^{(\ell)} =\frac{1}{2-2g}(t_1^{(\ell)}+t_2^{(\ell)}). \] Consequently, \[ t_1^{(\ell)}=t_2^{(\ell)}=0, \] which contradicts the nontriviality of $y$. The case $t_1^{(\ell)}+t_2^{(\ell)}=0$ is symmetric. 
\medskip 

\noindent\textbf{Second proof.} The second proof does not use the explicit classification in Theorem~\ref{thm: classify bi-affine for n punctures}, but instead uses the fact that boundary twists act trivially after capping. Let 
\[ \widetilde{\rho}:\Mod(S_g^2)\longrightarrow\GL_m(\CC) \] 
be the composition of $\rho$ with the capping map 
\[ \operatorname{Cap}:\Mod(S_g^2)\longrightarrow\Mod(S_{g,2}). \] By Theorem~\ref{thm: any low-dim rep is bi-affine}, $\widetilde{\rho}$ is bi-affine with core $H$. It remains to show that $\widetilde{\rho}$ cannot be neither affine nor co-affine. The Dehn twists $T_{\partial_1}$ and $T_{\partial_2}$ about the boundary components of $S_g^2$ lie in the kernel of the capping map. Hence 
\begin{equation}\label{eq: boundary twist vanish} \widetilde{\rho}(T_{\partial_1}) = \widetilde{\rho}(T_{\partial_2}) = I, \end{equation} 
and therefore 
\[ \widetilde{\rho}(T_{\partial_1}T_{\partial_2}^{-1})=I. \] Assume that $\widetilde{\rho}$ is neither affine nor co-affine. Then by Theorem~\ref{thm: boundary kills}, $\widetilde{\rho}$ is isomorphic to the direct sum of the representation 
\[ \Phi:\Mod(S_g^2)\longrightarrow \GL(H_1(S_{g+1};\CC)), \] induced by the embedding 
\[ S_g^2\hookrightarrow S_{g+1} \]
obtained by gluing a pair of pants to the two boundary components of $S_g^2$, and a trivial representation. However, 
\[ \Phi(T_{\partial_i})\neq I,\qquad i=1,2, \] since each boundary twist acts nontrivially on $H_1(S_{g+1};\CC)$. This contradicts \eqref{eq: boundary twist vanish}. Therefore $\widetilde{\rho}$, and hence $\rho$, must be affine or co-affine. 
\end{proof}

\subsection{Normal functions and biextensions}\label{sec: 2.6}

The bi-affine representations of the mapping class group admit a Hodge-theoretic interpretation in terms of normal functions and biextensions. We now explain this connection.

\subsubsection{General setup}

In general, let $\overline{X}$ be a complex projective manifold, let $D$ be a normal crossings divisor in $\overline{X}$, and let $\mathbb{V}$ be a variation of Hodge structures of weight $-1$ on $X=\overline{X}\setminus D$. The associated \emph{intermediate Jacobian} is
\[J(\mathbb{V})=\mathbb{V}_{\CC}/(F^0\mathbb{V}_{\CC}+\mathbb{V}_{\ZZ}),\]
which is a family of compact complex tori over $X$. Fiberwise, there is a natural isomorphism
\[J(\mathbb{V}_x)\cong \operatorname{Ext}_{\operatorname{MHS}}^1(\ZZ(0),\mathbb{V}_x),\]
so each point of the fiber determines a mixed Hodge structure with weight-graded quotients
\[\operatorname{Gr}^W_0=\ZZ(0),\qquad
\operatorname{Gr}^W_{-1}=\mathbb{V}_x,\qquad
\operatorname{Gr}^W_k=0 \text{ if } k\neq 0,-1.\]

A \emph{normal function} is a section of $J(\mathbb{V})\to X$ such that the corresponding family of mixed Hodge structures forms an \emph{admissible} variation of mixed Hodge structures. Admissibility imposes certain infinitesimal and asymptotic conditions; see \cite[Section~14.4.1]{MHS}. Denote the category of admissible variations of mixed Hodge structures over $X$ by $\operatorname{MHS}(X)$. Then normal functions correspond to elements of
\[\operatorname{Ext}_{\operatorname{MHS}(X)}^1(\ZZ_X(0),\mathbb{V}).\]

Hain \cite[Appendix~A]{HainExact} proved that there is an exact sequence
\begin{equation}\label{eq: exact seq of MHS}
0\to \operatorname{Ext}_{\operatorname{MHS}}^1(\ZZ(0),H^0(X;\mathbb{V}_{\ZZ}))
\to \operatorname{Ext}_{\operatorname{MHS}(X)}^1(\ZZ_X(0),\mathbb{V}_{\ZZ})
\xrightarrow{\delta} H^1(X;\mathbb{V}_{\ZZ}),
\end{equation}
where $\delta$ sends a normal function to its characteristic class; see \cite[Section~4.1]{HainNormalfunction} for three interpretations of this class. We recall the topological interpretation here. Let $\nu\colon X\to J(\mathbb{V})$ be a normal function. For each $x\in X$, let $\widehat{J(\mathbb{V})}_x$ be the space obtained from $J(\mathbb{V})_x$ by attaching an interval whose initial point is identified with $0$ and whose endpoint is identified with $\nu(x)$. As $x$ varies over $X$, the groups $H_1(\widehat{J(\mathbb{V})}_x;\ZZ)$ form a local system fitting into an extension
\[0\to H_1(J(\mathbb{V})_x;\ZZ)
\to H_1(\widehat{J(\mathbb{V})}_x;\ZZ)
\to \ZZ\to 0,\]
whose extension class in $H^1(X;\mathbb{V}_{\ZZ})$ is precisely $\delta(\nu)$.

It follows from the exact sequence \eqref{eq: exact seq of MHS} that if $H^0(X;\mathbb{V}_{\QQ})=0$, then, after tensoring with $\QQ$, a normal function is determined by its characteristic class in $H^1(X;\mathbb{V}_{\QQ})$, while generally we do not know  whether every class in $H^1(X;\mathbb{V}_{\ZZ})$ is realized by an admissible variation of mixed Hodge structures.

One can similarly consider admissible variations of mixed Hodge structures with three nonzero weight-graded pieces; their underlying local systems lead to bi-affine representations. Set
\[\check{\mathbb{V}}=\operatorname{Hom}(\mathbb{V},\ZZ_X(1)),\]
which is also a variation of Hodge structures of weight $-1$. There is a natural principal $\CC^*$-bundle $\mathcal{B}(\mathbb{V})^*$ over $J(\mathbb{V})\times_X J(\check{\mathbb{V}})$. The fiber over $x\in X$ of the composite
\[\mathcal{B}(\mathbb{V})^*\to J(\mathbb{V})\times_X J(\check{\mathbb{V}})\to X\]
is the set of mixed Hodge structures $B$ with weight-graded quotients
\[\operatorname{Gr}^W_0(B)=\ZZ(0),\qquad
\operatorname{Gr}^W_{-1}(B)=\mathbb{V}_x,\qquad
\operatorname{Gr}^W_{-2}(B)=\ZZ(1),\qquad
\operatorname{Gr}^W_k(B)=0 \text{ if } k\neq 0,-1,-2.\]
The projection
\[\mathcal{B}(\mathbb{V})^*\to J(\mathbb{V})\times_X J(\check{\mathbb{V}})\]
takes $B$ to $(B/\ZZ(1),W_{-1}B)$.

A \emph{(Hodge) biextension} is a section of $\mathcal{B}(\mathbb{V})^*\to X$ whose corresponding family of mixed Hodge structures is an admissible variation of mixed Hodge structures with weight-graded quotients $\ZZ_X(0)$, $\mathbb{V}$, and $\ZZ_X(1)$.

\subsubsection{The case of $\mathcal{M}_{g,n}$}
In our setting, we take $X=\mathcal{M}_{g,n}$, the moduli space of smooth projective curves of genus $g$ with $n$ distinct marked points, and consider the variation of Hodge structures $\mathbb{H}$ of weight $-1$ whose fiber at $(C,x_1,\cdots,x_n)$ is $H_1(C;\ZZ)$. The associated intermediate Jacobian bundle
\[J(\mathbb{H})\to \mathcal{M}_{g,n}\]
is exactly the universal Jacobian, which is the pullback of the universal abelian variety
\[\mathcal{X}_g\longrightarrow \mathcal{A}_g\]
along the period map
\[\begin{aligned}
J\colon \mathcal{M}_{g,n}&\longrightarrow\mathcal{A}_g,\\
(C,x_1,\cdots,x_n)&\longmapsto\operatorname{Jac}(C).
\end{aligned}\]
In this case,
\begin{equation}\label{eq: H0 zero}
    H^0(\mathcal{M}_{g,n};\mathbb{H}_{\QQ})
\cong H^0(\Mod(S_{g,n});H_1(S_g;\QQ))=0,
\end{equation}
so the exact sequence \eqref{eq: exact seq of MHS} implies that, up to torsion, a normal function of $J(\mathbb{H})\to\mathcal{M}_{g,n}$ is determined by its characteristic class in
\[H^1(\mathcal{M}_{g,n};\mathbb{H}_{\QQ})
\cong H^1(\Mod(S_{g,n});H_{\QQ})
\cong \QQ^n.\]
Moreover, Hain's classification \cite[Theorem~12.3]{HainTorelliGrp} implies that $\delta$ is surjective after tensoring with $\QQ$. Therefore,
\begin{equation}\label{eq: isom of ext and H1}
\operatorname{Ext}_{\operatorname{MHS}(\mathcal{M}_{g,n})}^1(\QQ(0),\mathbb{H}_{\QQ})
\cong H^1(\mathcal{M}_{g,n};\mathbb{H}_{\QQ})
\cong \QQ^n,
\end{equation}
with an explicit basis given by the normal functions\footnote{The corresponding integral statement is more subtle. Over $\ZZ$, the group of normal functions of $J(\mathbb{H})\to \mathcal{M}_{g,n}$ is torsion-free and is generated by $K_i$ ($1\le i\le n$) together with the normal functions $D_{ij}$ ($1\le i<j\le n$), defined by $D_{ij}(C,x_1,\cdots,x_n)=[x_i-x_j]$, which satisfy
\[(2g-2)D_{ij}=K_i-K_j.\]}
\[\begin{aligned}
K_i\colon \mathcal{M}_{g,n} &\longrightarrow J(\mathbb{H}),\\
(C,x_1,\cdots,x_n)&\longmapsto [(2g-2)x_i-K_C],
\end{aligned}\]
where $1\le i\le n$ and $K_C$ denotes the canonical divisor class of the curve $C$.

At the level of local systems, Kasahara's result \cite[Theorem~1.1]{kasahara} implies that the nontrivial $(2g+1)$-dimensional $\QQ$-local systems over $\mathcal{M}_{g,n}$ are either affine or co-affine with core $H_1(S_g;\QQ)$, and are classified by their extension classes in
\[H^1(\Mod(S_{g,n});H_1(S_g;\QQ))\cong \QQ^n.\]
By Hain's result above, every such affine local system is realized by an admissible variation of mixed Hodge structures on $\mathcal{M}_{g,n}$ with graded quotients $\mathbb{H}_{\QQ}$ and $\QQ(0)$. 

For a biextension of $\mathcal{B}^*(\mathbb{H})\to \mathcal{M}_{g,n}$, forgetting the Hodge structures gives a local system over $\mathcal{M}_{g,n}$ whose monodromy is a bi-affine representation of $\Mod(S_{g,n})$ with a filtration
\[0\subset V_1\subset V_2\subset V\]
such that $V_1\cong V/V_2\cong \ZZ$ and $V_2/V_1\cong H$. We completely classified such local systems over $\CC$ (or over $\QQ$ by the same argument) in Theorem \ref{thm: classify bi-affine for n punctures}. However, this does not by itself give a classification of Hodge biextensions, or equivalently of the corresponding admissible variations of mixed Hodge structures over $\mathcal{M}_{g,n}$.

Indeed, the same underlying local system may support different variations of mixed Hodge structures. Already over a point, one has
\[\operatorname{Ext}^1_{\operatorname{MHS}}(\ZZ(0),\ZZ(1))
\cong \CC/(2\pi i\,\ZZ)
\cong \CC^*.\]
More importantly, projecting a biextension
\[s\colon\mathcal{M}_{g,n}\to \mathcal{B}^*(\mathbb{H})\]
to $J(\mathbb{H})$ and $J(\check{\mathbb{H}})$ gives two normal functions, corresponding to extension classes
\[x_{\mathrm{H}}\in
\operatorname{Ext}^1_{\operatorname{MHS}(\mathcal{M}_{g,n})}
(\ZZ_{\mathcal{M}_{g,n}}(0),\mathbb{H})\]
and
\[y_{\mathrm{H}}\in
\operatorname{Ext}^1_{\operatorname{MHS}(\mathcal{M}_{g,n})}
(\ZZ_{\mathcal{M}_{g,n}}(0),\check{\mathbb{H}})
\cong
\operatorname{Ext}^1_{\operatorname{MHS}(\mathcal{M}_{g,n})}
(\mathbb{H},\ZZ_{\mathcal{M}_{g,n}}(1)).\]
The obstruction to fitting these extensions into a biextension is the vanishing of their Yoneda product
\begin{equation}\label{eq: MHS yoneda product}
y_{\mathrm{H}}\circ x_{\mathrm{H}}=0
\quad\text{in}\quad
\operatorname{Ext}^2_{\operatorname{MHS}(\mathcal{M}_{g,n})}
(\ZZ_{\mathcal{M}_{g,n}}(0),\ZZ_{\mathcal{M}_{g,n}}(1)).
\end{equation}

After forgetting the Hodge structures, the corresponding obstruction for the underlying extensions to fit into a bi-affine local system is the contracted cup product computed in Theorem \ref{thm: classify bi-affine for n punctures}:
\begin{equation}\label{eq: topological yoneda product}
H^1(\mathcal{M}_{g,n};H_{\QQ})\times H^1(\mathcal{M}_{g,n};H_{\QQ})
\longrightarrow H^2(\mathcal{M}_{g,n};\QQ),
\qquad
(x,y)\longmapsto \omega_*(x\cup y).
\end{equation}
Here we used the polarization to identify $\check{\mathbb{H}}_{\QQ}$ with $\mathbb{H}_{\QQ}$. 

From the exact sequence \eqref{eq: exact seq of MHS} and the fact \eqref{eq: H0 zero}, the vanishing of \eqref{eq: topological yoneda product} is necessary for \eqref{eq: MHS yoneda product} to hold. The converse, however, is not clear, and it would be interesting to determine whether this topological obstruction is also sufficient.

Nevertheless, the contracted cup product \eqref{eq: topological yoneda product} is closely related to the Hodge-theoretic setting. The canonical polarization on $\mathbb{H}$
\[S_{\mathbb{H}}\colon\mathbb{H}\otimes\mathbb{H}\longrightarrow \QQ(1),\]
whose underlying pairing of local systems is the intersection pairing $\omega$ on $H$, gives rise to a cohomology class
\[\phi_H\in H^2(J(\mathbb{H});\QQ).\]
The polarization also induces a map
\[i_S\colon J(\mathbb{H})\to J(\check{\mathbb{H}}).\]
Let $\widehat{\mathcal B}$ denote the pullback of the biextension line bundle
\[\mathcal B(\mathbb{H})\to
J(\mathbb{H})\times_{\mathcal{M}_{g,n}}J(\check{\mathbb{H}})\]
along the map
\[(id,i_S)\colon
J(\mathbb{H})\to
J(\mathbb{H})\times_{\mathcal{M}_{g,n}}J(\check{\mathbb{H}}).\]
Hain \cite[Proposition~7.3]{Arakelov} showed that
\[c_1(\widehat{\mathcal B})=2\phi_H.\]
Consequently, for a normal function $\nu\colon\mathcal{M}_{g,n}\to J(\mathbb{H})$, we have
\[c_1(\nu^*\widehat{\mathcal B})=
2(S_{\mathbb{H}})_*\bigl(\delta(\nu)\cup\delta(\nu)\bigr)
\in H^2(\mathcal{M}_{g,n};\QQ).\]
In fact, Hain's results \cite[Theorem~10.2 and Theorem~10.4]{HainNormalfunction} imply
\[
\begin{aligned}
c_1(K_i^*\widehat{\mathcal B})
&=2K_i^*\phi_{S_{\mathbb H}}
=-e_1+4g(g-1)\psi_i,\\
c_1\bigl((K_i,i_SK_j)^*\mathcal B\bigr)
&=-e_1+(2g-2)(\psi_i+\psi_j),\qquad i\ne j,
\end{aligned}
\]
which agree with the formulas \eqref{eq: pair k ii} and \eqref{eq: pair k i j}.

\section{Uniqueness of symplectic representations of symmetric mapping class groups}\label{sec: 3}
For the rest of the paper, let $S=S_g$, let $p$ be a prime, and let
\[\pi:\widetilde{S}\longrightarrow S\]
be the unbranched cyclic cover determined by a nonzero class
$[\beta]\in H_1(S;\ZZ/p\ZZ)$. The groups of interest are
\[
\Mod(S,[\beta])=\operatorname{Stab}_{\Mod(S)}([\beta]),
\qquad
\Mod(\widetilde{S},\sigma)=C_{\Mod(\widetilde{S})}(\sigma),
\]
where $\sigma$ generates the deck group. By Birman--Hilden
\cite{BirmanHilden}, there is a central exact sequence
\begin{equation}\label{eq:BH-central}
1\longrightarrow\langle\sigma\rangle
\longrightarrow \Mod(\widetilde{S},\sigma)
\xrightarrow{\pi_*}\Mod(S,[\beta])
\longrightarrow1.
\end{equation}
Let
\[\Phi:\Mod(S,[\beta])\longrightarrow\Aut(H_1(S;\ZZ),\omega)\cong\Sp_{2g}(\ZZ)\]
be the standard symplectic representation, and put
$\widetilde\Phi=\Phi\circ\pi_*$. 

The main result of this section is the following classification of symplectic representations of $\Mod(S,[\beta])$ and $\Mod(\widetilde{S},\sigma)$ for $p=3$.

\begin{theorem}\label{thm: classify Mod to Sp}
Let $g\geq6$, $h\leq g$, and $p=3$.  Suppose that
\[\rho:\Mod(S,[\beta])\longrightarrow\Sp_{2h}(\ZZ)
\quad\text{or}\quad
\rho:\Mod(\widetilde{S},\sigma)\longrightarrow\Sp_{2h}(\ZZ)\]
is a homomorphism. Then either $\rho$ has finite abelian image, or $h=g$ and $\rho$ is conjugate to $\Phi$ (respectively to $\widetilde\Phi$) by an element of the extended symplectic group
\[\Delta_{2g}(\ZZ)=\{X\in\GL_{2g}(\ZZ):X^tJX=\pm J\},\]
where $J$ is the matrix of the symplectic form.
\end{theorem}

\begin{remark}\label{rmk: GL false}
The restriction on the target group $\Sp_{2h}(\ZZ)$ is essential. Indeed, the Prym representation introduced by Looijenga in \cite{PrymRep} gives a non-abelian linear representation
\[\Mod(\widetilde{S},\sigma)\longrightarrow
U_{2g-2}(\ZZ[\zeta_p])\subset \GL_{2g-2}(\CC),\]
where $\zeta_p$ is a primitive $p$-th root of unity. By twisting this representation by a suitable character of $\Mod(\widetilde{S},\sigma)$, one can arrange that the deck transformation acts trivially. The resulting representation then descends to a non-abelian representation of $\Mod(S,[\beta])$ of dimension $2g-2$.
\end{remark}

\subsection{Preliminaries}
We collect here several preliminary results that will be used in later proofs. We first recall an explicit finite generating set for $\Mod(S,[\beta])$ and $\Mod(\widetilde{S},\sigma)$ from \cite[Theorem 2]{Dey}.

\begin{proposition}\label{prop: gens by Dey}
Let $g\ge3$ and $p\ge2$. Then $\Mod(S,[\beta])$ is finitely generated by the following elements:
\begin{itemize}
    \item the Dehn twists
    \[ T_{a_2},\ldots,T_{a_g}, T_{b_2},\ldots,T_{b_g},T_{c_1},\ldots,T_{c_{g-1}}, \]
    where the curves are indicated in Figure~\ref{fig: basis in surface};
    \item the bounding pair maps
    \[F_2,F_3,\ldots,F_{g-1}\]
    shown in Figure~\ref{fig: F bounding pair maps};
    \item a finite generating set for
    \[\Gamma_1(p)=\left\{
    \begin{pmatrix}
    a&b\\
    c&d
    \end{pmatrix}
    \in\SL_2(\ZZ)
    \mathrel{}\middle|\mathrel{} a,d\equiv1\pmod p,\ b\equiv0\pmod p\right\},\]
    identified with $\Mod(S_1^1,[\beta])$, where $S_1^1$ is a regular neighborhood of $a_1\cup b_1$.
\end{itemize}
Moreover, $\Mod(\widetilde{S},\sigma)$ is generated by lifts of the above generators together with the deck transformation $\sigma$.
\end{proposition}

\begin{figure}
\centering
\includegraphics[width=0.55\linewidth]{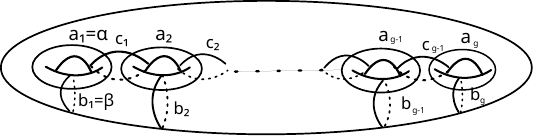}
\caption{\label{fig: basis in surface}The curves used in Proposition~\ref{prop: gens by Dey}.}
\end{figure}

\begin{figure}
\centering
\includegraphics[width=0.55\linewidth]{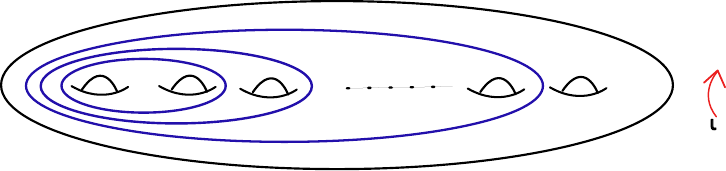}
\caption{\label{fig: F bounding pair maps}Each curve in the figure, together with its image under the hyperelliptic involution $\iota$, determines a bounding pair map $F_k$ for $2\le k \le g-1$.}
\end{figure}

We will also use the following consequence of the proof of \cite[Theorem 2]{Dey}.

\begin{lemma}\label{lem: generate Torelli}
The Torelli subgroup $\mathcal I(S)$ of $\Mod(S,[\beta])$ is finitely generated by a finite collection of bounding pair maps supported in $S\setminus\beta$, together with conjugates of the bounding pair maps
\[F_k,\qquad 2\le k\le g-1,\]
by elements of $\Mod(S,[\beta])$.
\end{lemma}

\begin{proof}
A finite generating set for $\mathcal I(S)$ was given by Johnson \cite{JohnsonTorelliI}; see also \cite[Theorem 3.6]{Dey}. Each bounding pair map in this generating set is either supported in $S\setminus\beta$ or is conjugate in $\Mod(S,[\beta])$ to one of the maps $F_k$, as shown in \cite[Theorem 3.7]{Dey}.
\end{proof}

\begin{remark}
For $p=2,3$, the congruence subgroup $\Gamma_1(p)$ is generated by
\[\begin{pmatrix}
1&0\\
-1&1
\end{pmatrix},
\qquad
\begin{pmatrix}
1&p\\
0&1
\end{pmatrix},\]
which are realized by $T_\beta$ and $T_\alpha^p$, respectively, in $\Mod(S,[\beta])$. This fails for $p\ge4$.
\end{remark}

We next record a relation that will be useful later.

\begin{lemma}\label{lem: alternate 2 chain relation}
Let $a,c$ be simple closed curves in $S$ that intersect each other once, and let $d$ be the boundary of a regular neighborhood of $a\cup c$. Then
\[(T_a^3T_c)^3=T_d.\]
\end{lemma}

\begin{proof}
The $2$-chain relation \cite[Proposition 4.12]{Primer} gives
\[(T_aT_c)^6=T_d.\]
We rewrite this using the braid relation $T_{a} T_{c} T_{a}=T_{c} T_{a} T_{c}$ in \cite[Proposition 3.11]{Primer}:
 \[\begin{aligned}
(T_{a} T_{c})^6&=\big(T_{a} (T_{c} T_{a} T_{c})\big)^3 = (T_{a}^2 T_{c} T_{a})^3 = T_{a}^{-1} (T_{a}^3 T_{c})^3 T_{a}=T_{d}. \end{aligned}\]
Since $d$ is disjoint from $a$, the Dehn twists $T_a$ and $T_d$ commute. Hence
\[(T_a^3T_c)^3=T_aT_dT_a^{-1}=T_d. \qedhere\]
\end{proof}

\begin{remark}\label{rmk: lift to star relation}
For $a=a_1$ and $c=c_1$ in Figure~\ref{fig: basis in surface}, the above $2$-chain relation lifts to the star relation \cite[Theorem 1]{GenerateModNonclosed} in $\widetilde{S}$, as illustrated in Figure~\ref{fig: starr}:

\begin{figure}
\centering
\includegraphics[width=0.275\linewidth]{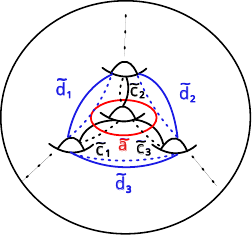}
\caption{\label{fig: starr}The star relation in $\widetilde{S}$.}
\end{figure}

\[(T_{\widetilde a}
 T_{\widetilde c_1}
 T_{\widetilde c_2}
 T_{\widetilde c_3})^3=
T_{\widetilde d_1}
T_{\widetilde d_2}
T_{\widetilde d_3},\]
where $\widetilde c_i$ and $\widetilde d_i$, $1\le i\le3$, are the lifts of $c$ and $d$, respectively, and $\widetilde a$ denotes the lift of $a$.
\end{remark}

We will also use the following result of the abelianization of $\Mod(S,[\beta])$ and $\Mod(\widetilde{S},\sigma)$. For $p=2$, these were obtained by Sato \cite[Theorem 0.2]{sato2}, while for odd primes they were obtained by the author in \cite[Theorem~1.1, Theorem~1.3]{zhong}.

\begin{lemma}\label{lem: abelianization}
Let $g\ge4$ and let $p$ be a prime. Then
\[
H_1(\Mod(\widetilde{S},\sigma);\ZZ)
\cong
\begin{cases}
\ZZ/4\ZZ,& p=2,\\
\ZZ/p\ZZ\oplus\ZZ/p\ZZ,& p\text{ odd},
\end{cases}
\]
and
\[
H_1(\Mod(S,[\beta]);\ZZ)
\cong
\begin{cases}
\ZZ/p\ZZ,
& g\not\equiv1\pmod p,\\
\ZZ/p\ZZ\oplus\ZZ/p\ZZ,
& g\equiv1\pmod p,\ p\text{ odd},\\
\ZZ/4\ZZ,
& g\equiv1\pmod p,\ p=2.
\end{cases}
\]
\end{lemma}

Finally, we recall the following consequence of the change-of-coordinates principle.
\begin{lemma}\label{lem: change curves in cut surface}
Let $\gamma$ and $\gamma'$ be nonseparating simple closed curves on $S_g$. Suppose that either both $\gamma$ and $\gamma'$ intersect $\beta$ once, or both are disjoint from $\beta$, with neither $\gamma$ nor $\gamma'$ isotopic to $\beta$. Then there exists $f\in\Mod(S_g)$ such that
\[f(\beta)=\beta\qquad\text{and}\qquad
f(\gamma)=\gamma'.\]
\end{lemma}

\subsection{Restriction of the representation to $\Mod(S_{g-1}^2)$}

Realize $[\beta]$ by a simple closed curve $\beta$. The cyclic cover
$\widetilde{S}$ can then be constructed by taking $p$ copies of
$S\setminus\beta\cong S_{g-1}^2$ and gluing their boundary components
cyclically. This gives a homomorphism
\[\Mod(S_{g-1}^2)\cong \Mod(S\setminus\beta)
\longrightarrow\Mod(S,[\beta])
\quad \text{(resp.~}\Mod(\widetilde{S},\sigma)\text{)}.\]
The kernel of this map is generated by
\[T_{\partial_1}T_{\partial_2}^{-1},\]
where $\partial_1$ and $\partial_2$ denote the two boundary components of $S_{g-1}^2$.

We denote the composition of $\rho$ with this homomorphism by
\[\rho|_{\Mod(S_{g-1}^2)}:
\Mod(S_{g-1}^2)\longrightarrow \Sp_{2h}(\ZZ).\]
Viewing $\Sp_{2h}(\ZZ)$ as a subgroup of $\GL_{2h}(\CC)$ via the standard inclusion, Theorem~\ref{thm: any low-dim rep is bi-affine} implies that,
for $2h\leq 3(g-1)-3$,
the representation $\rho|_{\Mod(S_{g-1}^2)}$ is either trivial or
bi-affine with core
$H_1(S_{g-1};\CC)$.
Moreover, the bi-affine representations of $\Mod(S_{g-1}^2)$ have been classified in Corollary~\ref{cor: 2g+2 bi-affine rep}. We now analyze these possibilities case by case. Unless otherwise specified, all conjugations below are by elements of $\GL_{2h}(\CC)$.
\begin{lemma}\label{lem: trivial to abelian}
Let $p=2$ or $p=3$. If $\rho|_{\Mod(S_{g-1}^2)}$ is trivial, then the image of $\rho$ is abelian. 
\end{lemma}
\begin{proof}
By the generating set in Proposition~\ref{prop: gens by Dey}, it suffices to consider the images of the remaining generators
\[T_{\alpha}^p,\ F_2,\ldots,F_{g-1},\ \sigma .\]
Indeed, all other generators lie in the subgroup $\Mod(S_{g-1}^2)$ and hence have trivial image under $\rho$ by assumption. Observe that these elements pairwise commute, therefore the image of $\rho$ is abelian.
\end{proof}

\begin{lemma}\label{lem: affine std rep}
    If $\rho|_{\Mod(S_{g-1}^2)}$ is affine or co-affine with core $H_1(S_{g-1};\CC)$, then
    $\rho|_{\Mod(S_{g-1}^2)}$ is conjugate to the direct sum of the standard action on $H_1(S_{g-1};\ZZ)$ and a trivial representation.
\end{lemma}
\begin{proof}
We only treat the co-affine case, since the affine case is dual. Since $h\le g$, the only nontrivial case is $h=g$. Then there exists $A\in \GL_{2g}(\CC)$ such that
\[\rho|_{\Mod(S_{g-1}^2)}(f)=
A
\begin{pmatrix}
\Psi(f)&\sigma_1(f)&\sigma_2(f)\\
0&1&0\\
0&0&1
\end{pmatrix}
A^{-1},
\]
where
\[
\Psi:\Mod(S_{g-1}^2)\longrightarrow \Sp_{2g-2}(\ZZ)
\]
is the standard action on $H_1(S_{g-1};\ZZ)$ and
\[[\sigma_i]=s_1^{(i)}[\kappa_1]+s_2^{(i)}[\kappa_2]\in H^1(\Mod(S_{g-1}^2);H_1(S_{g-1};\CC)),\qquad i=1,2.\]

We claim that $[\sigma_1]$ and $[\sigma_2]$ both vanish. Suppose, for example, that $[\sigma_1]\neq0$. Assume without loss of generality that $s_1^{(1)}\neq0$. Let $\pi_1(S_{g-1})\subset \Mod(S_{g-1}^2)$ denote a lift of the point-pushing subgroup of $\Mod(S_{g-1,1})$ under the map
\[\Mod(S_{g-1}^2)\longrightarrow \Mod(S_{g-1,1})\]
induced by capping the first boundary component of $S_{g-1}^2$ by a punctured disk and capping the
second one by a disk. By definition 
\[\Psi(\gamma)=I,\qquad \forall \gamma\in\pi_1(S_{g-1}).\]
Moreover, since
\[\sigma_1-s_1^{(1)}\kappa_1-s_2^{(1)}\kappa_2\]
is a principal crossed homomorphism, Proposition~\ref{prop: image of point pushing}
implies that
\[\sigma_1(\gamma)=s_1^{(1)}(2-2g)[\gamma]\in H_1(S_{g-1};\CC),
\qquad \forall\gamma\in\pi_1(S_{g-1}).\]

Let $\omega$ be the symplectic form on $\ZZ^{2g}$ preserved by $\rho$, and define the
non-degenerate symplectic form
\[\omega_A(x,y)=\omega(Ax,Ay)\]
on $A^{-1}\ZZ^{2g}$. This form is preserved by 
\[A^{-1}\rho|_{\Mod(S_{g-1}^2)}(f)A, \qquad \forall f\in \Mod(S_{g-1}^2).\]
Taking $f=\gamma\in\pi_1(S_{g-1})$, the preservation of $\omega_A$
implies that
\[\omega_A(x,s_1^{(1)}(2-2g)[\gamma])=0,
\qquad \forall x\in H_1(S_{g-1};\ZZ),\ \forall \gamma\in\pi_1(S_{g-1}).\]
Since $s_1^{(1)}\neq 0$, this implies that the subspace
$H_1(S_{g-1};\ZZ)$ is isotropic with respect to $\omega_A$.
However, a symplectic vector space of dimension $2g$ has maximal
isotropic dimension $g$, and $2g-2>g$ for $g\ge3$, which is a contradiction.

Therefore $[\sigma_1]=0$. The same argument shows that
$[\sigma_2]=0$. Consequently $\rho|_{\Mod(S_{g-1}^2)}$
is conjugate to the direct sum of the standard representation $\Psi$ and a
two-dimensional trivial representation.
\end{proof}
    
\begin{lemma}\label{lem: final rep}
     If $\rho|_{\Mod(S_{g-1}^2)}$ is nontrivial, and is neither affine nor co-affine, then $\rho|_{\Mod(S_{g-1}^2)}$ is conjugate to the standard action on $H_1(S_g;\ZZ)$.
\end{lemma}
\begin{proof}
    Since the boundary twists satisfy \[\rho(T_{\partial_1}T_{\partial_2}^{-1})=I,\]
   the result follows directly from Theorem \ref{thm: boundary kills}.
\end{proof}

\subsection{Obstruction to the symplectic representation on $H_1(S_{g-1};\ZZ)$}
We now show that the case in Lemma \ref{lem: affine std rep} is obstructed by the structure of the symplectic group.
\begin{proposition}\label{prop: std sym rep fails}
Let $g\ge5$, $h\le g$ and $p=3$. Then 
\[\rho|_{\Mod(S_{g-1}^2)}:
\Mod(S_{g-1}^2)\longrightarrow \Sp_{2h}(\ZZ)\]
cannot be conjugate to the direct sum of the standard action on
$H_1(S_{g-1};\ZZ)$ and a trivial representation.
\end{proposition}
We first establish the following obstruction lemma.
\begin{lemma}\label{lem: obstruction by scalar}
    Let $p=3$. There is no homomorphism
    \[
    \rho:\Mod(S,[\beta])\longrightarrow\GL_{2g-2}(\CC)
    \quad\text{or}\quad
    \rho:\Mod(\widetilde{S},\sigma)\longrightarrow\GL_{2g-2}(\CC)
    \]
    satisfying the following conditions:
    \begin{enumerate}
        \item The restriction $\rho|_{\Mod(S_{g-1}^2)}$ is conjugate to the standard action on
        $H_1(S_{g-1};\ZZ)$.
        \item $\rho(F_k)$ (or the image of a chosen lift of $F_k$ in
        $\Mod(\widetilde{S},\sigma)$) is a scalar for $2\leq k\leq g-1$, where
        $F_k$ are the bounding pair maps in Figure \ref{fig: F bounding pair maps}.
    \end{enumerate}
\end{lemma}
\begin{proof}
    Suppose that such a homomorphism $\rho$ exists. We first observe that the image of
    $T_\alpha^3\in\Mod(S,[\beta])$ (or its lift to $\Mod(\widetilde{S},\sigma)$) is a scalar.
    Indeed, $T_\alpha^3$ centralizes the subgroup
    \[
    \Mod(S_{g-1}^1)\subset \Mod(S_{g-1}^2),
    \]
    and by the first assumption the image of this subgroup under $\rho$ is conjugate to
    \[
    \Sp_{2g-2}(\ZZ)\subset \GL_{2g-2}(\CC).
    \]
    Hence $\rho(T_\alpha^3)$ lies in the centralizer of
    $\Sp_{2g-2}(\ZZ)$ in $\GL_{2g-2}(\CC)$, which consists only of scalar matrices.

    We first consider the case of $\Mod(S,[\beta])$. Let $\alpha$ and $c_1$ be the
    simple closed curves intersecting once as in Figure \ref{fig: basis in surface}.
    The rewritten two-chain relation from Lemma \ref{lem: alternate 2 chain relation}
    gives
    \begin{equation}\label{eq: 2 chain a c}
        (T_\alpha^3T_{c_1})^3=T_d,
    \end{equation}
    where $d$ is the boundary component of a regular neighborhood of
    $\alpha\cup c_1$. Since $d$ is a separating simple closed curve in $S$, we have
    \[
    T_d\in\mathcal I(S).
    \]
    By Lemma \ref{lem: generate Torelli}, the element $T_d$ can be expressed as a
    product of bounding pair maps supported in $S\setminus\{\beta\}$ and conjugates of
    the bounding pair maps $F_k$. Any bounding pair map supported in $S\setminus\{\beta\}$ maps to the identity under $\rho$ by the
    first assumption, and conjugates of
    the bounding pair maps $F_k$ map to scalars by the second assumption.
    Therefore $\rho(T_d)$ is a scalar matrix.

    On the other hand, since by the first assumption
    \[\rho(T_{c_1})=\rho(T_{b_2})\]
    is a nontrivial unipotent transvection, applying $\rho$ to
    \eqref{eq: 2 chain a c} gives a contradiction: the left-hand side is not scalar,
    while the right-hand side is.

    It remains to consider the case of $\Mod(\widetilde{S},\sigma)$. Pulling back the
    two-chain relation to the cover gives the corresponding star relation described in
    Remark \ref{rmk: lift to star relation}. Although the lift of the separating curve
    $d$ is a union of three non-separating simple closed curves in $\widetilde{S}$, the
    same argument still applies. Indeed, any two lifts of an element of
    $\Mod(S,[\beta])$ differ by a power of the deck transformation $\sigma$. Since
    $\sigma$ centralizes the image of $\Mod(S_{g-1}^1)$, its image under $\rho$ is a
    scalar. Therefore, the image of any lift of $T_d$ is scalar, and the star relation yields a contradiction.
\end{proof}

We now apply this obstruction to prove Proposition
\ref{prop: std sym rep fails}.
\begin{proof}[Proof of Proposition \ref{prop: std sym rep fails}]
Suppose for contradiction that $\rho|_{\Mod(S_{g-1}^2)}$ 
is conjugate to the direct sum of the standard action on
$H_1(S_{g-1};\ZZ)$ and a trivial representation. We derive a contradiction using Lemma \ref{lem: obstruction by scalar}.

Since $h\le g$, we consider the cases $h=g-1$ and $h=g$ separately.

\noindent\textbf{(1) The case $h=g-1$.} By assumption
\[\rho|_{\Mod(S_{g-1}^2)}:\Mod(S_{g-1}^2)\to \Sp_{2g-2}(\ZZ)\] is conjugate to the standard action on $H_1(S_{g-1};\ZZ)$ by some $A\in \GL_{2g-2}(\CC)$. By Lemma \ref{lem: obstruction by scalar}, it suffices to show that $\rho(F_k)$ is a scalar for every $2\le k\le g-1$. The same argument applies to the chosen lifts of $F_k$ in $\Mod(\widetilde{S},\sigma)$, since $\rho(\sigma)$ centralizes \[\operatorname{Im}(\rho|_{\Mod(S_{g-1}^2)})=A^{-1}\cdot \Sp_{2g-2}(\ZZ)\cdot A\] and is therefore a scalar.

We first prove that
\[\rho(F_k)^3=\rho(T_{d_k}^3 T_{d'_k}^{-3})=I.\]
Since both $d_k$ and $d_k'$ have geometric intersection number $1$ with $\beta$, by Lemma \ref{lem: change curves in cut surface} there exist
$f,f'\in\Mod(S,[\beta])$ satisfying
\[f(\alpha)=d_k,\qquad f'(\alpha)=d_k'.\]
Since $\rho(T_\alpha^3)$ centralizes the image of
$\Mod(S_{g-1}^1)$, namely $\Sp_{2g-2}(\ZZ)$, it is a scalar. Therefore,
\[\rho(T_{d_k}^3)=\rho(f T_{\alpha}^3 f^{-1})=\rho(f)\rho(T_{\alpha}^3) \rho(f)^{-1}=\rho(T_{\alpha}^3),\]
and similarly $\rho(T_{d'_k}^3)=\rho(T_{\alpha}^3)$. Hence
\[\rho(F_k)^3=\rho(T_{d_k}^3)\rho(T_{d'_k}^3)^{-1}=\rho(T_{\alpha}^3)\rho(T_{\alpha}^3)^{-1}=I.\]

\begin{figure}
\centering\includegraphics[width=0.65\linewidth]{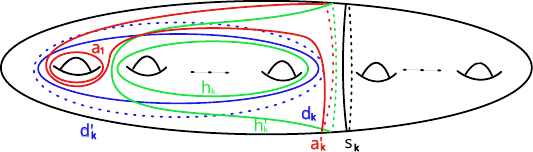}\caption{Lantern relation associated with $F_k$.\label{fig: F lantern relation}}\end{figure}

Now consider the lantern relation
(\cite[Proposition~5.1]{Primer}) illustrated in Figure
\ref{fig: F lantern relation},
\begin{equation}\label{eq: F k lantern}
    T_{a_1}T_{h_k}T_{d_k}T_{s_k}=T_{a'_k}T_{h'_k}T_{d'_k},
\end{equation}
which gives
\[F_k=T_{d_k}T_{d'_k}^{-1}=T_{s_k}^{-1}\cdot(T_{h'_k}T_{h_k}^{-1})\cdot (T_{a'_k}T_{a_1}^{-1}).\]
Since $\rho|_{\Mod(S_{g-1}^2)}$ is conjugate to the standard action on $H_1(S_{g-1};\ZZ)$, we have
\[\rho(T_{s_k})=I,
\qquad\rho(T_{h'_k}T_{h_k}^{-1})=I.\]
Consequently,
\[\rho(F_k)=\rho(T_{a'_k}T_{a_1}^{-1}).\]

We first consider the case $k=2$. Observe that both $a_1$ and $a_2'$ are disjoint from the curves 
\[a_2,b_2\]
and 
\[a_3,b_3,c_3,a_4,b_4,c_4,\cdots,c_{g-1},a_g,b_g,\]
shown in Figure \ref{fig: basis in surface}. Hence
$T_{a_2'}T_{a_1}^{-1}$ commutes with the Dehn twists about all of these curves. It follows that
\[\rho(F_2)=\rho(T_{a'_2}T_{a_1}^{-1})=A^{-1}\cdot \operatorname{diag}(\lambda_1,\lambda_1,\lambda_2,\cdots,\lambda_2)\cdot A.\]
Since $\rho(F_2)^3=I$ and $\rho(F_2)\in \Sp_{2g-2}(\ZZ)$, every eigenvalue of $\rho(F_2)$ is a power of $\zeta_3=e^{\frac{2\pi i}{3}}$, and $\zeta_3$ and $\zeta_3^2$ occur with the same multiplicity. Since $g\ge 4$, this forces
\[\lambda_1=\lambda_2=1,\]
and hence
\[\rho(F_2)=\rho(T_{a'_2}T_{a_1}^{-1})=I.\]

\begin{figure}
\centering
\includegraphics[width=0.275\linewidth]{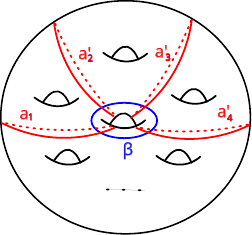}
\caption{Rotation.\label{fig: rotate}}
\end{figure}

Now let $k>2$. Observe that
\[\rho(F_k)=\rho(T_{a'_k}T_{a_1}^{-1})=\rho(T_{a'_k}T_{a'_{k-1}}^{-1})\rho(T_{a'_{k-1}}T_{a'_{k-2}}^{-1}) \cdots\rho(T_{a'_{2}}T_{a_1}^{-1}). \]
Furthermore, Figure \ref{fig: rotate} shows that there exists
$f_k\in\Mod(S,[\beta])$ by rotation that satisfies
\[f_k(a_1)=a_{k-1}',\qquad f_k(a_2')=a'_{k},\qquad f(\beta)=\beta.\]
Hence 
\[T_{a_k'}T_{a_{k-1}'}^{-1}=f_k(T_{a_2'}T_{a_1}^{-1})f_k^{-1},\]
so
\[\rho(T_{a_k'}T_{a_{k-1}'}^{-1})=\rho(f_k)\rho(T_{a_2'}T_{a_1}^{-1})\rho(f_k)^{-1}=I.\]
Therefore $\rho(F_k)=I$ for every $2\le k\le g-1$, contradicting
Lemma \ref{lem: obstruction by scalar}.

\noindent\textbf{(2) The case $h=g$.} By assumption, $\rho|_{\Mod(S_{g-1}^2)}$ is conjugate to the direct sum of the
standard action on $H_1(S_{g-1};\ZZ)$ and a trivial $2$-dimensional
representation. By applying the same lantern relation \eqref{eq: F k lantern} as in (1), we obtain that, for each $2\le k \le g-1$,
\[\rho(F_k)=\rho(T_{a'_k}T_{a_1}^{-1}),\]
which commutes with the images under $\rho$ of the Dehn twists about the curves
\[a_2,b_2,a_3,b_3,\cdots,a_g,b_g.\]
It follows that $\rho(F_k)$ preserves the subspace $H_1(S_{g-1};\CC)$, which is spanned by 
$\operatorname{Im}(\rho(T_{\gamma})-I)$ as $\gamma$ ranges over these curves. The same holds for $\rho(T_{\alpha}^3)$ and $\rho(\sigma)$. Consequently, $\rho$ induces a representation
\[\overline{\rho}:\Mod(S,[\beta])\longrightarrow\GL_{2g-2}(\CC)\]
(and similarly for $\Mod(\widetilde{S},\sigma)$). We emphasize that the image of $\overline{\rho}$ is not necessarily contained in
$\Sp_{2g-2}(\ZZ)$. Instead, the image of the original representation $\rho$ which is a two-dimensional extension of $\overline{\rho}$ lies inside $\Sp_{2g}(\ZZ)$. Therefore, some additional care is required when analyzing the eigenvalues of the images of $F_k$.

Our strategy is again to prove that
\[\overline{\rho}(F_k)=I,\qquad 2\le k\le g-1,\]
and then obtain a contradiction from Lemma \ref{lem: obstruction by scalar}. The same argument as above gives
\[\overline{\rho}(F_k)^3=I,\quad 2\le k \le g-1,\]
and that 
\[\overline{\rho}(F_2)=\overline{\rho}(T_{a'_{2}}T_{a_1}^{-1})=\operatorname{diag}(\lambda_1,\lambda_1,\lambda_2,\cdots,\lambda_2).\]
Since $g\ge5$ and a two-dimensional extension of
$\overline{\rho}(F_2)$ lies in $\Sp_{2g}(\ZZ)$, we have
\[\lambda_2=1,\qquad\lambda_1=\zeta_3^t\]
for some $t\in\{0,1,2\}$.

We now show that $t=0$. Consider 
\[\overline{\rho}(F_3)=\overline{\rho}(T_{a'_3}T_{a_1}^{-1})=\overline{\rho}(T_{a'_3}T_{a'_{2}}^{-1})\overline{\rho}(T_{a'_{2}}T_{a_1}^{-1}). \]
The curves $a_3'$ and $a_2'$ are disjoint from the curves
\[a_2,b_2,a_3,b_3,\]
and 
\[a_4,b_4,c_4,\cdots,c_{g-1},a_g,b_g.\]
Therefore $\overline{\rho}(T_{a_3'}T_{a_2'}^{-1})$ commutes with the images of the corresponding Dehn twists, and hence 
\[
\overline{\rho}(T_{a_3'}T_{a_2'}^{-1})=\operatorname{diag}
(\mu_1,\mu_1,\mu_2,\mu_2,\mu_3,\ldots,\mu_3).
\]
Since $\overline{\rho}(T_{a_3'}T_{a_2'}^{-1})$ is conjugate to $\overline{\rho}(F_2)=\overline{\rho}(T_{a'_{2}}T_{a_1}^{-1})$, the eigenvalues of this
matrix agree with those of $\overline{\rho}(F_2)$ with the same
multiplicities. Hence (since $g\ge 5$)
\[\mu_3=1,\]
and either
\[\mu_1=\lambda_1,\qquad \mu_2=1,\]
or
\[\mu_1=1,\qquad \mu_2=\lambda_1.\]
Moreover, since
\[\overline{\rho}(F_3)=\overline{\rho}(T_{a'_3}T_{a'_{2}}^{-1})\overline{\rho}(T_{a'_{2}}T_{a_1}^{-1})=\operatorname{diag}
(\zeta_3^t\mu_1,\zeta_3^t\mu_1,\mu_2,\mu_2,1,\ldots,1). \]
commutes with $\rho(T_{c_2})$, we obtain
\[\zeta_3^t\mu_1=\mu_2,\]
and hence
\[\mu_1=1,\qquad \mu_2=\lambda_1=\zeta_3^t.\]
It follows that $\overline{\rho}(F_3)$ has eigenvalue $\zeta_3^t$ with
multiplicity $4$ and eigenvalue $1$ with multiplicity $2g-6$. Since a two-dimensional extension of
$\overline{\rho}(F_3)$ lies in $\Sp_{2g}(\ZZ)$, this forces $t=0$, and hence
\[
\overline{\rho}(F_2)=I.
\]
Applying the same argument as in the case $h=g-1$, we obtain
\[\overline{\rho}(F_k)=I,\qquad 2\le k\le g-1,\]
contradicting Lemma \ref{lem: obstruction by scalar}.
\end{proof}

\subsection{Finishing the proof of Theorem~\ref{thm: classify Mod to Sp}}
Combining the results from the previous sections, we now finish the proof of
Theorem~\ref{thm: classify Mod to Sp}. Consider a homomorphism
\[
\rho:\Mod(S,[\beta])\longrightarrow\Sp_{2h}(\ZZ)
\quad\text{or}\quad
\rho:\Mod(\widetilde{S},\sigma)\longrightarrow\Sp_{2h}(\ZZ),
\]
with $g\ge 6$, $h\le g$ and $p=3$. We have analyzed the possible restrictions of $\rho$ to
$\Mod(S_{g-1}^2)$:
\begin{enumerate}
    \item If $\rho|_{\Mod(S_{g-1}^2)}$ is trivial, then the image of $\rho$ is abelian by Lemma \ref{lem: trivial to abelian}.
    \item If $\rho|_{\Mod(S_{g-1}^2)}$ is affine or co-affine, then by Lemma \ref{lem: affine std rep} it is conjugate to the direct sum of the standard action on $H_1(S_{g-1};\ZZ)$ and a trivial representation. This possibility is ruled out by Proposition \ref{prop: std sym rep fails}.
    \item If
    $\rho|_{\Mod(S_{g-1}^2)}$ is nontrivial and neither affine nor co-affine, by Lemma \ref{lem: final rep}, it is conjugate to the standard action on $H_1(S_g;\ZZ)$. 
\end{enumerate}
In the last case, it remains to prove that $\rho$ itself is conjugate to the standard action on $H_1(S_g;\ZZ)$.
By the generating set of $\Mod(S,[\beta])$ and
$\Mod(\widetilde{S},\sigma)$ given in Proposition \ref{prop: gens by Dey}, it suffices to determine the images of
\[T_\alpha^3,\qquad F_k\ (2\le k\le g-1),\qquad \text{and }\sigma.\]
We also need to control the conjugating matrix, reducing it from an arbitrary
element of $\GL_{2g}(\CC)$ to the extended symplectic group. We discuss these steps in order.

\subsubsection{Determine the image of $T_{\alpha}^3$}
\begin{proposition}\label{prop: image of T a cubic}
Suppose that
\[\rho|_{\Mod(S_{g-1}^2)}\colon\Mod(S_{g-1}^2)\to\Sp_{2g}(\ZZ)\]
is conjugate to the standard action on $H_1(S_g;\ZZ)$. Then there is a conjugation by an element in $\GL_{2g}(\CC)$ such that both $\rho|_{\Mod(S_{g-1}^2)}$ and 
$\rho(T_\alpha^3)$ (or the image of its lift to $\Mod(\widetilde{S},\sigma)$) are the standard actions on $H_1(S_g;\ZZ)$.
\end{proposition}

\begin{proof}
    We discuss $T_\alpha^3\in\Mod(S,[\beta])$. Everything below also applies to a chosen lift in $\Mod(\widetilde{S},\sigma)$, since all the relations we use have corresponding lifts in $\Mod(\widetilde{S},\sigma)$.

    Write everything in matrix form with respect to the standard symplectic basis
    \[[a_1]=[\alpha],[b_1]=[\beta],[a_2],[b_2],\ldots,[a_g],[b_g]\]
    of $H_1(S_g;\ZZ)$ as in Figure \ref{fig: basis in surface}, with $\rho|_{\Mod(S_{g-1}^2)}$ acting by the standard representation.
    
    Since $T_\alpha^3$ centralizes the subgroup $\Mod(S_{g-1}^1) \subset \Mod(S_{g-1}^2)$, its image $\rho(T_\alpha^3)$ centralizes \[\rho(\Mod(S_{g-1}^1))=\Sp_{2g-2}(\ZZ),\] which implies that
    \[\rho(T_\alpha^3)=\begin{pmatrix}
    A & \textbf{0} \\ \textbf{0} & \lambda I_{2g-2}
    \end{pmatrix},\qquad \text{ for some }\lambda\in\CC, \text{ and }A\in \GL_2(\CC).\]
    First, the rewritten two-chain relation (Lemma \ref{lem: alternate 2 chain relation}) for $\alpha$ and $\beta$ gives
    \begin{equation}\label{eq: chain relation 1}
        (T_{\alpha}^3 T_{\beta})^3=T_d,
    \end{equation}
    where $d$ is the boundary of a regular neighborhood of $\alpha\cup \beta$. Since $d$ is a separating simple closed curve in the subsurface $S_{g-1}^2$, we have $\rho(T_d)=I$. Applying $\rho$ to this relation gives
    \begin{equation}\label{eq: A relation 1}
        \left(A\cdot \begin{pmatrix}
    1 & 0\\-1& 1
    \end{pmatrix} \right)^3=I,
    \end{equation}
    and $\lambda^3=1$. Since $\rho(T_\alpha^3)$ is conjugate to a symplectic matrix, we have $\lambda=1$ and $A\in \SL_2(\CC)$. It remains to determine $A$, for which we use one more relation.
    
    Let $\eta$ be a simple closed curve in the subsurface
    $S_{g-1}^2=S\setminus\{\beta\}$ that is disjoint from $b_2$ and satisfies
    \[[\eta]=[\beta]+[b_2]\in H_1(S_g;\ZZ).\]
    Set
    \[\alpha'=T_\eta T_\beta^{-1}T_{b_2}^{-1}(\alpha).\]
    Then $\alpha'$ is a simple closed curve that intersects $\beta$ once, is disjoint from $\alpha$, and satisfies
    \[[\alpha']=[\alpha]-[b_2]\in H_1(S_g;\ZZ).\]
    Now set
    \begin{equation}\label{eq: chain relation 2}
        \beta'=T_{\alpha'}^3T_\alpha^{-3}(\beta),\qquad d'=T_{\alpha'}^3T_\alpha^{-3}(d).
    \end{equation}
   Then the rewritten two-chain relation (Lemma \ref{lem: alternate 2 chain relation}) for $\alpha$ and $\beta'$ gives
   \[(T_\alpha^3T_{\beta'})^3=T_{d'},\]
   which is precisely the conjugate of the relation \eqref{eq: chain relation 1} by $T_{\alpha'}^3T_\alpha^{-3}$. At first sight, this conjugation seems to give no additional information, but in fact the additional constraint comes from a more careful analysis of $\rho(T_{\beta'})$. From our assumption on $\rho|_{\Mod(S_{g-1}^2)}$, we have
   \[U=\rho(T_\eta T_\beta^{-1}T_{b_2}^{-1})=\begin{pmatrix}
   1 & 0 & 0 & 0 & \mathbf{0}\\
   0 & 1 & -1 & 0 & \mathbf{0} \\
   0 & 0 & 1 & 0 & \mathbf{0} \\
   -1 &  0 & 0 & 1 & \mathbf{0} \\
   \mathbf{0} & \mathbf{0} & \mathbf{0} & \mathbf{0} & I_{2g-4}
   \end{pmatrix},\]
   then 
   \begin{align*}
       \rho(T_{\beta'})&=\rho(T_{T_{\alpha'}^3T_\alpha^{-3} (\beta)}) \\
       & =\rho(T_{\alpha'}^3T_\alpha^{-3}) \rho(T_{\beta})\rho(T_{\alpha'}^3T_\alpha^{-3})^{-1} \\
       &=U\rho(T_{\alpha}^3)U^{-1}\rho(T_{\alpha}^3)^{-1} \rho(T_{\beta})\rho(T_{\alpha}^3)U\rho(T_{\alpha}^3)^{-1}U^{-1}.
   \end{align*}
   Write 
   \[A=\begin{pmatrix}
   s & r \\ t & q
   \end{pmatrix}, \qquad sq-tr=1,\]
   then a direct computation gives
   \[\rho(T_{\beta'})=\begin{pmatrix}
   1 & 0 & 0 & 0 & \mathbf{0}\\
   -1 & 1 & r & 0 & \mathbf{0} \\
   0 & 0 & 1 & 0 & \mathbf{0} \\
   r &  0 & -r^2 & 1 & \mathbf{0} \\
   \mathbf{0} & \mathbf{0} & \mathbf{0} & \mathbf{0} & I_{2g-4}
   \end{pmatrix}.\]
   Then applying $\rho$ to the relation \eqref{eq: chain relation 2} gives
   \begin{equation}\label{eq: A relation 2}
       \left(\begin{pmatrix}
       A & \mathbf{0} & \mathbf{0} \\
       \mathbf{0} &1 & 0 \\
       \mathbf{0} & 0 & 1
       \end{pmatrix}\cdot \begin{pmatrix}
   1 & 0 & 0 & 0 \\
   -1 & 1 & r & 0  \\
   0 & 0 & 1 & 0  \\
   r &  0 & -r^2 & 1 
   \end{pmatrix} \right)^3=I,
   \end{equation}
   where we used $\rho(T_{d'})=I$ since it is conjugate to $\rho(T_d)=I$.
   We now solve for $r$ and $A$ using the relations \eqref{eq: A relation 1} and \eqref{eq: A relation 2}.
   
  \noindent\textbf{(1) Exclude $r=0$.} First, we show that $r\neq 0$. Suppose, for contradiction, that $r=0$. Then
  \[A\cdot \begin{pmatrix}1&0\\ -1&1\end{pmatrix}\]
  is a lower triangular matrix in $\SL_2(\CC)$ of order $3$. Thus, either
  \[A\cdot \begin{pmatrix}1&0\\ -1&1\end{pmatrix}=I_2,\qquad A=\begin{pmatrix}1&0\\1&1\end{pmatrix},\]
  or
  \[A\cdot \begin{pmatrix}1&0\\-1&1\end{pmatrix}\]
  has distinct eigenvalues $\zeta_3$ and $\zeta_3^2$, in which case the diagonal entries of $A$ are $\zeta_3$ and $\zeta_3^2$.

  In either case, the images under $\rho$ of the bounding pair maps $F_k$ ($2\le k\le g-1$) (and of $\sigma$) preserve $\CC[b_1]$, since $F_k$ (and $\sigma$) commute with $T_\alpha^3$, and $\CC[b_1]$ is either
  \[\operatorname{Im}\bigl(\rho(T_\alpha^3)-I\bigr)\]
  or the $\zeta_3^t$-eigenspace ($t=1$ or $2$) of $\rho(T_\alpha^3)$. 

  Now consider the generating set of $\Mod(S,[\beta])$ (resp.~ $\Mod(\widetilde{S},\sigma)$) given in Proposition \ref{prop: gens by Dey}. The image of every generator under $\rho$ preserves $\CC[b_1]$, and hence $\rho$ induces a representation (and resp.~ of $\Mod(\widetilde{S},\sigma)$)
  \[\widehat{\rho}\colon \Mod(S,[\beta])\longrightarrow\GL\bigl(\CC^{2g}/\CC[b_1]\bigr)=\GL_{2g-1}(\CC).\]
  By definition, we have
   \[\widehat{\rho}(T_\alpha^3)=\begin{pmatrix}
\zeta_3^t&0&0&\mathbf 0\\
0&1&0&\mathbf 0\\
0&0&1&\mathbf 0\\
\mathbf 0&\mathbf 0&\mathbf 0&I_{2g-4}
\end{pmatrix},
\qquad
\widehat{\rho}(T_{c_1})=
\begin{pmatrix}
1&0&0&\mathbf 0\\
0&1&0&\mathbf 0\\
-1&1&1&\mathbf 0\\
\mathbf 0&\mathbf 0&\mathbf 0&I_{2g-4}
\end{pmatrix},\]
where $c_1$ is the simple closed curve shown in Figure \ref{fig: basis in surface}, and the matrices are expressed with respect to the basis
\[[a_1],[a_2],[b_2],\ldots,[a_g],[b_g].\]
Here $t\in\{0,1,2\}$. We first show that $t=0$. Suppose, to the contrary, that $t\neq0$. Consider $\widehat{\rho}(T_{d_2}^3)$, where $d_2$ is the curve shown in Figure \ref{fig: F lantern relation}. Since $T_{d_2}^3$ commutes with the Dehn twists about the curves
\[a_3,b_3,c_3,a_4,b_4,c_4,\ldots,c_{g-1},a_g,b_g,\]
we obtain that $\widehat{\rho}(T_{d_2}^3)$ has the form
\[\widehat{\rho}(T_{d_2}^3)=
\begin{pmatrix}
D_{3\times3}&\mathbf 0\\
\mathbf 0&\mu I_{2g-4}\end{pmatrix}.\]
Moreover, since $T_{d_2}^3$ commutes with $T_{\alpha}^3$, $\widehat{\rho}(T_{d_2}^3)$ preserves both the $\zeta_3^t$-eigenspace and the $1$-eigenspace of $\widehat{\rho}(T_\alpha^3)$. Since $\widehat{\rho}(T_{d_2}^3)$ also commutes with $\widehat{\rho}(T_{a_2})$, the matrix $D$ is of the form
\[D=
\begin{pmatrix}
\lambda'&0&0\\
0&\mu'&x\\
0&0&\mu'
\end{pmatrix}.\]
Finally, since $\widehat{\rho}(T_{d_2}^3)$ commutes with $\widehat{\rho}(T_{c_1})$, a direct computation then gives
\[x=0,\qquad \lambda'=\mu'.\]
Thus every eigenvalue of $\widehat{\rho}(T_{d_2}^3)$ has multiplicity at least $2$. On the other hand, by Lemma \ref{lem: change curves in cut surface}, there exists $f\in\Mod(S,[\beta])$ such that $f(\alpha)=d_2$. Hence $\widehat{\rho}(T_{d_2}^3)$ is conjugate to $\widehat{\rho}(T_\alpha^3)$ and therefore has the same eigenvalues. Since $t\neq0$, the eigenvalue $\zeta_3^t$ of $\widehat{\rho}(T_\alpha^3)$ has multiplicity $1$, giving a contradiction.

Therefore $t=0$, and hence $\widehat{\rho}(T_\alpha^3)=I$. We then apply the rewritten two-chain relation (Lemma \ref{lem: alternate 2 chain relation}) for $\alpha$ and $c_1$, which gives
    \begin{equation}\label{eq: two chain a c re}
        (T_\alpha^3T_{c_1})^3=T_{d''},
    \end{equation}
    where $d''$ is the boundary of a regular neighborhood of
    $\alpha\cup c_1$. Thus, it suffices to show that $\widehat{\rho}(T_{d''})=I$, since applying $\widehat{\rho}$ to the above relation would then yield a contradiction. 
    
    We now prove that $\widehat{\rho}(T_{d''})=I$. By Lemma \ref{lem: generate Torelli}, the element $T_{d''}$ is a product of bounding pair maps supported in $S_{g-1}^2$ and conjugates of the bounding pair maps $F_k$. By our assumption on $\rho|_{\Mod(S_{g-1}^2)}$, every bounding pair map supported in $S_{g-1}^2$ maps to the identity under $\widehat{\rho}$. For $F_k$, first note that
    \[\widehat{\rho}(F_k)^3=\widehat{\rho}(T_{d_k}^3)\widehat{\rho}(T_{d_k'}^3)^{-1}=I,\]
    since $\widehat{\rho}(T_{d_k}^3)$ and $\widehat{\rho}(T_{d_k'}^3)$ are both conjugate to $\widehat{\rho}(T_\alpha^3)=I$. As in the proof of Proposition \ref{prop: std sym rep fails}, we know that $\widehat{\rho}(F_2)$ commutes with the images under $\widehat{\rho}$ of the Dehn twists about the curves
    \[a_2,b_2,\quad a_3,b_3,c_3,a_4,b_4,c_4,\ldots,c_{g-1},a_g,b_g,\]
    as well as with $\widehat{\rho}(T_{c_1})$. It follows that
    \[\widehat{\rho}(F_2)=\operatorname{diag}(\lambda_1,\lambda_1,\lambda_1,\lambda_2,\ldots,\lambda_2),\qquad \lambda_1^3=\lambda_2^3=1.\]
    Since $\rho$ takes values in $\Sp_{2g}(\mathbb{Z})$ and $g\ge 5$, we must have
    \[\lambda_1=\lambda_2=1,\]
    and hence $\widehat{\rho}(F_2)=I$. Moreover, as shown in the proof of Proposition \ref{prop: std sym rep fails}, each $\widehat{\rho}(F_k)$ is a product of conjugates of $\widehat{\rho}(F_2)$. Therefore $\widehat{\rho}(F_k)=I$ for each $2\leq k\leq g-1$. Consequently, $\widehat{\rho}(T_{d''})=I$, and therefore $r\neq0$.

    A minor remark concerning the case of \(\Mod(\widetilde{S},\sigma)\) is that $\widehat{\rho}(\sigma)$ is scalar, since it commutes with
    $\widehat{\rho}(f)$ for any $f\in \Mod(S_{g-1}^2)$. Then the above argument applies without change.

    \noindent \textbf{(2) Solve $r\neq 0$.} In this case, the order-$3$ matrix
    \[A\cdot \begin{pmatrix}1&0\\ -1&1\end{pmatrix}=\begin{pmatrix}s-r & r\\ t-q&q\end{pmatrix}\]
    must have distinct eigenvalues $\zeta_3$ and $\zeta_3^2$. Hence its trace is
    \[s+q-r=\zeta_3+\zeta_3^2=-1.\]
    A direct computation of the relation \eqref{eq: A relation 2} shows that its $(4,3)$-entry is
    \[r^2\big( r(s+q+2)-r^2-3)\big)=0.\]
    Combining the trace formula, we have
     \[r^2\big( r(s+q+2)-r^2-3)\big)=r^2\big( r(r+1)-r^2-3)\big)=r^2(r-3)=0.\]
     Since $r\neq 0$, it follows that $r=3$, and hence
     \[A=\begin{pmatrix}
     s & 3 \\ \frac{-s^2+2s-1}{3} & 2-s
     \end{pmatrix}.\]
     We now conjugate $\rho$ by 
     \[\begin{pmatrix}
     1 & 0 & \mathbf{0} \\ 
     \frac{s-1}{3} & 1 &\mathbf{0} \\
     \mathbf{0} & \mathbf{0} & I_{2g-2}
     \end{pmatrix}.\]
     After this conjugation, we have
     \[\rho(T_{\alpha}^3)= \begin{pmatrix}
     1 & 3 & \mathbf{0} \\ 
     0 & 1 &\mathbf{0} \\
     \mathbf{0} & \mathbf{0} & I_{2g-2}
     \end{pmatrix},\]
     and $\rho|_{\Mod(S_{g-1}^2)}$ remains unchanged. This completes the proof.
\end{proof}

\subsubsection{Determining the image of $F_k$ and $\sigma$}

\begin{proposition}
Suppose that
\[\rho|_{\Mod(S_{g-1}^2)}:\Mod(S_{g-1}^2)\longrightarrow\Sp_{2g}(\ZZ)\]
is conjugate to the standard action on $H_1(S_g;\ZZ)$. Then the image under $\rho$ of each bounding pair map
\[F_k\in\Mod(S,[\beta]),\qquad 2\le k\le g-1,\]
shown in Figure~\ref{fig: F bounding pair maps} is the identity matrix. In the case of $\Mod(\widetilde{S},\sigma)$, both a lift of $F_k$ and the deck transformation $\sigma$ have trivial image under $\rho$.
\end{proposition}
\begin{proof}
First, we claim that it suffices to prove that the image under $\rho$ of each $F_k$ (resp.\ its lift in $\Mod(\widetilde{S},\sigma)$ and $\sigma$) is a scalar matrix. Let
\[\Phi\colon\Mod(S,[\beta])\longrightarrow\Sp_{2g}(\ZZ)
\qquad\text{and}\qquad\widetilde{\Phi}\colon\Mod(\widetilde{S},\sigma)\longrightarrow\Sp_{2g}(\ZZ)\]
denote the standard actions on $H_1(S_g;\ZZ)$. Suppose that $\rho|_{\Mod(S_{g-1}^2)}$ is the standard action on $H_1(S_g;\ZZ)$; we ignore conjugation from now on. By Proposition~\ref{prop: image of T a cubic}, $\rho(T_\alpha^3)$ is also the standard action of $T_\alpha^3$ on $H_1(S_g;\ZZ)$. Once we know that the images of the $F_k$ (resp.\ their lifts in $\Mod(\widetilde{S},\sigma)$ and $\sigma$) are scalar matrices, it follows that
\[\rho\Phi^{-1}\qquad\text{or}\qquad\rho\widetilde{\Phi}^{-1}\]
is scalar on every generator of $\Mod(S,[\beta])$ (resp.\ $\Mod(\widetilde{S},\sigma)$) in Proposition~\ref{prop: gens by Dey}. Moreover, $\rho\Phi^{-1}$ (resp.\ $\rho\widetilde{\Phi}^{-1}$) is a homomorphism:
\[(\rho\Phi^{-1})(xy)=\rho(xy)\Phi(xy)^{-1}
=\rho(x)\rho(y)\Phi(y)^{-1}\Phi(x)^{-1}
=\rho(x)\Phi(x)^{-1}\rho(y)\Phi(y)^{-1},\]
where we use the fact that $\rho(y)\Phi(y)^{-1}$ is scalar. The same argument applies to $\rho\widetilde{\Phi}^{-1}$. Hence $\rho\Phi^{-1}$ (resp.\ $\rho\widetilde{\Phi}^{-1}$) factors through the abelianization of $\Mod(S,[\beta])$ (resp.\ $\Mod(\widetilde{S},\sigma)$). By Lemma~\ref{lem: abelianization}, this abelianization is completely known and consists entirely of $3$-torsion. Since $\zeta_3 I\notin\Sp_{2g}(\ZZ)$,
it follows that
\[\rho=\Phi\qquad\text{or}\qquad\rho=\widetilde{\Phi},\]
respectively. Consequently, the images of the $F_k$ (resp.\ their lifts in $\Mod(\widetilde{S},\sigma)$ and $\sigma$) are the identity.

We now show that the image under $\rho$ of $F_k$ (resp.\ their lifts in $\Mod(\widetilde{S},\sigma)$ and $\sigma$) are scalar. It is immediate to obtain that $\rho(\sigma)$ is scalar since it commutes with $\rho(\Mod(S_{g-1}^2))$ and with $\rho(T_\alpha^3)$. We then determine the image of $F_k$ or its lift in  $\Mod(\widetilde{S},\sigma)$ without differentiation. First, observe that
\[\rho(F_k)^3=\rho(T_{d_k}^3)\rho(T_{d_k'}^3)^{-1}=I,\qquad 2\le k\le g-1.\]
Indeed, $\rho(T_{d_k}^3)$ and $\rho(T_{d_k'}^3)$ are the standard actions of $T_{d_k}^3$ and $T_{d_k'}^3$ on $H_1(S_g;\ZZ)$. These actions agree because $d_k$ and $d_k'$ are homologous in $S_g$. More explicitly, by Lemma~\ref{lem: change curves in cut surface}, there exist
\[
f_k,f_k'\in\Mod(S)
\]
such that
\[
f_k(\alpha)=d_k,\qquad f_k(\beta)=\beta,
\]
and
\[
f_k'(\alpha)=d_k',\qquad f_k'(\beta)=\beta.
\]
After isotopy, we may assume that $f_k$ and $f_k'$ fix $\beta$ pointwise, so that
\[
f_k,f_k'\in\Mod(S_{g-1}^2).
\]
By the assumption and Proposition~\ref{prop: image of T a cubic}, the images of $f_k$, $f_k'$, and $T_\alpha^3$ under $\rho$ are their standard actions on $H_1(S_g;\ZZ)$. Consequently, the same is true for $T_{d_k}^3$ and $T_{d_k'}^3$.

Next, as shown in the proof of Proposition~\ref{prop: std sym rep fails}, the lantern relation in Figure~\ref{fig: F lantern relation} gives
\[
\rho(F_k)=\rho(T_{a_k'}T_{a_1}^{-1}).
\]
In particular, for $k=2$, the curves $a_2'$ and $a_1=\alpha$ are disjoint from the curves
\[
a_2,b_2
\]
and
\[
a_3,b_3,c_3,a_4,b_4,c_4,\ldots,c_{g-1},a_g,b_g
\]
shown in Figure~\ref{fig: basis in surface}. Hence $T_{a_2'}T_{a_1}^{-1}$ commutes with the Dehn twists about all these curves. It also commutes with $T_\alpha^3$. These commuting relations imply that
\[
\rho(F_2)
=
\operatorname{diag}\left(
\begin{pmatrix}
\lambda_1 & \mu\\
0 & \lambda_1
\end{pmatrix},
\lambda_2,\lambda_2,
\lambda_3,\ldots,\lambda_3
\right).
\]
Since $\rho(F_2)^3=I$, we have $\mu=0$. Moreover, since $F_2$ commutes with $T_{c_1}$ in Figure \ref{fig: basis in surface}, we have
\[
\lambda_1=\lambda_2.
\]
Since $\lambda_i^3=1$ and $g\ge5$, while $\rho(F_2)$ is conjugate to a matrix in $\Sp_{2g}(\ZZ)$, it follows that 
\[
\lambda_1=\lambda_3=1.
\]
Thus
\[
\rho(F_2)=I.
\]
Finally, the proof of Proposition~\ref{prop: std sym rep fails} shows that every $\rho(F_k)$ is a product of conjugates of $\rho(F_2)$. Hence
\[\rho(F_k)=I,\qquad  2\le k\le g-1,\]
which completes the proof.
\end{proof}

\subsubsection{Controlling the conjugating matrix}
Up to this point, we have established that, unless $\rho$ has abelian image, it is conjugate to the standard action on $H_1(S_g;\ZZ)$ by some element of $\GL_{2g}(\CC)$. We now show that the conjugating matrix can be chosen in the extended symplectic group
\[\Delta_{2g}(\ZZ)=\left\{X\in\GL_{2g}(\ZZ):X^tJX=\pm J
\right\}.\]

Note that the image of the standard action of $\Mod(S,[\beta])$ (or $\Mod(\widetilde{S},\sigma)$) on $H_1(S_g;\ZZ)$ is
\[\Sp_{2g}(\ZZ)^{[\beta]}=\operatorname{Stab}_{\Sp_{2g}(\ZZ)}([\beta]),\]
which is a finite-index subgroup of $\Sp_{2g}(\ZZ)$ since $[\beta]\in H_1(S_g;\ZZ/p\ZZ)^*$.

We first restrict the conjugating matrix to $\Sp_{2g}(\CC)$ by the following lemma.
\begin{lemma}
Suppose $\Gamma$ is a finite-index subgroup of $\Sp_{2g}(\ZZ)$, and let $A\in\GL_{2g}(\CC)$ be such that
\[A\Gamma A^{-1}\subset \Sp_{2g}(\CC).\]
Then a scalar multiple of $A$ lies in $\Sp_{2g}(\CC)$.
\end{lemma}
\begin{proof}
For every $X\in\Gamma$, we have
\[AXA^{-1}\in\Sp_{2g}(\CC),\]
and hence
\[(AXA^{-1})^tJ(AXA^{-1})=J.\]
Multiplying on the left by $A^t$ and on the right by $A$ gives
\[X^t(A^tJA)X=A^tJA,\qquad \forall X\in \Gamma.\]
This means that $\Gamma$ preserves the non-degenerate alternating bilinear form defined by $A^tJA$. Since $\Gamma$ is a finite-index subgroup of $\Sp_{2g}(\ZZ)$, any $\Gamma$-invariant alternating bilinear form on $\CC^{2g}$ is a scalar multiple of the standard symplectic form. Hence
\[A^tJA=\mu J, \qquad \text{ for some }\mu\in\CC^*.\]
Then $\mu^{-\frac12} A\in\Sp_{2g}(\CC)$ since it preserves the symplectic form $J$.
\end{proof}

We now further restrict the conjugating matrix to $\Delta_{2g}(\ZZ)$ by the following lemma.
\begin{lemma}\label{lem: conjugate restrict to integral}
Let $g\ge 2$\footnote{The statement does not hold for $g=1$. Indeed, if $[\beta]$ lifts to the second basis vector of $\ZZ^2$, then
\[A=\begin{pmatrix}
1/\sqrt{p}&0\\0&\sqrt{p}\end{pmatrix}\in\SL_2(\mathbb R)\]
conjugates $\SL_2(\ZZ)^{[\beta]}$ into $\SL_2(\ZZ)$. This is essentially the reason that the commensurator of $\SL_2(\ZZ)$ in $\SL_2(\mathbb R)$ is larger than $\SL_2(\QQ)$.}, and let $p$ be any prime.
If $A\in\Sp_{2g}(\CC)$ satisfies
\[A\Sp_{2g}(\ZZ)^{[\beta]}A^{-1}\subset \Sp_{2g}(\ZZ),\]
then either
\[A\in\Sp_{2g}(\ZZ)\]
or
\[iA\in \left\{X\in\GL_{2g}(\ZZ):X^tJX=- J\right\}.\]
\end{lemma}

\begin{proof}
    For a vector $v\in\CC^{2g}$, define the transvection about $v$ by\[T_v\colon \CC^{2g}\longrightarrow\CC^{2g},\qquad u\longmapsto u+\omega(v,u)v,\]
    where $\omega$ is the fixed symplectic form on $\CC^{2g}$. Since $A\in\Sp_{2g}(\CC)$, we have
    \[AT_vA^{-1}=T_{Av}.\]
    Let $v_1,\ldots,v_{2g}$ be a symplectic basis of $(\CC^{2g},\omega)$ satisfying
    \[\omega(v_{2i-1},v_{2i})=1,\qquad\omega(v_i,v_j)=0\quad\text{if }|i-j|\ge 2.\]
    Suppose that $[v_2]=[\beta]\in(\ZZ/p\ZZ)^{2g}$. By definition,
    \[T_{v_1}^p\in\Sp_{2g}(\ZZ)^{[\beta]},\qquad T_{v_i}\in\Sp_{2g}(\ZZ)^{[\beta]}\quad\text{for }2\le i\le 2g.\]
    Write $A=(A_{ij})$. We divide the proof into three steps.
    
    \noindent \textbf{Step 1.} We first show that there exists $m\in\CC^*$ with $m^2\in\ZZ$ such that $mA$ has integral entries. This step also holds when $g=1$. By assumption,
    \[AT_{v_1}^pA^{-1}=T_{Av_1}^p\in\Sp_{2g}(\ZZ)\]
    and
    \[AT_{v_i}A^{-1}=T_{Av_i}\in\Sp_{2g}(\ZZ),\qquad 2\le i\le 2g.\]
    Write
     \begin{align*}
        T_{A(v_1)}^p(v_k)&=v_k+p\cdot\omega(A(v_1),v_k)A(v_1) \\
        &=v_k+p\cdot \omega(\sum\limits_{s=1}^{2g}A_{s1}v_s,v_k)\sum\limits_{t=1}^{2g}A_{t1}v_t,
    \end{align*}
    and for each $2\le i \le 2g$
     \begin{align*}
        T_{A(v_i)}(v_k)&=v_k+\omega(A(v_i),v_k)A(v_i) \\
        &=v_k+\omega(\sum\limits_{s=1}^{2g}A_{si}v_s,v_k)\sum\limits_{t=1}^{2g}A_{ti}v_t.
    \end{align*}
    The integrality of these transvections implies
    \begin{equation}\label{eq: integral 1}
     pA_{s1}A_{t1}\in\ZZ,\qquad A_{si}A_{ti}\in\ZZ\quad\text{for }2\le i\le 2g,\quad 1\le s,t\le 2g.\end{equation}
    Moreover, for $2\le i\le 2g$, we have
    \[T_{pv_1+v_i}\in\Sp_{2g}(\ZZ)^{[\beta]}.\]
    Therefore
    \[AT_{pv_1+v_i}A^{-1}=T_{A(pv_1+v_i)}\in\Sp_{2g}(\ZZ),\]
    which is
    \begin{align*}
        T_{A(pv_1+v_i)}(v_k)&= v_k+\omega(A(pv_1+v_i),v_k)A(pv_1+v_i)\\
        &=v_k+\omega(\sum\limits_{s=1}^{2g}(pA_{s1}+A_{si})v_s,v_k)\sum\limits_{t=1}^{2g}(pA_{t1}+A_{ti})v_t.
    \end{align*}
    Again by integrality, we obtain
    \[(pA_{s1}+A_{si})(pA_{t1}+A_{ti})\in\ZZ.\]
    Together with \eqref{eq: integral 1}, this gives
    \begin{equation}\label{eq: integral 2}
    pA_{s1}A_{ti}+pA_{t1}A_{si}\in\ZZ,\qquad 2\le i\le 2g,\quad 1\le s,t\le 2g.\end{equation}
    Choose $k$ such that $A_{k1}\neq0$, and set
    \[m=2p^2A_{k1}^3.\]
    By \eqref{eq: integral 1}, we have
    \[m^2=4(pA_{k1}^2)^2(pA_{k1}^2)\in\ZZ\setminus\{0\}.\]
    For the first column of $mA$, we have
    \[mA_{s1}=2p^2A_{k1}^3A_{s1}=2(pA_{k1}^2)(pA_{k1}A_{s1})\in\ZZ
    ,\]
    by \eqref{eq: integral 1}.
    For $2\le i\le2g$, we have
    \[\begin{aligned}mA_{si}&=2p^2A_{k1}^3A_{si}\\&=2pA_{k1}^2\bigl(pA_{k1}A_{si}+pA_{s1}A_{ki}\bigr)-(pA_{k1}A_{s1})(2pA_{k1}A_{ki}),\end{aligned}\]
    which is integral by \eqref{eq: integral 1} and \eqref{eq: integral 2}. Thus $mA$ is an integral matrix.
    
    \noindent \textbf{Step 2.} We next show that, when $g\ge2$, the integer $|m^2|$ is a perfect square. Suppose otherwise. Then there exists a prime $q$ and an integer $N>0$ such that
    \[q^{2N-1}\mid m^2,\qquad q^{2N}\nmid m^2.\]
    Write the integral matrix
    \[B=mA=(B_{ij}).\]
    By \eqref{eq: integral 1}, we have
    \[A_{ij}^2\in\ZZ\qquad\text{for }j\ge2.\]
    Hence
    \begin{equation}\label{eq: integral 3}
        B_{ij}^2=m^2A_{ij}^2, \qquad j\ge2,
    \end{equation}
    is divisible by $m^2$, and therefore
    \[q^N\mid B_{ij},\qquad j\ge2.\]
    Since $g\ge2$, we compute
    \[\begin{aligned}\omega(Bv_3,Bv_4)=\omega(mAv_3,mAv_4)=m^2\omega(Av_3,Av_4)=m^2\omega(v_3,v_4)=m^2\end{aligned},\]
    which also equals to
   \[\omega(Bv_3,Bv_4)=\sum_{s,t=1}^{2g}B_{s3}B_{t4}\omega(v_s,v_t).\]
   Since $B_{s3}$ and $B_{t4}$ are divisible by $q^N$, we obtain
   \[q^{2N}\mid\omega(Bv_3,Bv_4).\]
   Thus $q^{2N}\mid m^2$, contradicting the choice of $q$ and $N$. 
    
    \noindent \textbf{Step 3.} 
    Finally, we show that $A$ or $iA$ is integral. Since $m^2=\pm n^2$ for some $n\in \mathbb{N}_+$, we have
    \[A=\frac{1}{n}B\qquad\text{or}\qquad A=\frac{i}{n}B,\]
    for some $B\in\operatorname{Mat}_{2g\times2g}(\ZZ)$. 
    By \eqref{eq: integral 3}, we have
    \[n|B_{ij},\quad j \ge 2,\]
    therefore the last $2g-1$ columns of $\frac{1}{n}B$ are integral. It remains to consider the first column of $\frac{1}{n}B$. By \eqref{eq: integral 1}, we have
    \[pA_{s1}^2\in\ZZ,\]
    which implies 
    \[p(\frac{B_{s1}}{n})^2=p\frac{m^2A_{s1}^2}{n^2}\in \ZZ.\]
    Comparing the prime factorizations of $B_{s1}$ and $n$, it follows that
    \[n|B_{s1}.\]
    Hence $\frac{1}{n}B$ is integral. 
    
    If $A=\frac{1}{n}B$, then $A\in\Sp_{2g}(\ZZ)$. If $A=\frac{i}{n}B$, then $iA=-\frac{1}{n}B$ is integral and
    \[(iA)^tJ(iA)=-A^tJA=-J.\qedhere\]
\end{proof}
The proof of Theorem \ref{thm: classify Mod to Sp} is now accomplished.

\section{Rigidity of the period map up to finite covers}\label{sec: 4}
In this section, we give a geometric application of the algebraic results established in Theorem~\ref{thm: classify Mod to Sp}. We study the uniqueness of the period map among nonconstant holomorphic maps between certain moduli spaces. This generalizes the work of Farb~\cite{FarbRigidity} and Serván~\cite{CarlosPrym}.

\subsection{Background and setup}
We work in the category of complex orbifolds considered in \cite[Remark 2.1]{FarbRigidity}. We recall the relevant definitions.
\begin{definition}
A \emph{complex orbifold} is a quotient space $X/\Gamma$ of a complex manifold $X$ by a group $\Gamma$ acting properly discontinuously on $X$ by biholomorphic automorphisms.

Let $X/\Gamma$ and $Y/\Lambda$ be complex orbifolds, and let \[\rho\colon\Gamma\to\Lambda\]
be a group homomorphism. A map
\[F\colon X/\Gamma\longrightarrow Y/\Lambda\]
is a \emph{holomorphic map of complex orbifolds} if it is induced by a holomorphic map
\[\widetilde F\colon X\longrightarrow Y\]
satisfying
\[\widetilde F(f\cdot x)=\rho(f)\widetilde F(x),\qquad
f\in\Gamma,\ x\in X.\]
\end{definition}

\begin{example}
We will be particularly interested in the following examples of complex orbifolds.
\begin{enumerate}
    \item The moduli space of smooth closed complex curves of genus $g$:
    \[\mathcal{M}_g=\Teich(S_g)/\Mod(S_g), \]
    where $\Teich(S_g)$ denotes the Teichm\"uller space of $S_g$. It is well known that
    \[\Teich(S_g)\cong\CC^{3g-3}\]
    as a complex manifold, and $\Mod(S_g)$ acts on it properly discontinuously.
    \item The moduli space of principally polarized abelian varieties of dimension $h$:
    \[ \mathcal{A}_h =\mathfrak{H}_h/\Sp_{2h}(\ZZ),\]
    where $\mathfrak{H}_h$ is the Siegel upper half space, consisting of symmetric $h\times h$ complex matrices with positive-definite imaginary part. The symplectic group $\Sp_{2h}(\ZZ)$ acts properly discontinuously on $\mathfrak{H}_h$ by fractional linear transformations.
\end{enumerate}
\end{example}

The main example we consider is the moduli space of genus-$g$ curves equipped with a $p$-sheeted (unbranched) normal covering. It can be described as
\[R_g^{(p)}=\left\{
(X,\theta_X)
\mathrel{}\middle|\mathrel{}
\begin{array}{c}
X\text{ is a smooth curve of genus }g,\\
\theta_X\in H^1(X;\ZZ/p\ZZ)^*
\end{array}
\right\}\bigg/\sim,\]
where $(X_1,\theta_{X_1})\sim(X_2,\theta_{X_2})$ if and only if there exists a biholomorphism $f\colon X_1\to X_2$ such that
\[f^*(\theta_{X_2})=\theta_{X_1}.\]
Equivalently, by covering space theory, each nonzero class $\theta_X$ determines an unbranched $p$-fold cyclic cover $Y\to X$.

There are two natural complex orbifold structures on $R_g^{(p)}$, introduced in \cite[Section~2]{CarlosPrym}.
\begin{definition}\label{def: two orbifold structures}
The moduli space $R_g^{(p)}$ admits two natural complex orbifold structures, as follows.
\begin{enumerate}
    \item Fix a nonzero element
    \[[\beta]\in H_1(S_g;\ZZ/p\ZZ).\]
    Recall that
    \[\Mod(S,[\beta]) =\operatorname{Stab}_{\Mod(S_g)}([\beta]). \]
    The first orbifold structure is
    \[\widehat{R}_g^{(p)}=\Teich(S_g)/\Mod(S,[\beta]).\]
    \item Let 
    $ \widetilde{S}\to S_g$
    be the $p$-fold unbranched cyclic cover determined by $[\beta]$, and let $\sigma$ be a generator of its deck transformation group. Recall that
    \[\Mod(\widetilde{S},\sigma) = C_{\Mod(\widetilde{S})}(\sigma)\]
    is the centralizer of $\sigma$ in $\Mod(\widetilde{S})$. Let
    \[ \Teich(\widetilde{S})^\sigma\]
    denote the fixed-point locus of $\sigma$ in $\Teich(\widetilde{S})$. The second orbifold structure is
    \[  R_g^{(p)}=
    \Teich(\widetilde{S})^\sigma/\Mod(\widetilde{S},\sigma).\]
\end{enumerate}

The two orbifold structures have the same underlying set of points, but the second action is not effective: the deck transformation $\sigma$ acts trivially on $\Teich(\widetilde{S})^\sigma$. We distinguish the two orbifold structures by denoting the first by $\widehat{R}_g^{(p)}$ and retaining $R_g^{(p)}$ for the second.
\end{definition}

\begin{example}
We will be interested in the following holomorphic maps between these complex orbifolds arising from algebraic geometry.
\begin{enumerate}
    \item The period map:
    \[\begin{aligned}
    J\colon \mathcal{M}_g &\longrightarrow \mathcal{A}_g,\\
    X &\longmapsto \operatorname{Jac}(X)
    =\frac{\Omega^1(X)^\vee}{H_1(X;\ZZ)},
    \end{aligned}\]
    where $\Omega^1(X)^\vee\cong\CC^g$ denotes the dual of the space of holomorphic $1$-forms on $X$. Here $H_1(X;\ZZ)$ is realized as a lattice in $\Omega^1(X)^\vee$ via
    \[ \begin{aligned}
    H_1(X;\ZZ)&\longrightarrow \Omega^1(X)^\vee,\\
    \gamma&\longmapsto
    \left(\omega\longmapsto\int_\gamma\omega\right).
    \end{aligned} \]
    The Jacobian $\operatorname{Jac}(X)$ carries a principal polarization induced by the intersection pairing on $H_1(X;\ZZ)$. On orbifold fundamental groups, the period map $J$ induces the standard symplectic representation
    \[\Phi: \Mod(S_g)\longrightarrow\Aut(H_1(S_g;\ZZ),\omega)=\Sp_{2g}(\ZZ).\]
    \item The Prym map:
    \[ \begin{aligned}
    \text{Prym}\colon R_g^{(p)}&\longrightarrow
    \mathcal{A}_{(p-1)(g-1)},\\
    (Y\xrightarrow{\pi} X)&\longmapsto
    \operatorname{Prym}(Y/X)=\frac{\operatorname{Jac}(Y)}{\pi^*\big(\operatorname{Jac}(X)\big)},
    \end{aligned}\]
    where $\operatorname{Prym}(Y/X)$ denotes the Prym variety associated to the cyclic cover $Y\to X$.
\end{enumerate}
\end{example}

These maps exhibit a remarkable rigidity. More precisely:
\begin{enumerate}
    \item Farb \cite[Theorem~1.1]{FarbRigidity} proved that, for $g\ge3$, the period map $J$ is the unique nonconstant holomorphic map from $\mathcal{M}_g$ to $\mathcal{A}_{h}$ for $h\le g$.
    \item For $p=2$ and $g\ge 4$, Serv\'an \cite[Theorem~1.1]{CarlosPrym} proved that the Prym map is the unique nonconstant holomorphic map from $R_g^{(2)}$ to $\mathcal{A}_{h}$ for $h\le g-1$. Note that the Prym map is not well defined on the other orbifold structure $\widehat{R}_g^{(2)}$. Moreover, Serván \cite[Theorem~1.2]{CarlosPrym} proved that every holomorphic map from $\widehat{R}_g^{(2)}$ to $\mathcal{A}_{h}$ is constant for $h\le g-1$.
\end{enumerate}
   
\subsection{Rigidity of the period map for $p=3$}\label{sec: rigidity}
We now generalize the above rigidity results to the case $p=3$. For $p\ge3$, the target of the Prym map has dimension greater than $g$ when $g\ge3$. Instead, there are natural holomorphic maps from both $R_g^{(p)}$ and $\widehat{R}_g^{(p)}$ to $\mathcal{A}_g$, induced by the period map. First, consider
\[\widehat{R}_g^{(p)}=\Teich(S_g)/\Mod(S,[\beta]).\]
Since $\Mod(S,[\beta])$ is a finite-index subgroup of $\Mod(S_g)$, the quotient map
\[\widehat{R}_g^{(p)}\longrightarrow\mathcal{M}_g\]
is a finite orbifold covering. We denote by $J^{(p)}$ the composition
\[\widehat{R}_g^{(p)}\to\mathcal{M}_g\xrightarrow{J}\mathcal{A}_g.\]
Now consider the other orbifold structure
\[R_g^{(p)}=\Teich(\widetilde{S})^\sigma/\Mod(\widetilde{S},\sigma).\]
The covering map $\widetilde{S}\to S_g$ induces a holomorphic map
\[\Teich(\widetilde{S})^\sigma\longrightarrow\Teich(S_g)\]
by sending a $\sigma$-invariant complex structure on $\widetilde{S}$ to the induced complex structure on $S_g$. This map is equivariant with respect to the group homomorphism
\[\Mod(\widetilde{S},\sigma)\longrightarrow\Mod(S,[\beta]).\]
Consequently, it descends to a holomorphic map
\[\nu\colon R_g^{(p)}\longrightarrow\widehat{R}_g^{(p)}.\]
Composing $\nu$ with the period map gives a holomorphic map
\[J^{(p)}\circ\nu\colon R_g^{(p)}\xrightarrow{\,\nu}\widehat{R}_g^{(p)}\xrightarrow{\,J^{(p)}\,}\mathcal A_g.\]
We show that, when $p=3$, these are the unique nonconstant holomorphic maps to $\mathcal A_h$ for $h\le g$.
\begin{theorem}\label{thm: unique holomorphic map}
Let $g\ge6$ and $h\le g$. If
\[F\colon R_g^{(3)}\to\mathcal A_h
\qquad\text{or}\qquad
F\colon\widehat{R}_g^{(3)}\to\mathcal A_h\]
is a nonconstant holomorphic map of complex orbifolds, then $h=g$ and
\[F=J^{(3)}\circ\nu \qquad\text{or}\qquad F=J^{(3)},\]
respectively.
\end{theorem}

\begin{remark}
If we assume only that $F$ is continuous, then part of the proof of Theorem~\ref{thm: unique holomorphic map} also shows that either $F$ is homotopic to $J^{(3)}\circ\nu$ (resp.~$J^{(3)}$), or there exists a finite cover $\widetilde{R}_g^{(3)}$ of $R_g^{(3)}$ (resp.~$\widehat{R}_g^{(3)}$) such that the lift of $F$ to $\widetilde{R}_g^{(3)}$ is homotopic to a constant map.
\end{remark}
\begin{proof}
The proof is based on Theorem~\ref{thm: classify Mod to Sp}, which classifies homomorphisms from $\Mod(S,[\beta])$ and $\Mod(\widetilde{S},\sigma)$ to $\Sp_{2h}(\ZZ)$ for $h\le g$. We then follow the approach developed by Farb \cite{FarbRigidity} in his proof of the global rigidity of the period map. The proof is divided into the following steps. In the next subsection, we also give an alternative approach to one of the steps using the theory of variations of Hodge structures, following an approach suggested to Farb by Richard Hain.

\noindent\textbf{Step 1: Uniqueness up to homotopy.}
We first show that $F$ is homotopic to $J^{(3)}\circ\nu$ (resp.~$J^{(3)}$). The map $F$ induces a homomorphism on orbifold fundamental groups
\[F_*\colon\pi_1^{\mathrm{orb}}(R_g^{(3)})=\Mod(\widetilde{S},\sigma)
\longrightarrow
\pi_1^{\mathrm{orb}}(\mathcal A_h)=\Sp_{2h}(\ZZ),\]
or
\[F_*\colon\pi_1^{\mathrm{orb}}(\widehat{R}_g^{(3)})
=\Mod(S,[\beta])
\longrightarrow
\pi_1^{\mathrm{orb}}(\mathcal A_h)=\Sp_{2h}(\ZZ).\]
By Theorem~\ref{thm: classify Mod to Sp}, $F_*$ either has abelian image, or is conjugate to the standard action on $H_1(S_g;\ZZ)$ by an element of the extended symplectic group $\Delta_{2g}(\ZZ)$ (in this case $h=g$).

Suppose first that $F_*$ has abelian image. The abelianization of $\Mod(\widetilde{S},\sigma)$ (resp.~$\Mod(S,[\beta])$) is known in Lemma~\ref{lem: abelianization} and is finite. Hence there exists a finite cover
\[\widetilde{R}_g^{(3)}\longrightarrow R_g^{(3)}\qquad
(\text{resp. }\widetilde{R}_g^{(3)}\longrightarrow\widehat{R}_g^{(3)})\]
such that the lift of $F$ to $\widetilde{R}_g^{(3)}$ induces the trivial homomorphism on orbifold fundamental groups, which further lifts to a holomorphic map
\[\widetilde{F}\colon\widetilde{R}_g^{(3)}\longrightarrow\mathfrak{H}_h.\]
Since $\mathfrak{H}_h$ is biholomorphic to a bounded domain, $\widetilde{F}$ is constant. Hence $F$ is constant, contradicting our assumption.

Therefore, $F_*$ is conjugate to the standard symplectic representation on $H_1(S_g;\ZZ)$
by an element
\[A\in\Delta_{2g}(\ZZ)=\{X\in\GL_{2g}(\ZZ):X^tJX=\pm J\}.\]
The standard symplectic representation is precisely the representation on orbifold fundamental groups induced by $J^{(3)}\circ\nu$ (resp.~$J^{(3)}$).

Suppose first that $A\in\Sp_{2g}(\ZZ)$. Since $\Sp_{2g}(\ZZ)$ acts by isometries on the Siegel upper half space $\mathfrak{H}_g$, there is an equivariant homotopy from
\[\widetilde F\colon\Teich(\widetilde{S})^\sigma\longrightarrow\mathfrak{H}_g\qquad \text{(resp. }\
\widetilde F\colon\Teich(S_g)\longrightarrow\mathfrak{H}_g\text{)}\]
to
\[\widetilde{J^{(3)}\circ\nu}\colon
\Teich(\widetilde{S})^\sigma\longrightarrow\mathfrak{H}_g \qquad
\text{(resp. }\
\widetilde{J^{(3)}}\colon\Teich(S_g)\longrightarrow\mathfrak{H}_g\text{)}.\]
This homotopy descends to a homotopy from $F$ to $J^{(3)}\circ\nu$ (resp.~$J^{(3)}$).

The case in which $A$ reverses the symplectic form can be excluded by the same argument as in \cite[Section 4.3.2]{CarlosPrym}. Consider the anti-holomorphic involution
\[\tau\colon\mathfrak{H}_g\longrightarrow\mathfrak{H}_g,
\qquad
\tau(Z)=-\overline{Z}.\]
If $A$ is anti-symplectic, the preceding argument implies that
$\tau\circ\widetilde F$ is equivariantly homotopic to
$\widetilde{J^{(3)}}$ (resp.~$\widetilde{J^{(3)}\circ\nu}$). This is impossible. Indeed, there exists a generic smooth complete curve (see Step 5 below)
\[C\subset R_g^{(3)}
\qquad\text{(resp. }C\subset\widehat{R}_g^{(3)}\text{)}\]
whose inclusion induces a surjection
\[\pi_1(C)\longrightarrow \pi_1^{\mathrm{orb}}(R_g^{(3)})
\qquad\text{(resp. }\pi_1^{\mathrm{orb}}(\widehat{R}_g^{(3)})\text{)}.\]
Integrating the pullbacks to $C$ of the K\"ahler form on $\mathcal A_g$ via $F$ and via $J^{(3)}\circ\nu$ (resp.~$J^{(3)}$), and applying Stokes' theorem, shows that $F$ is constant on $C$. Since $C$ is generic, this implies that $F$ is constant, which is a contradiction.

\noindent\textbf{Step 2: The Borel--Narasimhan theorem.}
We next apply the Borel--Narasimhan theorem \cite[Theorem~3.6]{BorelNarasimhan}, following Step~2 of Farb's proof \cite{FarbRigidity}. The theorem states that if $X$ is a connected complex manifold admitting no nonconstant plurisubharmonic function bounded above, and $Y$ is a complex manifold covered by a bounded domain in $\CC^N$, then two holomorphic maps \[F,G\colon X\to Y\] that agree at one point and induce the same map on fundamental groups must agree everywhere. The same statement applies to the orbifolds considered here.

In our setting, $Y=\mathcal A_g$ satisfies the required condition since $\mathfrak{H}_g$ is biholomorphic to a bounded symmetric domain in $\CC^{\frac{g(g+1)}{2}}$. On the other hand, both $R_g^{(3)}$ and $\widehat{R}_g^{(3)}$ are quasiprojective complex varieties, and hence admit no nonconstant plurisubharmonic functions bounded above by \cite[Proposition~2.1]{BorelNarasimhan}. Thus the Borel--Narasimhan theorem applies, and it remains only to find a point where $F$ and $J^{(3)}\circ\nu$ (resp.~$J^{(3)}$) agree.

\noindent\textbf{Step 3: Improving the homotopy to be holomorphic.}
Let $H_t$ ($0\le t\le 1$) be the homotopy from $F$ to
$J^{(3)}\circ\nu$ (resp.~$J^{(3)}$) obtained in Step~1.
Let $C$ be a smooth (not necessarily compact) curve in
$R_g^{(3)}$ (resp.~$\widehat{R}_g^{(3)}$). 

Following Step~3 of
Farb's proof \cite{FarbRigidity}, we can modify the homotopy so that $H_t|_C$ is holomorphic for every $t\in[0,1]$. Indeed, the argument carries over without change, since it only uses the fact that the target is covered by a bounded domain. More precisely, one first deforms $H|_C$ to a geodesic homotopy and then uses the argument of Antonakoudis--Aramayona--Souto \cite[Section~4]{AAS}, together with a variant of the Wirtinger inequality, to show that each $H_t|_C$ isholomorphic.

\noindent\textbf{Step 4: Improving the homotopy to be algebraic.}
Following Step~4 of Farb's proof \cite{FarbRigidity}, we further show
that each $H_t|_C$ is algebraic, i.e.\ a morphism of algebraic
varieties. By \cite[Theorem~2]{KO}, $H_t|_C$ extends to a holomorphic
map from the projective closure of $C$ to the Satake compactification
of $\mathcal{A}_g$. Chow's theorem then implies that $H_t|_C$ is
algebraic.

\noindent\textbf{Step 5: Finding a curve on which the period map is rigid.}
We construct a non-compact smooth curve
\[C\subset R_g^{(3)}
\qquad\text{(resp. }C\subset\widehat{R}_g^{(3)}\text{)}\]
such that
$J^{(3)}\circ\nu|_C$ (resp.~$J^{(3)}|_C$) is rigid, i.e.~an isolated point of $\operatorname{Mor}(C,\mathcal{A}_g)$.

We use Saito's rigidity theorem \cite[Theorem~8.6]{SaitoRigid}, which implies that a map $C\to\mathcal{A}_g$ is rigid provided that
\begin{enumerate}
    \item the monodromy of the pullback of the standard polarized variation of Hodge structures on $\mathcal{A}_g$ is irreducible;
    \item the local monodromy around some point in the boundary of $C$ has infinite order.
\end{enumerate}

We first construct such a curve for $\widehat{R}_g^{(3)}$; its preimage under $\nu$ (which is a biholomorphism in the complex category) then gives the desired curve for $R_g^{(3)}$. For $\widehat{R}_g^{(3)}$, we work in the finite cover given by the moduli space of genus-$g$ curves with a level-$3$ structure
\[\mathcal{M}_g(3)=\Teich(S_g)/\Mod(S_g,3),\]
where
\[\Mod(S_g,3)=\operatorname{Ker}\big(\Mod(S_g)\to
\Sp_{2g}(\ZZ/3\ZZ)\big).\]
Since
\[\Mod(S_g,3)\subset\Mod(S,[\beta]),\]
this gives a finite cover
\[\pi\colon\mathcal{M}_g(3)\longrightarrow\widehat{R}_g^{(3)}.\]
Moreover, $\operatorname{Mod}(S_g,3)$ acts freely on $\operatorname{Teich}(S_g)$, so $\mathcal{M}_g(3)$ is a complex manifold.

The moduli space of principally polarized abelian varieties with a level-$3$ structure is
\[\mathcal{A}_g(3)=\mathfrak{H}_g/\Sp_{2g}(\mathbb{Z},3),\]
where
\[\Sp_{2g}(\mathbb{Z},3)=
\operatorname{Ker}\big(\Sp_{2g}(\mathbb{Z})\to
\Sp_{2g}(\mathbb{Z}/3\mathbb{Z})\big).\]
The map $J^{(3)}$ lifts to a holomorphic map
\[\widehat{J^{(3)}}\colon\mathcal{M}_g(3)\longrightarrow\mathcal{A}_g(3),\]
giving a commutative diagram
\[\xymatrix{
\mathcal{M}_g(3)\ar[r]^{\widehat{J^{(3)}}}\ar[d]^{\pi}&\mathcal{A}_g(3)\ar[d]\\
{\widehat{R}_g^{(3)}}\ar[r]^{J^{(3)}}&\mathcal{A}_g}.\]
It therefore suffices to find a smooth noncompact curve
$C\subset\mathcal{M}_g(3)$, disjoint from the orbifold locus of $\pi$, such that
$\widehat{J^{(3)}}|_C$ is rigid. Its image under $\pi$ is the desired curve in $\widehat{R}_g^{(3)}$.

Let $\mathcal{O}$ denote the orbifold locus of
\[\pi\colon\mathcal{M}_g(3)\longrightarrow\widehat{R}_g^{(3)}.\]
The locus $\mathcal{O}$ has complex dimension at most $2g-1$ by the Riemann--Hurwitz formula. Hence
\[\operatorname{codim}_{\mathcal{M}_g(3)}\mathcal{O}
\ge
(3g-3)-(2g-1)=g-2\ge2.\]
It follows that
\[\pi_1(\mathcal{M}_g(3)\setminus\mathcal{O})\cong\pi_1(\mathcal{M}_g(3)).\]

Let $\overline{\mathcal{M}}_g^{DM}$ be the Deligne--Mumford compactification of $\mathcal{M}_g$. Let $\overline{\mathcal{M}}_g^{DM}(3)$ be the normalization of $\overline{\mathcal{M}}_g^{DM}$ in the function field of $\mathcal{M}_g(3)$, which is a projective compactification of $\mathcal{M}_g(3)$. We then obtain a smooth projective compactification
\[\mathcal{M}_g(3)\setminus\mathcal{O}\subset \overline{\mathcal{M}_g(3)\setminus\mathcal{O}},\]
whose boundary is a normal-crossings divisor. 

Choose a projective embedding \[\overline{\mathcal{M}_g(3)\setminus\mathcal{O}}\hookrightarrow\mathbb{P}^{3g-3},\]
and take $3g-4$ generic hyperplanes
$P_1,\ldots,P_{3g-4}$. Set
\[\overline{C}=\big(\overline{\mathcal{M}_g(3)\setminus\mathcal{O}}\big)\cap P_1\cap \cdots \cap P_{3g-4},\qquad C=\overline{C}\cap \left( \mathcal{M}_g(3)\setminus\mathcal{O}\right).\]
By Bertini's theorem $C$ a smooth curve. We now verify that $C$ satisfies the desired conditions.
\begin{enumerate}
    \item By the Lefschetz hyperplane theorem, the map \[\pi_1(C)\longrightarrow\pi_1(\mathcal{M}_g(3)\setminus\mathcal{O})\cong\pi_1(\mathcal{M}_g(3))\]
    is surjective. The monodromy of the pullback of the standard polarized variation of Hodge structures on $\mathcal{A}_g(3)$ via $\widehat{J^{(3)}}|_C$ is the composition
    \[\pi_1(C)\longrightarrow\pi_1(\mathcal{M}_g(3))\longrightarrow\pi_1(\mathcal{A}_g(3))=\Sp_{2g}(\mathbb{Z},3),\]
    which is therefore surjective. Since $\Sp_{2g}(\mathbb{Z},3)$ is Zariski dense in $\Sp_{2g}$, the resulting representation is irreducible.
    \item Let 
    \[Z\subset\overline{\mathcal{M}}_g^{DM}\setminus\mathcal{M}_g\]
    be the codimension-$1$ stratum parametrizing stable curves obtained by pinching a nonseparating simple closed curve. Let 
    \[Z_3\subset \overline{\mathcal{M}_g(3)\setminus\mathcal{O}} \setminus \left(\mathcal{M}_g(3)\setminus\mathcal{O}\right)\]
    be an irreducible component lying over $Z$. Since $Z_3$ is an effective divisor and the hyperplane class $[P]$ is ample, we have 
    \[[Z_3]\cdot [P]^{3g-4}>0.\]
    Hence $\overline{C}$ intersects $Z_3$. At a puncture of $C$ in $\overline{C}\cap Z_3$, the associated stable curve is obtained by pinching a nonseparating simple closed curve $\gamma$ on $S_g$. The local monodromy around this puncture is $T_\gamma^3$, which has infinite order.
\end{enumerate}

\noindent\textbf{Step 6: Finishing the proof.}
Let $C$ be the curve constructed in Step~5. By Step~4, each
$H_t|_C$ is algebraic, and hence $H_t|_C$ gives a path in
$\operatorname{Mor}(C,\mathcal{A}_g)$ from $F|_C$ to
$J^{(3)}\circ\nu|_C$ (resp.~$J^{(3)}|_C$). Since the latter is an
isolated point of $\operatorname{Mor}(C,\mathcal{A}_g)$ by Step~5,
this path is constant. Therefore
\[F|_C=J^{(3)}\circ\nu|_C\qquad\text{(resp.~$F|_C=J^{(3)}|_C$)}.\]
Thus $F$ and $J^{(3)}\circ\nu$ (resp.~$J^{(3)}$) agree at a point of $C$, so Step~2 implies that they agree everywhere.
\end{proof}

\subsection{An alternative step using variations of Hodge structures}\label{sec: PVHS}
We give an alternative argument replacing Steps 3--5 above, which does not use the rigidity theorem and instead uses polarized variations of Hodge structures (PVHS), following the method suggested by Richard Hain to Farb in \cite[Section 3]{FarbRigidity}.

Let $C\subset \mathcal{M}_g(3)$ be the smooth curve constructed in Step 5 above, with the key property that the inclusion induces a surjection
\[\pi_1(C)\to \pi_1(\mathcal{M}_g(3)).\]
The maps $F$ and $J^{(3)}$ lift to holomorphic maps 
\[\widehat{F},\widehat{J^{(3)}}:\mathcal{M}_g(3) \to \mathcal{A}_g(3).\]
Their restrictions to $C$ are algebraic by the same argument as in Step 4 above.

Let $\mathcal{V}_g$ be the standard polarized variation of Hodge structure of weight one on $\mathcal{A}_g(3)$, with
\[\mathcal{V}_g(A)=H^1(A;\ZZ),\qquad A\in \mathcal{A}_g(3).\] Pulling back $\mathcal{V}_g$ along $\widehat{F}|_C$ and $\widehat{J^{(3)}}|_C$ gives two polarized variation of Hodge structures
on $C$, which we denote by $\mathcal{V}_F$ and $\mathcal{V}_J$. Their underlying local systems have monodromy representations
\[\pi_1(C)\twoheadrightarrow \pi_1(\mathcal{M}_g(3))
\xrightarrow{\widehat{F}_*}
\pi_1(\mathcal A_g(3)),\]
and 
\[\pi_1(C)\twoheadrightarrow \pi_1(\mathcal{M}_g(3))
\xrightarrow{\widehat{J^{(3)}}_*}
\pi_1(\mathcal A_g(3)),\]
respectively. Both monodromy representations are irreducible, since their images are $\Sp_{2g}(\ZZ,3)$, which is Zariski dense in $\Sp_{2g}$. Hence both $\mathcal{V}_F$ and $\mathcal{V}_J$ are simple PVHS. Moreover, by Step 1, the two monodromy representations are conjugate by an element of $\Sp_{2g}(\ZZ)$, which induces an isomorphism between $\mathcal{V}_F$ and $\mathcal{V}_J$ as local systems. 

By the theorem of the fixed part, this isomorphism is fiberwise an isomorphism of Hodge structures. Therefore, by the rigidity theorem of Schmid \cite[Theorem 7.24]{Schmid}, it extends to an isomorphism of polarized variations of Hodge structure
\[\mathcal{V}_F\cong \mathcal{V}_J.\]
For $p\ge 3$, the space $\mathcal{A}_g(p)$ is a fine moduli space of principally polarized abelian varieties with level-$p$ structures. It follows that
\[\widehat{F}|_C=\widehat{J^{(p)}}|_C.\]
Consequently, for every $x\in \pi(C)$, we have $F(x)=J^{(p)}(x)$, which replaces Steps 3--5 above.

\appendix
\setcounter{table}{0}
\renewcommand{\thetable}{A\arabic{table}}

\bibliographystyle{alpha}
\bibliography{ref}

\end{document}